\documentclass[11pt, oneside]{amsart}
\usepackage{pb-diagram,amssymb,epic,eepic,verbatim,braket, xcolor}
\usepackage{wasysym,graphics,epsfig,psfrag,braket,marginnote}
\usepackage{hyperref,url}
\hypersetup{colorlinks}
\usepackage{color}
\usepackage[numbers]{natbib}

\usepackage[shortlabels]{enumitem}
\usepackage{soul}
\usepackage{bm}
\newcommand{\beqn}{\begin{equation*}}
\newcommand{\eeqn}{\end{equation*}}
\newcommand{\beq}{\begin{equation}}
\newcommand{\eeq}{\end{equation}}

\newcommand{\mhat}{\hat}

\usepackage{amsfonts,transparent}
\DeclareMathAlphabet{\mathpgoth}{OT1}{pgoth}{m}{n}
\DeclareMathAlphabet{\mathesstixfrak}{U}{esstixfrak}{m}{n}
\DeclareMathAlphabet{\mathboondoxfrak}{U}{BOONDOX-frak}{m}{n}
\usepackage{amsmath,amssymb}
\usepackage[bbgreekl]{mathbbol}
\usepackage{yfonts,mathtools}
\definecolor{darkred}{rgb}{0.5,0,0}
\definecolor{darkgreen}{rgb}{0,0.5,0}
\definecolor{darkblue}{rgb}{0,0,0.5}
\hypersetup{ colorlinks,
linkcolor=darkblue,
filecolor=darkgreen,
urlcolor=darkred,
citecolor=darkblue }
\usepackage{xr}
\newtheorem{theorem}{Theorem}[section]

\newtheorem{corollary}[theorem]{Corollary}

\newtheorem{proposition}[theorem]{Proposition}
\newtheorem{lemma}[theorem]{Lemma}
\newtheorem{lem}[theorem]{}
\theoremstyle{definition}
\newtheorem{definition}[theorem]{Definition}
\theoremstyle{remark}
\newtheorem{remark}[theorem]{Remark}
\newtheorem{example}[theorem]{Example}
\newcommand{\blem}{\begin{lem} \rm}
\newcommand{\elem}{\end{lem}}

\newcommand\M{\mathcal{M}}

\renewcommand\M{\mathcal{M}}

\newcommand\supp{\on{supp}}
\newcommand\XX{\mathbb{X}}

\newcommand{\J}{\mathcal{J}}

\newcommand{\R}{\mathbb{R}}
\renewcommand{\H}{\mathbb{H}}
\newcommand{\RR}{\mathcal{R}}
\newcommand{\C}{\mathbb{C}}
\newcommand{\CP}{\mathbb{C}P}

\newcommand{\cS}{\mathcal{S}}
\newcommand{\cM}{\mathcal{M}}

\newcommand{\cW}{\mathcal{W}}

\newcommand{\Z}{\mathbb{Z}}

\newcommand{\ddt}{\frac{\d}{\d t}}
\newcommand{\dds}{\frac{\d}{\d s}}

\renewcommand{\P}{\mathbb{P}}

\newcommand{\on}{\operatorname}

\newcommand\white{{\includegraphics[width=.05in]{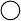}}}

\newcommand\black{{\includegraphics[width=.05in]{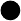}}}

\newcommand{\ainfty}{{$A_\infty$\ }}
\newcommand{\ab}{\ab}

\newcommand{\wt}{\on{wt}}

\newcommand{\red}{{\on{red}}}

\newcommand{\dual}{\vee}

\newcommand{\Edge}{\on{Edge}}

\newcommand{\Cliff}{\on{Cliff}}
\newcommand{\Hopf}{\on{Hopf}}

\newcommand{\Ver}{\on{Vert}}

\newcommand{\End}{\on{End}}

\newcommand{\Aut}{ \on{Aut} }

\newcommand{\Ad}{ \on{Ad} }

\newcommand{\Ind}{ \on{Ind}}
\renewcommand{\ker}{ \on{ker}}
\newcommand{\coker}{ \on{coker}}
\newcommand{\im}{ \on{im}}
\newcommand{\ind}{ \on{ind}}

\newcommand\dirac{/\kern-1.2ex\partial} 
\newcommand\qu{/\kern-.7ex/} 
\newcommand\lqu{\backslash \kern-.7ex \backslash} 

\newcommand\dr{r_+ \kern-.7ex - \kern-.7ex r_-}
 
\newcommand{\labell}\label

\renewcommand{\d}{{\on{d}}}
\newcommand{\ol}{\overline}
\newcommand{\olp}{\ol{\partial}}

\newcommand\eps{\epsilon}

\newcommand\om{\omega}

\newcommand{\lan}{\langle}
\newcommand{\ran}{\rangle}

\newcommand{\ti}{\tilde}

\newcommand\pt{\on{pt}}
\newcommand\cE{\mathcal{E}}

\newcommand\cR{\mathcal{R}}
\newcommand\cT{\mathcal{T}}

\newcommand\cI{\mathcal{I}}

\newcommand\cA{\mathcal{A}}

\newcommand\mE{\mathcal{E}}

\newcommand\Rep{\on{Rep}}
\newcommand\Map{\on{Map}}
\newcommand\rank{\on{rank}}

\newcommand\Vect{\on{Vect}}
\renewcommand\ul{\underline}
\newcommand\mO{\mathcal{O}}

\renewcommand\Im{\on{Im}}

\newcommand\grad{\on{grad}}

\newcommand\bdefn{\begin{definition}}
\newcommand\edefn{\end{definition}}
\newcommand\bea{\begin{eqnarray*}}
\newcommand\eea{\end{eqnarray*}}
\newcommand\bcv{\left[ \begin{array}{r} }
\newcommand\ecv{\end{array} \right] }

\newcommand\bma{\left[ \begin{array}{l} }
\newcommand\ema{\end{array} \right]}
\newcommand\ben{\begin{enumerate}}
\newcommand\een{\end{enumerate}}
\newcommand\bex{\begin{example}}
\newcommand\bsj{\left\{ \begin{array}{rrr} }
\newcommand\esj{\end{array} \right\}}

\newcommand\eex{\end{example}}
\newcommand\crit{{\on{crit}}}

\newcommand\sx{*\kern-.5ex_X}

\newcommand{\bGamma}{\mathbb{\Gamma}}

\newcommand{\cyl}{\on{cyl}}

\def\mathunderaccent#1{\let\theaccent#1\mathpalette\putaccentunder}
\def\putaccentunder#1#2{\oalign{$#1#2$\crcr\hidewidth \vbox
to.2ex{\hbox{$#1\theaccent{}$}\vss}\hidewidth}}

\renewcommand{\aa}{\mathfrak{a}}
\renewcommand{\ab}{\on{ab}}
\newcommand{\bb}{\mathfrak{b}}
\renewcommand{\aa}{\mathfrak{a}}
\newcommand{\cc}{\mathfrak{c}}
\renewcommand{\ss}{\mathfrak{s}}

\newcommand{\bv}{\mathbb{v}}

\newcommand{\bh}{\mathbb{h}}

\newcommand\cwo[1]{ { \color{darkred}  } }

\makeatletter
\newcommand\bigDiamond{\mathop{\mathpalette\bigDi@mond\relax}}
\newcommand\bigDi@mond[2]{%
  \vcenter{\hbox{\m@th
    \scalebox{\ifx#1\displaystyle 2\else1.2\fi}{$#1\Diamond$}%
  }}%
}
\newcommand\bigLozenge{\mathop{\mathpalette\bigL@zenge\relax}}
\newcommand\bigL@zenge[2]{%
  \vcenter{\hbox{\m@th
    \scalebox{\ifx#1\displaystyle 2\else1.2\fi}{$#1\blacklozenge$}%
  }}%
}
\makeatother

\definecolor{darkred}{rgb}{0.5,0,0}
\definecolor{darkpurple}{rgb}{0.5,0,.5}
\definecolor{darkpink}{rgb}{0,.5,0.5}
\definecolor{cyan}{rgb}{.25,0,0.75}
\definecolor{darkgreen}{rgb}{0,0.5,0}
\definecolor{darkblue}{rgb}{0,0,0.5}

\begin{document}

\title[Augmentation varieties and disk potentials ]{Augmentation varieties and disk potentials II}


\author{Kenneth Blakey}
\address{Department of Mathematics,182 Memorial Drive, Cambridge, MA 02139, U.S.A.} \email{kblakey@mit.edu}

\author{Soham Chanda}
\address{Mathematics-Hill Center, Rutgers University, 110
  Frelinghuysen Road, Piscataway, NJ 08854-8019, U.S.A.  } \email{sc1929@math.rutgers.edu}

\author{Yuhan Sun}
\address{Huxley Building,
South Kensington Campus,
Imperial College London,
London,
SW7 2AZ, U.K}\email{yuhan.sun@imperial.ac.uk}

  \author{Chris Woodward}
\address{Mathematics-Hill Center, Rutgers University, 110
  Frelinghuysen Road, Piscataway, NJ 08854-8019, U.S.A.}
\email{ctw@math.rutgers.edu}

\thanks{Chanda and Woodward were partially supported by NSF grant DMS 
  2105417 and Sun was partially supported by the
EPSRC grant EP/W015889/1.
 Any opinions, findings, and conclusions or recommendations 
  expressed in this material are those of the author(s) and do not 
  necessarily reflect the views of the National Science Foundation.  }

\begin{abstract} 
This is the second in a sequence of papers in which we construct Chekanov-Eliashberg algebras for Legendrians in circle-fibered contact manifolds and study the associated augmentation varieties. In this part, we first define the Chekanov-Eliashberg algebra and its Legendrian contact homology.  For a tame Lagrangian cobordism between Legendrians,  we define a chain map between their Chekanov-Eliashberg algebras.  
\end{abstract}

\maketitle

\tableofcontents

\section{Introduction} 

Contact homology as introduced by Eliashberg, Givental, 
and Hofer \cite{egh} is a theory whose differential counts holomorphic curves in the symplectization of a contact manifold.  Its relative version, Legendrian contact homology, has been developed  by Ekholm-Etnyre-Ng-Sullivan \cite{ees:lch, ees:orient, ees:leg, eens:knot}
and related works in situations where the contact manifold fibers over, for example, an exact symplectic manifold.  This paper develops a version of Legendrian contact homology in case that the contact manifold fibers over a monotone symplectic manifold, extending work of Sabloff \cite{sabloff}, Asplund \cite{asplund}, and Dimitroglou-Rizell-Golovko \cite{dr:bs}. We consider a Legendrian $\Lambda$ in a compact contact manifold $Z$ that
is a negatively-curved circle-fibration over a symplectic base $Y$;  we have in mind especially the example that $Z$ is the unit canonical bundle over a Fano projective toric variety $Y$, and $\Lambda$ is the horizontal lift of a Lagrangian in the base.  Another
typical example is the horizontal lift of the Clifford torus in complex projective space, which is a Legendrian torus in the standard contact sphere. We prove the following: 

\begin{theorem}  \label{thm:he} Suppose that $\Lambda $ is a tame Legendrian in a contact manifold
$ (Z, \alpha)$.  Counting punctured holomorphic disks $u: C \to \R \times Z$ bounding $\R \times \Lambda$
defines Legendrian contact homology groups
\[ HE(\Lambda) = \frac{ \on{ker}(\delta)}{ \on{im}(\delta)}, \quad \delta:
CE(\Lambda) \to CE(\Lambda) \]
independent of the choice of almost complex structures and perturbations, and is an invariant of the 
Legendrian isotopy class of $\Lambda$.
\end{theorem} 

The construction is a special case of symplectic field theory with 
Lagrangian boundary conditions, where the contact manifold is circle-fibered. 
The simplifications in this case are comparable to that of genus-zero relative Gromov-Witten theory compared to full symplectic field theory with contact boundary.  In the first paper \cite{BCSW1} in this series, we used Cieliebak-Mohnke \cite{cm:trans} perturbations to regularize the moduli spaces
of treed punctured disks, and proved compactness and orientability of the moduli spaces.   The natural chain complex $CE(\Lambda)$ corresponding
to limits of such disks 
is generated by words in chains on the Reeb chords and chains on the Legendrian, both are realized by critical points of Morse functions; the addition
of these {\em classical generators} does not occur in the version of Legendrian contact
homology over exact symplectic manifolds in Ekholm-Etnyre-Sullivan \cite{ees:leg}.
These classical generators arise because of nodes developing 
along points in the interior of the cobordism, which then lead to trajectories as in 
Ekholm-Ng \cite{ekholmng:higher}; in particular \cite{ekholmng:higher} assigns
augmentations to the Harvey-Lawson filling in the three-dimensional setting of knot contact homology while the construction  here is intended to work in all dimensions, under some restrictions. 

The maps between homology groups associated to different choices of almost complex 
structures are special cases of cobordism maps associated to Lagrangian cobordisms
between Legendrians.   Recall from \cite{BCSW1} that by a {\em  symplectic cobordism} with convex end 
a circle-fibered contact manifold $(Z_+,\alpha_+)$
and concave end a circle-fibered contact manifold $(Z_-,\alpha_-)$
we mean a symplectic manifold with boundary $(\widetilde{X},\partial \widetilde{X}, \omega)$ whose boundary is equipped with a partition
\[ \partial \widetilde{X}= \partial_{+} \widetilde{X}\cup \partial_{-} \widetilde{X} \] 
and diffeomorphisms 
\[\iota_\pm : Z_\pm \to \partial_\pm \widetilde{X}  \]
so that
\begin{enumerate} 
\item the two forms satisfy the proportionality relation
\[ \iota_\pm^* \omega\mid_{\partial_{\pm} \tilde{X}}=  \lambda_\pm  \omega_\pm, \quad \omega_\pm := \d\alpha_\pm
\]
for some positive constants $\lambda_\pm \in \R$, and 
\item the 
isomorphism  between the normal bundle of $Z_\pm$
in $\widetilde{X}$ and the kernel of $D p: TZ_\pm \to TY_\pm, Y_\pm := Z_\pm/S^1$
induced by the symplectic form is orientation preserving resp. reversing for the convex resp. concave boundary. 
\end{enumerate}
Denote by $X$ the interior of $\widetilde{X}$, viewed as a manifold with cylindrical ends, and, abusing terminology, call $X$ the cobordism.  
Cobordisms can be composed in the standard way using the coisotropic embedding theorem.  However, our cobordisms do not form the  morphisms of any category as the trivial cobordism does not admit a symplectic structure satisfying the conditions we require. 

Our cobordisms will often arise by removing symplectic hypersurfaces from compact symplectic manifolds as follows.    Let $\ol{X}$ be a compact symplectic manifold and
\[ Y_-,Y_+ \subset \ol{X} \] 
be codimension two symplectic submanifolds
of $\ol{X}$.  Let $\widetilde{X}$ be the real blow-up 
of $\ol{X}$ along $Y_-,Y_+$, obtained by replacing 
$Y_\pm$ with their unit normal bundles.  
\[ Z_\pm^{\pm 1} \subset  \widetilde{X} \]
where $Z_-^{-1}$ denotes the principal circle bundle 
obtained by reversing the direction of the action on $Z_-$.
We assume that the restriction of the symplectic form on $\ol{X}$
restricts to two-forms on $Y_\pm$ that are  positive multiples of the curvature of $Z_\pm$.
We call the interior $X$ of $\widetilde{X}$ a symplectic cobordism with concave end $Z_-$ and convex end $Z_+$.   This definition of cobordism is essentially
equivalent to the definition in Etnyre-Honda \cite{etnyrehonda}; there are many 
other definitions of symplectic cobordism in the literature.    Conversely, 
given any such $X$ one may obtain a compact symplectic manifold without boundary $\ol{X}$ by Lerman's symplectic cut 
construction \cite{le:sy}.  

Given a symplectic cobordism, a Lagrangian cobordism is a Lagrangian submanifold 
that has the desired Legendrian boundary.  That is, a Lagrangian cobordism  between Legendrians $\Lambda_-, \Lambda_+$ is a Lagrangian
$L \subset X$  whose closure $\ol{L} \subset \widetilde{X}$ has intersection with the boundary $\partial \widetilde{X}$ given by 
\[ L \cap Z_\pm = \Lambda_\pm .\] 
A punctured disk or sphere to $X$ of finite energy extends to an
unpunctured disk or sphere mapping to the compactification $\ol{X}$, 
by removal of singularities. We assume that  the Lagrangian satisfies {\em tameness} conditions detailed in \cite[Definition \ref{I-def:tamepair}]{BCSW1},
which imply that rigid holomorphic spheres and disks with only outgoing punctures do not appear in the boundary of our moduli spaces, and there is  no breaking of holomorphic disks  along closed Reeb orbits. However, the Lagrangian is not necessarily exact. This enables us to discuss the relation between fillings and augmentations in a general setting.

The construction of cobordism maps depends on the existence of bounding chain similar to the definition in  Fukaya-Oh-Ohta-Ono \cite{fooo:part1}.  By definition such a chain is an element 
\[ b \in CE(L,\partial L) \] 
(where $CE(L,\partial L)$ denotes  relative chains with coefficients in some completed group ring, see Equation \eqref{cel}) that solves a {\em Maurer-Cartan equation}:   Let
\[ m_d: CE(L,\partial L)^{\otimes d} \to CE(L, \partial L) \]
denote the counts (defined precisely in Definition \ref{defmc}) of treed holomorphic disks with leaves at critical points of the Morse function in the interior as in  \eqref{cel}.    The Maurer-Cartan equation
\[ \sum_{d \ge 0} m_d(b,\ldots, b)  = 0 \in CE(L,\partial L) .\]
has a space of solutions denoted $MC(L)$.   (The maps $m_d$ do not form an \ainfty algebra, because of the additional terms arising from the cylindrical
ends of the Legendrian.)

\begin{theorem} \label{chainmap} Let $(X,L)$ be a tame cobordism 
with concave end $(Z_-,\Lambda_-)$ and convex end $(Z_+,\Lambda_+)$ as in Theorem \ref{thm:he} equipped with a bounding
chain $b \in MC(L)$. Counting punctured holomorphic disks defines  a chain map
\[ \varphi{(L,b)} : CE(\Lambda_-,\hat{G}(L)) \to CE(\Lambda_+,\mhat{G}(L)), \quad \varphi \delta_-
= \delta_+ \varphi \]
between the complexes associated to the ends over a coefficient ring
$\hat{G}(L)$ given by a completed group ring on $H_2(\ol{X},\ol{L})$, and independent of the choice of almost complex 
structure, Morse functions, and perturbations for disks bounding $L$ up to chain homotopy and a change
in the bounding chain $b \in MC(L)$.
\end{theorem} 

A special case of the theorem, for trivial cobordisms, may be used
to prove invariance of contact homology under Legendrian isotopy in Theorem \ref{thm:he}.  As a second case of Theorem \ref{chainmap},
 fillings  of a Legendrian with good properties (such as the Harvey-Lawson filling below) define augmentations:   chain maps 
\[ \varphi{(L,b)} : \ CE(\Lambda) \to \hat{G}(L) .\]
In the third part \cite{BCSW3} of this series we study the augmentations associated to these fillings, and apply the resulting augmentation varieties to distinguish 
Legendrian isotopy classes.

\section{Classification results}

In this section,  we carry out a classification of holomorphic disks that allows us to partially compute certain differentials in the Legendrian contact homology, which will be introduced in the next section.  For example, we wish to compute the image of the differential on degree one generators in the case of Legendrian lifts of Lagrangian torus orbits in toric varieties. 

First let us recall the geometric setup from \cite[Section 2]{BCSW1}. Let $(Y,\omega_Y)$ be a closed symplectic manifold such that $[\omega_Y] \in H^2(Y,\Z)$ is an integral cohomology class. There exists a principle $S^1$-bundle $p:Z\to Y$ whose first Chern class is $- [\omega_Y] \in  H^2(Y,\Z)$ together with a connection one-form form $\kappa \in \Omega^1(Z)$ such that:
    \begin{itemize}
        \item $\alpha = - \kappa$ is a contact form on $Z$;
        \item the curvature form of $\kappa$ is $
        - \omega_Y$, that is, $\d\kappa = 
        - p^*\omega_Y$;
        \item and the generating vector field defining the principal $S^1$-action on $Z$ coincides with opposite of the Reeb vector field $R_\alpha$ of $\alpha$.
    \end{itemize}
We call such a contact manifold $(Z,\alpha)$ a \textit{circle-fibered contact manifold}. For a Legendrian submanifold $\Lambda$ in $Z$, we write $\Pi:=p(\Lambda)\subset Y$, which is a (possibly immersed) Lagrangian in $Y$.

\begin{example} \label{cliffleg4}
The Clifford Lagrangian in complex projective space is denoted
\[ \Pi_{\Cliff} := \Set{ [ z_1,\ldots, z_n ] \ | \ \sum_{i=1}^n |z_i|^2 = 1 } \cong T^{n-1}
\subset \CP^{n-1} .\]
We view the standard contact sphere as a circle bundle over the complex projective space. The Clifford Legendrian $\Lambda_{\Cliff}$ is a horizontal lift of the Clifford Lagrangian defined as 
\begin{equation}   
\Lambda_{\Cliff} := \Set{ (z_1,\ldots, z_n) \in \C^n |  \begin{array}{l}
                                                     |z_1|^2 = \ldots
                                                     = |z_n|^2 = \frac{1}{n} \\  z_1 
    \ldots z_n \in (0,\infty)  \end{array} } \cong T^{n-1} \subset S^{2n-1} .
\end{equation}
The corresponding Lagrangian cylinder in the symplectization is 
  \[ \R \times \Lambda_{\Cliff} \cong \Set{ (z_1,z_2, \ldots, z_n) | 
  \begin{array}{l} |z_1| = |z_2| = \ldots =
    |z_n| \\  z_1 z_2 \ldots z_n \in (0,\infty) \end{array} }  \subset \C^n - \{ 0 \} . \]
\end{example}

\begin{example}\label{ex:HLfilling}
An example of a Lagrangian filling of $\Lambda_{\Cliff}$ is given by the Harvey-Lawson type filling.  Let
\[ a_1,\ldots, a_n \in \R \]
be 
real numbers with exactly two  zero.  Define as in Joyce
\cite[(37)]{joyce}
\begin{equation} \label{hl}
L_{(1)} := \Set{ (z_1,\ldots, z_n) \in \C^n | \begin{array}{l} |z_1|^2 - a_1^2
    = |z_2|^2 - a_2^2 = \ldots =
                                       |z_n|^2 -a_n^2 \\
                                       z_1 \ldots z_n \in
                                       [0,\infty) \end{array} } .\end{equation}
The parameters $a_i$ are often omitted in the notation.
\end{example}

A special feature of the contact form $\alpha$ is that its Reeb orbits are multiple covers of the circle fibers. Our convention identifies a circle fiber with $\R/\Z$. In other words, the action of a simple Reeb orbit is one.

\subsection{Examples of punctured holomorphic disks} 

In Section \ref{liftsec}, we describe more generally how holomorphic disks in the base lift to punctured disks in the symplectization.  In this section, we discuss particular motivational examples. 

\begin{example}    The following describes lifts of Maslov-index-two-disks bounding the Clifford Lagrangian.
For $  \Im(z) \ge 0, c >0 ,a > 0 $ and phases
  $e^{i \theta_1}, \ldots, e^{i \theta_n}$ with unit product
\[ e^{i \theta_1} \cdot \ldots \cdot e^{i \theta_k} = 1 . \]
let 
  \begin{multline} \label{filldisks} u : \mathbb{H} \to \C^n, \\
  z \mapsto  ( e^{i \theta_1}  ( cz - ai) (cz +
    ai)^{(2/n)-1}, e^{i \theta_2} (cz+ ai)^{2/n}, \ldots, e^{i \theta_n}  (cz + ai)^{2/n} ) \end{multline}
The composition 
  $\ol{u}(z)$ of $u(z)$ with the projection to $\C P^{n-1}$ is
\[ \ol{u}(z) = \left[ \frac{cz - ai}{cz + ai}, 1,\ldots, 1\right] .\] 
The disk $\ol{u}$ has with one intersection with the anticanonical
divisor, at $cz = ai$, and so has Maslov index $2$. We will see in Section \ref{liftsec} that this  is a complete 
list of punctured disks with one incoming Reeb chord of minimal length.   One obtains examples of treed disks by adding Morse trajectories 
limiting to index one critical points.
\end{example}

\begin{example} \label{hopf2}
We next describe the disks that bound a particular connected filling of a disconnected Legendrian.
Let $\Lambda_{\Cliff} \subset S^{2n-1}$ be the Clifford Legendrian in Example \ref{cliffleg4}.  Define the {\em Hopf Legendrian} 
\begin{eqnarray} 
\label{hopf} \Lambda_{\on{Hopf}} &=& \Lambda_{\Cliff} \cup \exp\left( \frac{\pi i}{n} \right)
 \Lambda_{\Cliff}
  \\
  &=& \Set{ (z_1,\ldots, z_n) \in \C^n |  \begin{array}{l}
                                                     |z_1|^2 = \ldots
                                                     = |z_n|^2 = \frac{1}{n} \\  z_1
    \ldots z_n \in \R \end{array} }    \\
    &\cong& T^{n-1} \sqcup T^{n-1} \subset  S^{2n-1}
    \end{eqnarray}
 the union of $\Lambda_{\Cliff}$ with its image $ \exp\left( \frac{\pi i}{n} \right)
 \Lambda_{\Cliff} $ under multiplication by
 $\exp\left( \frac{\pi i}{n} \right)$; the terminology is based on the 
 analogy with the Hopf link in the three-sphere.  An immersed filling of $\Lambda_{\on{Hopf}}$ is given by the union of the Harvey-Lawson fillings:
\[ L_{(1,1)} := L_{(1)} \cup \exp \left( \frac{\pi i}{n} \right) L_{(1)} \cong 
((S^1)^{n-2} \times \R^2) \sqcup 
((S^1)^{n-2} \times \R^2) 
\subset \C^n.\] 
The
intersection of the two components can be made transverse after perturbation.  On the other hand,
a connected filling is given by the set
\begin{equation} \label{L2}
L_{(2)} := \{ (z_1,\ldots, z_n) |  z_1 \ldots z_n \in \R + i \eps, \  |z_1|^2 = \ldots = |z_n|^2
. \} \end{equation} 

The exact filling described in \eqref{L2} can be described in terms of symplectic
parallel transport as follows.  Let
\[ \pi: \C^n \to \C, \quad (z_1,\ldots, z_n) \mapsto z_1 \ldots z_n \]
be the map given by the product of components.  Define the {\em symplectic
  connection} on the fibration $\pi: \C^n - \{ 0 \} \to \C$ as the
symplectic perpendicular to the vertical subspace
\[ T^v \C^n - \{ 0 \} = \ker(D \pi) \] 
of the differential $D\pi$, with respect to the standard symplectic form.  Consider
the path 
\[ \vartheta: \R \to \C, \quad t \mapsto   t + i\eps \] 
or more generally  any path avoiding the critical
value $0 \in \C$.  Symplectic parallel transport 
$\rho_{t}^{t'}$ of the
Lagrangian torus $T^{n-1}$ in any fiber
\[ T^{n-1} \cong \Set{ |z_1| = \ldots = |z_n| \ | \  \pi(z_1,\ldots, z_n) = t }, \quad  t \gg 
0 \]
along the path $\vartheta$ defines the Lagrangian filling
\[ L_{(2)}   := \bigcup_{t' \in \R } \gamma_{t}^{t'} T^{n-1}  \cong T^{n-1} \times \R \] 
of $\Lambda$. 

The holomorphic disks bounding this filling now have the possibility of limiting to Reeb chords connecting the two sheets of the Legendrian at infinity. 
For each $k = 1,\ldots, n$ define a 
holomorphic map 
\begin{multline} \label{nor1} 
u^k_{01}: \H \to \C^n, \\
z \mapsto  \left((-z-i\eps)^{\frac{1}{n}} , \dots,
\underbrace{\frac{-z+i\eps}{-z-i\eps} (-z-i\eps)^{\frac{1}{n}}}_{k^{th}\text{
    coordinate}} , \dots , (-z-i\eps)^{\frac{1}{n}} \right) .\end{multline}
This map is asymptotic along the unique strip-like end to a Reeb chord $\gamma_{12}$ of action $1/n.$
Second, consider the holomorphic map 
\begin{equation} \label{nor12} u_{10} = ((z+i\eps)^{\frac{1}{n}} , \dots ,
(z+i\eps)^{\frac{1}{n}} ) .\end{equation}
This map is asymptotic to a Reeb chord $\gamma_{21}$ of action $1/n$ connecting the sheets in the reverse order.
\end{example}

\subsection{Lifting disks from the base} 
\label{liftsec}

In this section, we prove a bijection between 
disks bounding the Lagrangian in the symplectization and disks bounding the Lagrangian projection of the Legendrian.  This bijection is a version of a result in Dimitriglou-Rizell \cite[Theorem 2.1]{riz:lift}
adapted to our setting.  Later this lifting result will be used to prove that the augmentation variety is contained in the zero level set of the Landau-Ginzburg potential.   To state the result, recall the following type of almost complex structures on a symplectic cobordism.

\begin{definition}
Let  $X = \R \times Z$ be the trivial cobordism. An almost complex structure $ J : TX \to TX$ is {\em  cylindrical} if there exists an almost complex structure $\ol{J} : TY \to TY $ so that the projection $p_X: X \to Y$ is almost complex and $J$ is the standard almost complex structure on any fiber.  More precisely, let 
\[ \partial_s \in \Vect(\R \times Z), \quad \partial_{\theta} \in \Vect(\R \times Z) \]
denote the translational vector field on $\R$ resp. rotational vector field on $Z$. 
The almost complex structure $J$ is determined
on the vertical part of $TX$ by 
\[ \quad J  \partial_s = \partial_\theta, \quad J \partial_\theta = - \partial_s  \  \text{in}  \ \Vect(\R \times Z) .\] 
On the other hand, the projection to $Y$ is required to be almost complex:
\[   D {p}_X J = \ol{J} D {p}_X \  \text{in}  \ \Map(TX,TY) . \]
Suppose now $X$ is an arbitrary cobordism with concave end $Z_-$ and convex end $Z_+$.
An almost complex structure $J$ on $X$ is called {\em  cylindrical}
if it is the restriction of cylindrical almost complex structures
on $\R \times Z_\pm$ on the cylindrical ends $\pm (0,\infty) \times Z_\pm \to X$.  The space of cylindrical almost complex structures is denoted $\J_{\cyl}(X)$.
This ends the Definition.
\end{definition} 

Let $\M(\Lambda)$ denote the moduli space of holomorphic punctured disks of finite Hofer energy bounding $\R\times\Lambda$ as in \cite{BCSW1}
and $\M(\Pi)$ the moduli space of (unpunctured) holomorphic disks bounding $\Pi$.
Here we mean disks with irreducible domain, not treed disks, although an extension to treed disks will be used later.

\begin{lemma} Let $Z$ be a circle-fibered contact manifold equipped with a cylindrical almost complex structure.  Composition with the projection $p: \R\times Z \to Y$ defines a map 
of moduli spaces
\begin{equation}\label{eq:projonmoduli}
 \mathfrak{P}:\M(\Lambda) \to \M(\Pi), \quad u \mapsto p \circ u.
\end{equation}
\end{lemma}

\begin{proof}  Given any punctured disk $u$
mapping to $\R \times Z$, composition with the projection from $\R \times Z$ to $Y$ defines a punctured disk mapping to $Y$, since the projection is almost complex.  The fact that $u$ is finite Hofer energy implies that 
the projection is finite energy, and so extends over the punctures
by removal of singularities.  
\end{proof}

We consider the problem of lifting a particular map to the base. 

\begin{definition}  Consider a map bounding $\Pi$
\[ u_Y : S \to Y, \quad u_Y ( \partial S) \subset \Pi  .\]
A lift $u$ of $u_Y$  to a map to $\ol{X}$ is equivalent to a 
section
\[ \sigma: S \to u_Y^* ( Z \times_{S^1} \P^1
  \to Y) \]
of the pull-back  $u_Y^* ( Z \times_{S^1} \P^1
\to Y) $ of the $\P^1$-bundle $Z \times_{S^1} \P^1
\to Y$ associated to $Z$ that we call the {\em section  associated
to the lift}. 
\end{definition}

Since the domains of our maps are disks, these sections are given by their zero-and-pole-structure as we now explain.

\begin{definition} The {\em angle map} assigns to any 
punctured disk $u: S \to X$ the collection of angles of Reeb chords at the punctures
\begin{equation} \label{utotheta} 
\M_{\bGamma}(\Lambda) \to \R_{> 0 }^{\# \Edge_{\rightarrow,\white}(\bGamma)}, \quad
 u \mapsto (\theta_e(u))_{e \in \mE(S)}.\end{equation} 
\end{definition} 

\noindent Introduce the following notation for the set of  possible angle changes. Let 
\[ \cA(u_Y )= \{ (\theta_e(u) )_{e \in \mE(S)} | p \circ u = u_Y \}  \subset  \R_{> 0 }^{\# \Edge_{\rightarrow,\white}(\bGamma)} \] 
be the set of tuples such which satisfy the following condition in \cite[Lemma \ref{I-anglechange}]{BCSW1}, which we recall now. The set of punctures $\mE(S)$ is divided into two disjoint subsets, the outgoing punctures $\mE_-(S)$ and the incoming punctures $\mE_+(S)$. They correspond to punctures asymptotic to the concave end and convex end respectively. The angle change condition is
\begin{equation}\label{eq:anglechange}
\sum_{e \in \cE_+ (S)} \theta_e - \sum_{e \in \cE_-(S)} \theta_e = \int_{S} u^* \d \alpha =  \int_S u_Y^* \omega_Y .
\end{equation}

\vskip .1in
The angle map is a bijection on the space of lifts, similar to the exact case by an observation of Dimitriglou-Rizell \cite[Theorem 2.1]{riz:lift}.

\begin{theorem} \label{liftthm}  Suppose the projection $\Pi$ of $\Lambda$ is embedded in $Y$ and all Reeb chords from $\Lambda$ to $\Lambda$ have angles that are multiples 
of $1/k$ for some $k \in \Z_{> 0}$. 
The map \eqref{utotheta} defines a bijection between 
equivalence classes of lifts $u:S \to X$ of a map $u_Y:\ol{S} \to  Y$ 
and length collections $\theta_e, e \in \mE(S)$  satisfying the 
angle change condition where two lifts are considered equivalent if they are equal up to a scalar multiplication
preserving $\Lambda$.
\end{theorem} 

\begin{proof}  
We claim that the angle map \eqref{utotheta} is bijective on the space of lifts of a given map, up to multiplication by scalars. In particular, we will show that for any $u_Y \in \M_\bGamma(\Pi)$, the restriction of the angle map (\ref{utotheta}) to $\mathfrak P ^{-1} (u_Y)$ is bijective onto $\cA(u_Y)$. 

\vskip .1in
\noindent {\em Step 1:  The angle map is injective. }   Consider two lifts
$u,u'$ of $u_Y$.    Since $\C^\times$ acts
transitively on the fibers of $p$, there exists a function $f$ so that 
the lifts $u,u'$ are related by multiplication by $f$, using the $\C^\times$
action on $\R \times Z$:
\[  f: S \to \C, \quad u' = f u . \]
Since both $u$ and $u'$ are holomorphic, the function $f$ is holomorphic as well. 
 By exponential convergence of punctured holomorphic disks to Reeb chords (see for example \cite[Theorem 3.1]{chanda}) the maps $u,u'$ are  asymptotic to maps of the form 
 \[ (s,t) \mapsto ( \theta_e s, t^{\theta_e}) (c,z) \ \text{resp. }  \ (s,t) \mapsto ( \theta_e' s, t^{\theta_e'})  \] 
 for some $(c,z) \in Z \times \R$ and constants $\theta_e,\theta_e'$.  So $f$ is asymptotic to  
\[ f(s,t) \sim ((\theta_e  - \theta_e') s, t^{\theta_e - \theta_e'}) \] 
on each end $e \in \mE(S)$. If the angles $\theta_e,\theta_e'$ are equal then $f$ is 
asymptotically constant along the ends, so bounded on $S$.  
Any such bounded holomorphic function is constant, which corresponds to the scalar multiplication in the fiber direction. 

\vskip .1in
\noindent {\em Step 2:  The angle map is surjective in the case of no outgoing punctures and the Lagrangian in the base is embedded.} The bundle $u_Y^* (\R \times Z \to Y)$ has totally real boundary condition that is a union of $k$ real subspaces $u_Y^* (\R \times  \Lambda)$.
Consider the $-k$-fold tensor product 
\[ u_Y^* (\R \times Z^{\otimes -k} \to Y) .\]
A boundary condition is given by the real sub-bundle 
\[ (\partial u_Y^*) (\R \times  \Lambda^{\otimes -k}) \subset 
u_Y^* (\R \times Z^{\otimes -k} \to Y) \] 
where
$\Lambda^{\otimes -k}$ denotes the image of $\Lambda$ under $Z \mapsto Z^{\otimes -k}$. (The branches combine under tensor product.)     
The map $u_Y$ extends over the punctures, by removal of singularities, and so the bundle pair above extends over the punctures as well.
By results of Oh \cite{oh:rh}, the bundle pair above is isomorphic to the pair $(\mO(d), \mO_\R(d))$.   Negativity of the bundle $Z \to Y$ implies that $d$ is non-negative.  After identifying the domain punctured
disk with $\C$ (with puncture at infinity) 
the  global sections with the given boundary conditions are real polynomials 
\[  \{ c (z - z_1 ) \ldots (z - z_d) \ | \  c, z_1,\ldots, z_d \in \R \}  \subset \Gamma( \mO(d), \mO_\R(d) )\] 
whose zeroes all occur on the boundary,  and the location of these zeroes may be arbitrarily specified.     

We obtain sections of the original boundary value problem by lifting.  Since $S$ is simply connected
and the map 
\[  u_Y^* (\R \times Z \to Y) \to  u_Y^* (\R \times Z^{\otimes k} \to Y) \]
is a $k$-fold cover,  the section $\ti{s}$ lifts to a section 
\[ s: S \to  u_Y^* (\R \times Z \to Y) \] 
with boundary values in $ \partial u_Y^* (\R \times  \Lambda^{\otimes k})$ as claimed with poles
at $z_1,\ldots, z_d$.
This shows that the map 
\eqref{utotheta} is surjective in the case that $S$ has no outgoing punctures. 

\vskip .1in
\noindent {\em Step 3:  The angle map is surjective in the case of both incoming and outgoing punctures.}  Suppose $\theta_e' \in 2 \pi \Z/k$ is a collection of angles at the same punctures $z_i$ satisfying the condition in \cite[Lemma \ref{I-anglechange}]{BCSW1}.  
Consider the function
\begin{equation} \label{rationaleq} f: S \to \C,  \quad z \mapsto \prod_{e \in \mE(S)}  (z - z_i)^{\pm (\theta_e' - \theta_e)} \end{equation}
(with sign depending on whether the puncture is incoming or outgoing). Because the angle sums are equal, the function
$f$ is bounded at infinity. The section $fs$ has angle changes $\theta_e'$ and real boundary
conditions as desired. 
\end{proof}

To prove the extension property used in the proof, we first note
that real Cauchy-Riemann operators on a given rank one bundle $E$ on the disk $S$ with real boundary condition $F \subset E| \partial S$ are all gauge equivalent.  That is,

\begin{lemma}  \label{lem:rankone} Let $S$ be a disk, 
$E \to S$ a complex line bundle, and $F \subset E |_{\partial S}$
a totally real boundary condition.    Let $D'_{E,F} $ and $D''_{E,F}$ be real Cauchy-Riemann
  operators on $E$ with boundary values in $F$.  Then there exists a
  real gauge transformation $g: S \to \Aut_\R(E)$ equal to the
  identity on the boundary so that
  \begin{equation} \label{realgauge}  g D'_{E,F} g^{-1} = D''_{E,F} . \end{equation}
\end{lemma}

\begin{proof} We remark that 
in the integrable, higher rank case, Cauchy-Riemann operators are classified by their splitting type as in Oh \cite{oh:rh}.  In the
rank one case the argument is straight-forward:
It suffices to show that the orbit through every real
  Cauchy-Riemann operator is open.  By definition $D'_{E,F}, D''_{E,F}$ are of the
  form
  \[ \xi \mapsto  \olp \xi  + A' \xi , \quad \xi \mapsto \olp \xi +
    A'' \xi  \]
  for some real-matrix-valued  $0,1$-forms $A',A''$.  A real gauge transformation $g:E \to E$ 
  acts on the space of such operators by
  \[ A \mapsto g \olp g^{-1} + \Ad(g) A . \]
  The tangent space to the orbit
  is the image of $\Omega^0(S, \End_\R(E))$ under the map
  \[  \xi \mapsto   \olp_{-A} := \olp \xi + [\xi,A]  = \olp \xi - [A,\xi] . \]
  The cokernel of this operator is the kernel of the adjoint
  $\olp_{-A}^*$, and equal to the kernel of the generalized Laplacian
  $\olp_{-A} \olp_{-A}^*$.   On the other hand, the domain of this
  Laplacian consists of functions that vanish on the boundary. 
  The kernel of $\olp_{-A}^*$ is trivial by the unique continuation
  principle as in \cite[Section 2.3]{ms:jh}.
\end{proof}

Rigidity and regularity of a map to the base is equivalent to rigidity  and regularity of any lift to the symplectization:

\begin{proposition} \label{prop:liftprop}  Let $C$ be a treed disk, 
$Z$ a circle-fibered contact manifold, $J \in \J(\R \times Z)$ a (possibly domain-dependent) cylindrical almost complex structure,  and $\Lambda \subset Z$
a compact Legendrian.
\begin{enumerate}
    \item A holomorphic treed disk $u: C \to \R \times Z$ bounding $\R \times \Lambda$ is regular if and only if $p \circ u: C \to Y$ is regular.
\item A holomorphic treed disk  $u: C \to \R \times Z$ bounding $\R \times \Lambda$ is rigid if and only if 
$p \circ u: C \to Y$ is rigid. 
\end{enumerate}
\end{proposition} 

We first introduce a long exact sequence that will be used in the proof. 
Let $\ti{D}_u, \ti{D}_{u_Y}$
be the linearized Cauchy-Riemann operators of $u$ and $u_Y$ from \cite[Equation \eqref{I-linop}]{BCSW1}. The tangent space of $T{X}$ splits into the kernel of $Dp$, denoted $T^{\bv} {X}$ and a bundle isomorphic to $p^* TY$.

\begin{lemma}  \label{lem:les}
Let $u: C \to X : \R \times Z$ be a treed disk and $u_Y = p \circ u$ its projection to the base.  The short exact sequence of bundles
\[ 0 \to T^{\bv} {X} \to T {X} \to p^* TY \to 0 \]
induces a long exact sequence of kernels and cokernels:
\[ 0 \to \ker(D_u^{\bv}) \to \ker(\ti{D}_u) \to
\ker( \ti{D}_{u_Y} ) \to \coker( {D}_u^{\bv})  \to \ldots  \]
\end{lemma} 

\begin{proof} The short exact sequence induces a short exact sequence of 
two-term complexes.  The first term of this complex is the short exact sequence of one forms with Lagrangian boundary values
\begin{multline}
    0 \to \Omega^0(C, u^* T^{\bv} X, (\partial u)^* T^{\bv} L ) \to T_C \M_\Gamma \oplus \Omega^0(C, u^* T {X},(\partial u)^* T L) \\
\to T_C \M_\Gamma \oplus \Omega^0(C, u_Y^* Y, (\partial u_Y^*) T\Pi ) \to 0 .
\end{multline} 
Here $\Omega^0(C,\ldots)$ includes both the surface $S$ and tree $T$ components,  with matching condition at the intersection $S \cap T$.  The second term in the short exact sequence is the corresponding exact sequence of one-forms.  The differentials of the complex are the vertical linearized operator $D_u^{\bv}$, the parametrized linear operator $ \ti{D}_u$ and the parametrized linear operator for the projection $\ti{D}_{u_Y}$ respectively.  The domains of these operators on spaces with  finite Banach norms of \cite[Section \ref{I-fredtheory}]{BCSW1} which we omit to simplify notation.  The claim follows from standard homological algebra. 
\end{proof}

\begin{proof}[Proof of Proposition \ref{prop:liftprop}]
For the first claim, by Lemma \ref{lem:les}, it suffices to show that the horizontal and vertical parts of the cokernel vanish.  The kernel of the vertical part $\ker(D_u^{\bv})$  has a section given by 
 the infinitesimal action of $\R^\times$:
 \[ \xi(z) = \ddt (  t u(z)) |_{t = 0 } \in u^* T^{\bv} Z .\] 
   As a rank one boundary value problem, either the kernel or the cokernel vanishes as explained in Oh \cite{oh:rh}.   Since the kernel is non-vanishing, $\coker(\ti{D}_u^{\bv})$ vanishes.  The argument for the second 
 claim is similar.  
 
 For the second claim about rigidity, we show that the kernel of the linearized operator 
 has the same dimension as the horizontal part up to a factor generated by translation.    
 The pull-back bundle $(u|S)^* T(\R \times Z)^{\bv}$ is isomorphic to the trivial bundle via the identification 
  \[ (u|S)^* T(\R \times Z)^{\bv} \mapsto S \times \C, \quad  \xi \mapsto D\exp_u \xi .\]  
The map $D \exp_u$ is  multiplication by $u$ after identifying any particular fiber with $\C^\times$.  The kernel $\ker(D_{u|S}^{\bv})$ is the space of  bounded holomorphic functions.  The map $D \exp_u$ identifies the kernel  with complex valued functions with the same order at the poles and zeros of the map $u|S$.

 We claim that the vertical part of the kernel is one-dimensional. Since $S$ is a union  of disks, $\ker(D_{u|S}^{\bv})$ consists of constant functions. 
The matching conditions at the edges $T_e$ connecting components $S_{v_-}, S_{v_+}$
impose a codimension one set of constraints,
so that $\ker(D_u^{\bv})$ is one-dimensional.  
Since translations act on the space
 of buildings, a disk $u$ is rigid if and only if $\ker( \ti{D}_{u_Y} )$ vanishes, which is to say that $u_Y$ is rigid. 
\end{proof} 

Because the equation on the segments is invariant, the projection map above extends naturally to a map of treed disks:

\begin{lemma}\label{lem:projmap} Composing with the projection $\R \times Z \to Y$ defines a 
map of treed disks $\M_\bGamma(\Lambda)$  of type $\bGamma$
to treed disks $\M_{p(\bGamma)}(\Pi)$, where $p(\bGamma)$ is the 
map type for treed disks obtained by  the corresponding map on homotopy groups on each component.   Given a rigid map $u \in \M_\bGamma(\Lambda)$
the corresponding map $p \circ u \in \M_{p(\bGamma)}(\Pi)$ is rigid, 
and the converse holds if the vertical part of the index vanishes. The projection map $\M(\Lambda)_0 \to \M(\Pi)_0 $  is orientation preserving.
\end{lemma}

\begin{proof}  The first few statements are immediate from the definitions.  For the statement on orientations, recall from \cite[Section \ref{I-sec:orientations}]{BCSW1} that the the orientations on moduli spaces
$\M(\Lambda),\M(\Pi)$ are induced by deformations of spin structure.  The Cauchy-Riemann operator
for $\M(\Lambda)$ splits into vertical and horizontal parts, and these may be deformed separately.  Namely,
let $u: S \to \R \times Z$ be a rigid holomorphic map bounding 
$\R \times \Lambda$.  The 
kernel of the linearized map ${D}_u$
is naturally isomorphic to the direct sum 
\[ \ker({D}_u) \cong 
\ker({D}_u^{\bv}) \oplus  \ker({D}_u^{\bh}) \cong \R \oplus \ker({D}_{p \circ u}) .\]
Hence
\[ \ker({D}_u)/\R \cong \ker(\ti{D}_{p \circ u})  .\]
To see
that this isomorphism is orientation preserving, note that the orientations on both sides are obtained by capping off the strip-like ends to obtain a surface without strip-like
ends $\ol{S}_v$ and deforming to the operator obtained by gluing together a trivial Cauchy-Riemann operator on a disk $\ol{S}_v'$ with a Cauchy-Riemann operator on a sphere $\ol{S}_v''$.   Then the operator on the vertical part is the standard Cauchy-Riemann operator $ {D}_u^{\bv}$ on a trivial bundle with trivial boundary conditions.  Thus $D_u^{\bv}$ has trivial cokernel
and one-dimensional kernel canonically identified with $\R$.
The capping operation induces the standard orientation on this factor,  and so the orientation is that induced by the horizontal part. 
\end{proof}

We apply the above results to 
the disks that contribute to the disk potential of the Lagrangian projection.    Let $Y$ be a monotone symplectic manifold and $\Pi \subset Y$
a compact, connected, relatively spin, monotone Lagrangian submanifold.   Recall the {\em disk potential}
defined as follows, as in for example Cho-Oh \cite{chooh:fano}.  Choose a generic almost complex structure $J$, a generic point $\pt \in \Pi$ and let $\M(\Pi,\pt)_0$ denote the moduli space
of rigid (necessarily Maslov index two) $J$-holomorphic disks passing through $\pt \in \Pi$.

\begin{definition}
Let $\Pi \subset Y$ be a compact monotone relatively spin Lagrangian.
The disk potential of $\Pi$ is the function $W_\Pi$ on the space of local systems $\Rep(\Pi)$ obtained by counting Maslov-index-two-disks weighted by holonomy:
\[ W_\Pi: \Rep(\Pi) \to \C^\times , \quad \rho \mapsto   \sum_{u \in \M(\Pi, \pt)_0 }  \rho( [\partial u]) .\]
\end{definition}

By Proposition \ref{prop:liftprop}, disks contributing to the potential in the base lift to punctured
disks in the symplectization of a circle bundle contributing to the contact differential. 
Let $Z$ denote a circle bundle with connection $\alpha$ whose curvature is $-\omega$.  Since $c_1(Z) = -[\omega]$, we have 
\[ c_1(Y) = -\tau_Y c_1(Z) \] 
where $\tau_Y$ is the monotonicity constant for $(Y,\omega)$.  Let $\Lambda \subset Z$ denote a horizontal lift of $\Pi$.

\begin{corollary}  Suppose $Y$ is monotone with monotonicity constant $\tau_Y$ and $\Pi \subset Y$ a monotone Lagrangian.  Any Maslov index two disk $(u_Y: \ol{S} \to Y) \in \M(\Pi,\pt)_0$ lifts to a 
rigid once-punctured disk $(u: S \to \R \times Z) \in \M(\Lambda,\aa)_0$ asymptotic to a Reeb chord
with action $1/\tau_Y$ at the puncture. 
\end{corollary}

\begin{proof}  Let $u_Y:  \ol{S} \to Y$ be a disk of Maslov index two.   By the monotonicity assumption, the symplectic
area of $u_Y$ is $1/\tau_Y$.  The result then follows from the angle
change formula \eqref{eq:anglechange}.
\end{proof}

\begin{example} \label{cliffleg5} In the case 
of the Clifford torus, continuing Example \ref{cliffleg4}, the Maslov index two disks
have simple formulas given in \eqref{filldisks}.
\end{example}

Finally we discuss lifts of constant disks.   The following Proposition shows in particular that lifts of such disks are rigid only if they have three punctures:

  \begin{proposition} \label{prop:vertex} Let $Z \to Y$ be a circle-fibered contact  manifold and $\Lambda \subset Z$ a
  compact Legendrian submanifold.   Let $u: C \to \R \times Z$ bounding $\R \times \Lambda$ be a rigid stable holomorphic disk with only boundary punctures whose image in $Y$ is constant.   Then $u$ is a three-punctured disk with incoming Reeb chord $\gamma_{02}$ and incoming  Reeb chords $\gamma_{01}, \gamma_{12}$
  so that $\gamma_{02}$ is the concatenation of $\gamma_{01}$
  and $\gamma_{12}$, as in Figure \ref{fig:zeroarea}.  
  \end{proposition} 
  
  \begin{proof} Such maps are discussed in more detail in the context of symplectic field theory in Ekholm \cite[Section 6.1.E]{ek:ft}.   We first show that any rigid such map $u$ has exactly three punctures by an index computation.    The bundles 
  \[ u^* T (\R \times Z) \to S  , \quad 
  (\partial u)^* T (\R \times \Lambda) \to \partial S  \] 
  are trivial in this case, and the reader may check that the kernel of $D_u$ is one-dimensional and the   cokernel vanishes; in particular these maps are regular.     Assuming only boundary punctures, The index of the trivial Cauchy-Riemann operator  $ \rank(F) + d(\white)-3$ = $\d(\white)  - 2$, and after 
  quotienting the action of $\R$ by translation, the dimension is $d(\white) - 3$.  
  By rigidity, the number of punctures is $d(\white) = 3$.
    \end{proof}

\begin{remark} 
  The existence of such  maps $u$ in the fibers of $\R \times Z \to Y$
  follows, for example, from the Cartheodory-Torhorst extension of the
  Riemann mapping theorem, as explained in Rempe-Gillen \cite{rempe}.
  Torhorst \cite{torhorst} shows that the biholomorphism on the interior of a domain 
(in this case the image $u(S)$ as a subset of a fiber)  has a continuous
  extension if the boundary (in this case $u(\partial S)$) is locally connected.  The map $u$ is a diffeomorphism on the interior, and from this it follows that $u$ is unique up to automorphism of the domain.  
\end{remark}
 
\begin{figure}[ht]
    \centering
    \scalebox{.7}{
\begingroup%
  \makeatletter%
  \providecommand\color[2][]{%
    \errmessage{(Inkscape) Color is used for the text in Inkscape, but the package 'color.sty' is not loaded}%
    \renewcommand\color[2][]{}%
  }%
  \providecommand\transparent[1]{%
    \errmessage{(Inkscape) Transparency is used (non-zero) for the text in Inkscape, but the package 'transparent.sty' is not loaded}%
    \renewcommand\transparent[1]{}%
  }%
  \providecommand\rotatebox[2]{#2}%
  \newcommand*\fsize{\dimexpr\f@size pt\relax}%
  \newcommand*\lineheight[1]{\fontsize{\fsize}{#1\fsize}\selectfont}%
  \ifx\svgwidth\undefined%
    \setlength{\unitlength}{486.27841259bp}%
    \ifx\svgscale\undefined%
      \relax%
    \else%
      \setlength{\unitlength}{\unitlength * \real{\svgscale}}%
    \fi%
  \else%
    \setlength{\unitlength}{\svgwidth}%
  \fi%
  \global\let\svgwidth\undefined%
  \global\let\svgscale\undefined%
  \makeatother%
  \begin{picture}(1,0.13276871)%
    \lineheight{1}%
    \setlength\tabcolsep{0pt}%
    \put(0,0){\includegraphics[width=\unitlength,page=1]{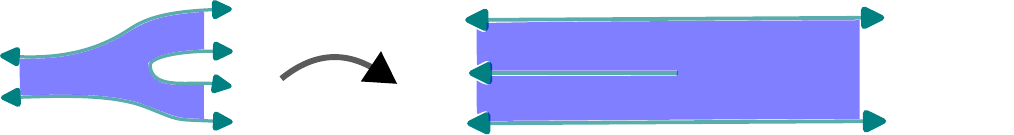}}%
    \put(0.85818361,0.06235299){\color[rgb]{0,0,0}\makebox(0,0)[lt]{\lineheight{1.25}\smash{\begin{tabular}[t]{l}$\gamma_{00}$\end{tabular}}}}%
    \put(0.37796421,0.03082354){\color[rgb]{0,0,0}\makebox(0,0)[lt]{\lineheight{1.25}\smash{\begin{tabular}[t]{l}$\gamma_{01}$\end{tabular}}}}%
    \put(0.37230505,0.08903191){\color[rgb]{0,0,0}\makebox(0,0)[lt]{\lineheight{1.25}\smash{\begin{tabular}[t]{l}$\gamma_{10}$\end{tabular}}}}%
  \end{picture}%
\endgroup%
}
    \caption{A zero-area map with three punctures}
    \label{fig:zeroarea}
\end{figure}
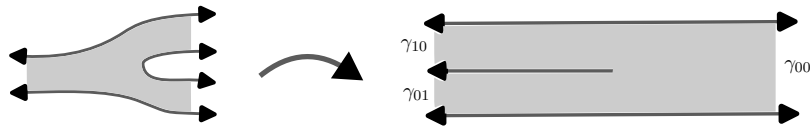

\subsection{Disks bounding toric Lagrangians}
\label{toricleg}

We describe some classification results in the case that the Legendrian is a horizontal lift of a Lagrangian torus orbit in a toric variety.   By the equivalence with geometric invariant theory quotients, one may
obtain a classification of disks in toric varieties as in Cho-Oh \cite{chooh:fano}.   Recall from Delzant \cite{delzant} that any projective toric variety  is equivariantly symplectomorphic to a symplectic quotient of a vector space by a Hamiltonian torus action.   The polytopes occurring  are called {\em Delzant polytopes}, and have the property that the primitive lattice vectors to the facets at any vertex form a lattice basis.   Let $Y_P$ be a symplectic toric manifold correponding to a Delzant polytope $P$ as in  \cite{delzant}.   This description was used by Cho-Oh \cite{chooh:fano} to obtain a description of the holomorphic disks bounding toric moment fibers.  As explained in Kirwan's thesis \cite{ki:coh}, the toric variety may be  obtained alternatively as the K\"ahler quotient of the semistable locus by the complexified group action. Denote the semistable locus, given as the union of orbits of the complex orbit through the zero level set, by
\[ \C^{k,\ss} = H_\C \Psi^{-1}(0)  \subset \C^k \]
The \textit{K\"ahler quotient} is the quotient of the stable locus by the complexified
group action 
\[ Y_P = \C^{k,\ss} / H_\C \] 
and  is diffeomorphic to the symplectic quotient as explained in \cite{ki:coh}.
Here we use the assumption that $H_\C$ acts with only finite stabilizers on the semistable locus.  Let $\Pi$ be a Lagrangian orbit of the residual group $T = U(1)^k/ H$ acting on $Y_P$.

\begin{lemma} \label{lem:blaschkelem} (Cho-Oh \cite{chooh:fano} )
Let $Y_P = \C^k \qu H $ as above.   Any holomorphic disk $u: S \to Y_P $ bounding a Lagrangian torus fiber $\Pi$ is given by a tuple of Blaschke products
\begin{equation} \label{blaschke2} 
  u:  S \to Y_P, \quad z \mapsto  \left( c_j \prod_{i=1}^{d(j)} \frac{z - a_i}{ 1 - \ol{a_i} z}
  \right)_{j=1}^k  .\end{equation}
\end{lemma} 

The proof uses the fact disks in a git quotient lift to the original space:  In this case the pull-back of the  bundle $\C^{k,\ss} \to Y_P$ to any disk is trivial.   Any disk to $Y_P$ bounding $\Pi$ lifts to a disk in $\C^{k,\ss}$ bounding the product of circles in the factors.  Each such factor is necessarily a Blaschke product of the form in \eqref{blaschke2}.

\begin{example} For example, complex projective space $\CP^{n-1}$ is the symplectic quotient of $\C^n$.
The Maslov index two disks in $\CP^{n-1}$ have lifts to $\C^n$ of the form 
\[ u: S \to \C^n, \quad z \mapsto [e^{i \theta_1}, \ldots, e^{i \theta_{k-1}}, e^{i \theta_k} z, 
e^{i\theta_{k+1}}, \ldots, e^{i \theta_n } ] \]
for some constants $\theta_1,\ldots,\theta_n \in \R.$
\end{example} 

\subsection{Disks in the Harvey-Lawson filling}
\label{sec:diskshl}

In this section, we classify some of the disk bounding the Harvey-Lawson filling in Example \ref{ex:HLfilling}.

\begin{lemma}
Any collection of Blaschke products $u_1,\ldots, u_{n-2}$ in the first $n-2$ coordinates as in \eqref{blaschke2} defines a disk  $u= (u_1,\ldots, u_{n-2}, 0,0): S \to \C^n$ bounding the Harvey-Lawson Lagrangian $L_{(1)}$. 
  \end{lemma}
  
  \begin{proof} The statement of the Lemma follows since the intersection  of the Harvey-Lawson Lagrangian with $\C^{n-2} \times \{(0,0)\}$ is a standard Lagrangian torus $(S^1)^{n-2}$ in $\C^{n-2}$.
\end{proof}

\begin{example} 
  In particular, for $n = 3$ the {\em basic disk} bounding the Harvey-Lawson Lagrangian 
  $L_{(1)} \subset \C^3$ is defined by 
\begin{equation} \label{basic}
u : \{ |z| \leq 1 \} \to \C^3 , \quad z \mapsto  ( z, 0, 0 ). \end{equation} 
\end{example}

Generalizations of these basic disks to other fillings are considered in 
Song Yu \cite{songyu}.

\begin{lemma}  \label{lem:allmult}  For the standard complex structure on $\C^n$
all unpunctured holomorphic disks in $\C^n$ bounding 
the Harvey-Lawson filling $L_{(1)}$ have image contained in $\C^{n-2}$.  In particular, 
for $n=3$ any such disk is a multiple cover of (a scalar multiple of) the basic
disk of \eqref{basic}.
\end{lemma} 

\begin{proof} We examine the composition of any such disk with the  map taking the product of components, which  must be a holomorphic disk in the complex line bounding the reals.   Let $S$ be an (unpunctured) disk and $u: S \to \C^n$ a holomorphic map bounding $L$.  The composition of $u$ with the product map
\[ \pi : \C^n \to \C, \quad (z_1,\ldots, z_n) \mapsto z_1 \ldots z_n \] 
must have boundary $\pi \circ u |_ {\partial S}$ contained in $[0,\infty)$.
  Since $\pi \circ u$ is bounded, $\pi \circ u$ must be
  a constant map equal to some $c \in \R$.  If this constant $c$ is
  non-zero then the boundary $\pi \circ u | \partial S$ must lie on an
  $(n-1)$-torus
\[ \{  |z_1|^2 - a_1^2 = |z_2|^2 - a_2^2 = \ldots = |z_n|^2 -a_n^2 = c \}  \subset L_{(1)} .\]
The winding number
\[ [ \pi \circ u | \partial S \cong S^1  ] \in \pi_1(S^1) \cong \Z \] 
of $\pi \circ u$ around $0$ is the sum of the
winding numbers
\[ [ \partial u_j ] \in \pi_1(S^1) \cong \Z \] 
of the components $u_1,\ldots, u_n$ around $0$ and equal to $0$.  Since $u_j$ is holomorphic, each winding number
$ [ \partial u_j ] $ is non-negative, and at least one is non-zero.  This is  a contradiction.
\end{proof}

\section{Circle-fibered Legendrian contact homology}

In this section, we define the Chekanov-Eliashberg algebra and the Legendrian contact homology, which will be generated by words in certain Reeb chords as well as critical points of a Morse function.

\begin{definition}  \label{def:tame2}
A circle-fibered contact manifold $(Z,\alpha)$ with base $(Y, \omega_{Y})$ with $\d\alpha = p^*\omega_{Y}$
is {\em  tame} if for some constant $\tau_Z \ge 1$ and integral symplectic form
\[ \omega_{Y,0} \in \Omega^2(Y,\R), \quad [\omega_{Y,0}] \in H^2(Y,\Z)  \] 
we have
\[   \omega_Y = \tau_Z \omega_{Y,0} \]
and the base is  monotone in the sense that there exists a {\em  monotonicity constant} $\tau_Y \ge 3$ so that 
\[ c_1(Y) = \tau_Y [\omega_{Y,0}] \in H^2(Y,\R) .\] 
Note that since $\omega_Y$ is proportional to $\omega_{Y,0}$, the Chern class $c_1(Y)$ of $Y$ does not depend on the choice of the symplectic form $\omega_Y$. For any compact spin embedded Legendrian $\Lambda \subset Z$ with embedded
image $\Pi \subset Y$ we call $(Z,\Lambda)$ a {\em  tame pair}.   
\end{definition}

The definition guarantees, in particular, that each rigid disk in $\R\times Z$ bounding $\R\times\Lambda$ has at least one incoming puncture, and no punctured spheres are rigid. See Subsection 3.4 in \cite{BCSW1} for a detailed discussion. 

\subsection{Coefficient rings}

Our coefficient ring is the group ring in relative homology over the rationals completed with the filtration 
defined by area.  Recall that our contact manifold fibers over a monotone symplectic manifold $Y$, and our Legendrian $\Lambda$ is a finite cover over a (not necessarily monotone) Lagrangian in the base. 

\begin{definition}
\vskip .1in \noindent 
The {\em group ring  on second homology} is the space of compactly supported
functions 
\[ \C[H_2(Y,\Pi)]:= \{  g: H_2(Y,\Pi) \to \C,  \on{supp}(g) \ \text{compact}  \}.\]

\vskip .1in \noindent The {\em area function} is 
\[ 
A: H_2(Y,\Pi) \to \R , \quad [u_Y] \mapsto \lan [u_Y], [ \omega_Y] \ran
.\]

\vskip .1in \noindent 
The {\em area filtration } is 
\[  \C[H_2(Y,\Pi)] = \bigcup_{A \in \R} \C[H_2(Y,\Pi)]_A  \]
where 
\[ \C[H_2(Y,\Pi)]_A = \{ g \in
\C[H_2(Y,\Pi)] ,  A(\beta) < A \implies g(\beta) = 0 \} \] 
is the subgroup of functions supported on classes with area at least $A$. 

\vskip .1in \noindent The {\em group ring} for $\Lambda$ is 
\begin{eqnarray} \label{glam} 
G(\Lambda) = \Set{ g: H_2(Y,\Pi) \to \C | \# \on{supp}(g) < \infty }.
\end{eqnarray}

\vskip .1in \noindent The {\em completed group ring} is 
\begin{eqnarray} \label{cglam} 
\mhat{G}(\Lambda) &=& \lim_{\infty \leftarrow A} \C[H_2(Y,\Pi)]/\C[H_2(Y,\Pi)]_A   \\
\nonumber
&=&  \Set{ g: H_2(Y,\Pi) \to \C | 
\forall A, \# \on{supp}(g)_A < \infty  } \end{eqnarray}
where 
\[ \on{supp}(g)_A := \supp(g) \cap \{c |  A(c) < A  \}. \]
\end{definition}

\begin{remark}  If $(Y,\Pi)$ is monotone then the completion with respect to the area filtration will not be necessary for the definition of the differential and the uncompleted group ring will suffice. However, even in the simplest examples our augmentations will require a completed coefficient ring.  In the monotone case we often choose to work with \label{goodcases1} the group ring $\C[H_1(\Lambda)]$ over the first homology $H_1(\Lambda)$, to keep with standard conventions in the literature.
\end{remark}

The group ring has the natural convolution product.  For any homology class $\mu$ we follow standard abuse of notation in the field and denote by
\[ [\mu] \in \mhat{G}(\Lambda) \]
the delta function at $\mu$, so that
\[ [{\mu_1}] [  {\mu_2}] = [{\mu_1 + \mu_2}] .\]
The product extends in the obvious way to the completion $\mhat{G}(\Lambda).$

\subsection{The chain group}

Chains are generated by words in the generators. We first recall the definition of a Morse datum from \cite{BCSW1}.

\begin{definition} \label{morsedatum2} 
A {\em Morse datum} for $(Z,\Lambda)$ consists of a pair  of vector fields on the space of Reeb chords and on the cylinder on the Legendrian 
\[ \zeta_\white \in \Vect({\cR}(\Lambda)),  \quad  \zeta_\black \in \Vect(\R \times \Lambda)^{\R} \] 
arising as follows:
\begin{enumerate} 
\item  There exists a  Morse function on the space of Reeb chords
\[  f_\white:  {\cR}(\Lambda) \to \R \] 
so that $\zeta_\white$ is the gradient vector field:
\begin{equation} \label{zdiam} 
\zeta_\white := \grad(f_\white) \in \Vect({\cR}(\Lambda)) .\end{equation}
\item There exists a Morse function 
\[ f_\black :   \Lambda \to \R ;\] 
with gradient vector field 
\[ \grad(f_\black) \in \Vect(\Lambda) \] 
so that $\zeta_\black$  is a translation-invariant lift of $\grad(f_\black)$\label{zwhite}.
\end{enumerate}
\end{definition} 

\begin{remark}
    Each component of the space of Reeb chords is isomorphic
to the Legendrian, and we could take the vector fields to be equal under the identification
of the various components; however, this assumption is not necessary.
\end{remark}

\begin{definition}  A vector field $\zeta_\black \in \Vect(\R \times \Lambda)$ is {\em positive} if there exists a function
\[ a: \Lambda \to \R_{> 0} \]
so that 
\begin{equation} \label{zetablack} \zeta_\black = a(\lambda) \partial_s + {p}^* \grad(f_\black) 
\end{equation}
where 
\[ {p}^* : \Vect(\Lambda) \to \Vect(\R \times \Lambda)^\R  \] 
is the obvious identification of
translationally-invariant vector fields trivial in the $\R$-direction
with 
vector fields on $\Lambda$.
\end{definition}

The limit of any Morse trajectory along any infinite length trajectory is a zero of the gradient vector field.   We introduce labels for the possible limits of the trajectories above as follows. 
Denote by 
\[ \ul{\R} \cong T\R \times \Lambda \] 
the translational factor in
$T(\R \times \Lambda) = T\R \oplus T \Lambda$.  The  zeroes of the vector field
$p_*(\zeta_\black)$ correspond to tangencies of $\zeta_\black$ with the translational factor:
\[  {p}_*(\zeta_\black)^{-1}(0) = \zeta_\black^{-1}(\ul{\R}) \subset
  \R \times \Lambda . \] 

\begin{definition} The union of the zeroes of the vectors fields is denoted 
\begin{equation} \label{gens2} \cI(\Lambda) := 
\cI_\white(\Lambda) \cup \cI_\black(\Lambda), 
\quad  \cI_\white(\Lambda) := \zeta_\white^{-1}(0), \quad \cI_\black(\Lambda) := \zeta_\black^{-1}(\ul{\R}).
\end{equation}
Let $\cW(\Lambda)$ denote the space of ordered words in the generators
$\cI(\Lambda)$ defined as above:
\begin{equation} \label{words} 
\cW(\Lambda) = \bigcup_{d \ge 0} \cI(\Lambda)^d . \end{equation}
For any word $w  = \gamma_1 \ldots \gamma_k \in \cW(\Lambda)$ denote by 
\[ \ell(w) = k \in \Z_{\ge 0} \] 
the length of $w$ and by $\ell_\black(w)$ the number of {\em classical generators}
 $\gamma_i \in \cI_\black(\Lambda)$. The space of {\em contact chains} is 
the completion 
\begin{equation} \label{CE} CE(\Lambda) = \Set{ \sum_{i=1}^\infty c_i
    \Sigma_i |  \ \Sigma_i \in \cW(\Lambda), c_i \in \mhat{G}(\Lambda), \quad \lim_{i \to \infty}( \ell_{\black}(\Sigma_i)) = \infty }
\end{equation} 
of $\mhat{G}(\Lambda)$-valued functions on $\cW(\Lambda)$ with
respect to the filtration given by the classical word length $\ell_\black$.  We also call $CE(\Lambda)$ the Chekanov-Eliashberg algebra of the Legendrian $\Lambda$, with multiplication given by concatenation of words; the
empty word is the unit 1. Denote by
\[ CE_{\ell } (\Lambda) =  \bigoplus_{ \ell(w) = \ell}
\mhat{G}(\Lambda) w \subset CE(\Lambda)   \] 
the subspace generated by words $w$ of length $\ell$.    This ends the Definition.
\end{definition} 

The definition of Chekanov-Eliashberg chains is similar to that used for immersed Lagrangian Floer theory in Akaho-Joyce \cite{akaho} and for Legendrian contact
homology using contact instantons in Oh \cite{oh:sft}.   More generally, for any 
$G(\Lambda)$-module $G$ we denote by 
$CE(\Lambda,G) = CE(\Lambda) \otimes_{G(\Lambda)} G$ the chains with coefficients in $G$.

For the construction of Legendrian contact homology we also need a $\Z_2$-grading. 
Similar to the case of orbifold quantum cohomology, in the monotone case Legendrian contact homology admits 
  a natural grading by the reals, given by a shifted sum of angles of the Reeb chords.  The $\Z_2$-grading will be used to construct orientations on the  moduli spaces, while the grading used on the contact homology groups will be the one induce by the $\R$-grading.

\begin{definition} \label{gradings} {\rm (Gradings)}  Let $\Lambda \subset Z$
be a Legendrian with projection $\Pi \subset Y$.
\begin{enumerate}     
\item {\rm ($\Z_2$-grading)}    If $\Pi$ is embedded then any Reeb orbit
$\gamma$ over a critical point $x \in \cI_\white(\Lambda)$ define  the {\em $\Z_2$-grading}
\[ \deg_{\Z_2}: \cI_\white (\Lambda) \to \Z_2, \quad \gamma \mapsto  \ind_x(f_\white) + 1 \ \text{mod} \  2 \Z  .\] 
\item {\rm ($\R$-grading) } Suppose that $\Pi \subset Y$ is monotone with monotonicity constant $\tau \in \R$.  Define the {\em real grading}
\[ \cI_\white (\Lambda) \to \R, \quad \gamma \mapsto \deg_\R(\gamma) = \on{ind}
(f_\white)(\gamma) + \tau \theta - 1 \] 
where $\theta$ is the action of the Reeb chord $\gamma$.  Define
\[ \cI_{\black}(\Lambda) \to \R, \quad \gamma \mapsto \deg_\R(\gamma) :=
\on{ind} (f_\black)(\gamma)- 1  . \]
\end{enumerate}

For words define  the degree map as the sum of the degrees
of the factors:
\[ \cW(\Lambda) \to \R, \quad \gamma_1 \otimes \ldots \otimes \gamma_k \mapsto
  \deg_\R(\gamma_1 \otimes \ldots \otimes  \gamma_k) := \sum_{i=1}^k \deg_\R(\gamma_i) .\]
\end{definition}

\begin{lemma} If $\Lambda$ is connected and $Y$ is simply-connected, then the actions of Reeb chords are multiples of $2/\tau$ and the real 
grading defines an $\Z$-grading.
\end{lemma}

\begin{proof}
    Let $\gamma$ be a Reeb chord.  Since $\Lambda$ is connected, the start $\gamma(0)$ and end $\gamma(1)$
    of the path are joined by a path $\phi:[0,1] \to \Lambda$, whose projection $p \circ \phi$ is a loop in 
    $\Pi$. Since $Y$ is simply-connected, 
    $p \circ \phi$ bounds a disk $u: S \to Y$.  By Stokes' theorem, the action of $\gamma$
    is the integral of the derivative $\d \alpha$ over the 
    disk $S$, which $1/\tau$ times the Maslov index by monotonicity.    Since $\Pi$ is oriented, the Maslov index is an even integer. 
\end{proof}

\begin{example} Since the real degree of a classical generator is its Morse degree shifted by one we have embeddings of the classical generators of Morse-degree $0$ resp. $1$
\[ CM_0(\Lambda) \subset CE_{-1}(\Lambda), \quad CM_1(\Lambda) \subset CE_0(\Lambda) .\] 
Suppose $\Pi\subset Y$ is connected and monotone with monotonicity constant $\tau$,
and 
\[ \cR(\Lambda)_{2/\tau} \subset \cR(\Lambda) \] 
is the locus of Reeb chords of action $2/\tau.$  We have by definition inclusions
with a degree shift of one
\[ CM_0(\cR(\Lambda)_{2/\tau}) 
\subset CE_1(\Lambda), \quad CM_1(\cR(\Lambda)_{2/\tau}) \subset CE_2(\Lambda) .\]
The remaining components of $\cR(\Lambda)$ contribute generators in higher degrees. 
This ends the Example.
\end{example}

In the case that $\Pi$ is disconnected, there may be generators of negative $\R$-grading, as we discuss
in Section 7 of \cite{BCSW3}.

\begin{example}
    Consider the Clifford Lagrangian $\Pi$ in $\C P^2$, with its Legendrian lift $\Lambda$ in $S^5$. The monotonicity constant for $\Pi$ is $\tau=6$. The minimal action of a Reeb chord is $\theta=2/\tau=1/3$. Therefore, its \textit{shifted} $\R$-grading is 
    \[
    -1+ \on{ind}(f_\white)(\gamma) + \tau \theta= -1+ \on{ind}(f_\white)(\gamma) + 2. 
    \]
    In particular, when it corresponds to a Morse critical point of index zero, its shifted $\R$-grading is $1$.

The shifted grading has the following relation to the Conley-Zehnder index. A simple Reeb orbit in $S^5$ is a circle fiber over a point in $\C P^2$. Its Conley-Zehnder index is $2c_1(\C P^2)(H)=6$ where $H$ is the hyperplane class, computed in \cite[Equation (2.11)]{bkk}.  Therefore, the Conley-Zehnder index of a Reeb chord with action $1/3$ is $6/3=2$. The shifted grading is $2-1=1$, which matches the definition in \cite[Page 182]{ees:lch}, for example. The extra term $\on{ind}(f_\white)(\gamma)$ comes from the Morse-Bott nature.
\end{example}

For later use, we introduce notation for the subspaces generated
by the classical and Reeb generators.  Define 
\[ \cW_\black(\Lambda) = \bigcup_{d \ge 0} \cI_\black(\Lambda)^d ,  \quad
\cW_\white(\Lambda) = \bigcup_{d \ge 0} \cI_\white(\Lambda)^d .\]
and let
\begin{equation} \label{sectors}
CE_\white(\Lambda) \subset CE(\Lambda), \quad CE_\black(\Lambda) \subset CE(\Lambda)  \end{equation}
denote the subspaces generated by words in $\cW_\black(\Lambda)$
resp. $\cW_\white(\Lambda).$  We call $CE_\black(\Lambda)$ the {\em classical sector}.

\subsection{The contact differential} 

As above, 
suppose that $(Z,\Lambda)$ is a tame pair. Assume $\ul{P}$ is a collection of perturbation data that are coherent, stabilizing, and regular as in 
\cite[Section \ref{I-foundsec}]{BCSW1}.  

\begin{definition} The contact differential $\delta
\in \End(CE(\Lambda))$ is the weighted
count of punctured holomorphic surfaces, with all components disks and with one incoming puncture on each component:
\begin{equation} \label{deltadef}
  \delta: CE(\Lambda) \to CE(\Lambda), \
  \ul{\gamma}_- \mapsto \sum_{u \in
    \M(L,\ul{\gamma}_-,\ul{\gamma}_+)_0} (-1)^{\heartsuit} \wt(u) 
    \ul{\gamma}_+ \end{equation}
    where 
\begin{equation} \label{heartsign} \heartsuit = \sum_{i=1}^{d_-} i |\gamma_{-,i}| \end{equation}
$d_-$ is the length of $d_-$ and the weight $\wt(u)$ is defined as the product of factors
    \begin{equation} \label{wu} \wt(u):= c_i d(\white)^{-1} \eps(u)
 [ {p} \circ u ]
 \end{equation}
with $u:C \to X$ being a map from a curve $C$ of some type $\Gamma$, $c_i$ being the coefficient of the perturbation $P_{\Gamma,i}$ in the multivalued perturbation $P_\Gamma$, $d(\white)$ is the number
of interior edges mapping to the Donaldson hypersurface and 
$\eps(u)$ is the orientation sign.  Here the domain type $\Gamma$ ranges over rigid types of disconnected punctured disks with one incoming puncture on each component, and $[u]$ is the homology class of the capped off
boundary defined in \cite[Equation \eqref{I-cappedoff}]{BCSW1}.   This ends the Definition.
\end{definition}

There are many variations using other coefficient rings.
In case that $(Y,\Pi)$ is monotone, one can replace $[{p} \circ u] \in H_2(Y,\Pi)$  in \eqref{wu}] with $[\partial u]$ as defined in \cite[Equation \eqref{I-cappedoff}]{BCSW1}, or work over the complex numbers as coefficient ring.

\begin{proposition}  The map $\delta: CE(\Lambda) \to CE(\Lambda)$ is well-defined.
\end{proposition}

\begin{proof}   We show that the completion with respect to word length makes the infinite sum in the Definition \eqref{deltadef} well-defined.  Let 
\[ \ul{\gamma}_-
\in \cI(\Lambda)^d \]
be a collection of incoming generators. 
It suffices to check that for any particular word length $\ell \in \Z_{\ge 0}$ there
exist finitely many terms in $\delta(\ul{\gamma}_-)$ bounded by the given word length $\ell$, 
and for each such possible output $\ul{\gamma}_+$
and energy $A > 0$, there exist finitely many disks of energy at most $A$
contributing to the coefficient of $\ul{\gamma}_+$ in $\delta(\ul{\gamma}_-)$.
By \cite[Lemma \ref{I-anglechange}]{BCSW1}, the total action in $\ul{\gamma}_+$ is bounded.  
Thus only finitely many Reeb chords in $\cR(\Lambda)$ are possible, for the given input 
$\ul{\gamma}_-$.  It follows that the possibly outputs $\ul{\gamma}_+$
of a rigid configuration are finite in number.  
By Gromov compactness for maps to $(Y,\Pi)$, the number of homology classes
in $(Y,\Pi)$ represented by holomorphic disks of area less than $A$ is finite.
For each such disk there exist finitely many possible lifts to punctured disks mapping to $(\R \times Z, \R \times \Lambda)$,
by Theorem \ref{liftthm}.  It follows that the number of disks contributing to each $\ul{\gamma}_+$ of area below the given constant $A$ is finite.
\end{proof}

The differential can be understood in terms of 
disk counts, rather than counts of disconnected surfaces.

\begin{definition} 
A {\em trivial strip} is a map $u: \R \times [0,1] \to \R \times Z$
that is an open immersion in some fiber ${p}^{-1}(u)$ of the projection $\R \times Z \to Y$
of the form 
\[ u(s,t) =  e^{\theta(s + it)}z \] 
for some constant $\theta$
and point $z \in Z$.
\end{definition}

\begin{lemma}  \label{trivstrip} Any disconnected
stable rigid configuration
$u: C \to X$ has 
exactly one connected component $u_i : C_i \to X$ that is not a trivial strip.
\end{lemma} 

\begin{proof}   The statement of the Lemma follows from the fact that the components of the union may be translated separately.  Suppose that $u : C \to \R \times Z$ is a configuration given as the disjoint union of $u_1$ and $u_2$ with connected components and 
$\lambda \in \R$ is scalar.  The union of  $u_1$
and $\lambda u_2$ is a union of punctured disks, necessarily isomorphic to 
$u$ since $u$ is rigid.  Thus there exists a scalar $\mu(\lambda)$
and an automorphism $\psi: C \to C $ restricting to automorphisms 
\[ \psi_1:= \psi |_{C_1}: C_1 \to C_1, \quad \psi_2:= \psi |_{C_2} :C_2 \to C_2 \] 
such that 
\[ \mu(\lambda) u_1 \circ \psi_1
\sqcup (\mu(\lambda) + \lambda) u_2 \circ \psi_2 = u_1 \sqcup u_2 . \] 
Suppose $\lambda \ne 0$, and without loss of generality $\lambda+ \mu(\lambda)$  is non-zero.   Then
\[ (\lambda + \mu(\lambda))u_2 = u_2 \circ \psi_2^{-1}  .\] 
This equality implies that $u_2$ has an automorphism group  of positive dimension.
Necessarily, $C_2$ is a disk with two punctures, one incoming and one outgoing, and 
$u_2$ maps $C_2$ to a single fiber of the projection $\R \times Z \to Y$. 
Thus $u_2$ is a trivial strip.  
By the stability condition, not all 
the components are trivial strips, so $u_1$
must not be a trivial strip.  The case of maps with more than two connected components in the domain is similar.   
\end{proof} 

\begin{proposition}  The map $\delta$ satisfies the  Leibniz rule
\[  \delta (\gamma_1 \gamma_2) = \delta (\gamma_1) \gamma_2 + (-1)^{\deg_{\Z_2}(\gamma_1)} \gamma_1
\delta (\gamma_2) . \]
\end{proposition} 

\begin{proof}   Since each connected component $C_i$ of the domain $C$ is assumed to have 
one incoming end $e \in \mE(S_i)$, any configuration $u: C \to X$ with two inputs $\gamma_1, \gamma_2$
is the disjoint union of  two components $u_1: C_1 \to X$ and $u_2 : C \to X $
each with a single incoming end. Exactly one of these is a  trivial strip by 
Lemma \ref{trivstrip}.  The statement of the Proposition follows. 
\end{proof}

\begin{theorem} With $(Z,\Lambda)$ and $\ul{P}$ as above, \label{tsquare} $\delta^2 = 0.$
\end{theorem} 

\begin{proof} By the description of the true boundary components of the moduli space of buildings $u \in \M_\Gamma(\Lambda)$ in \cite[Theorem \ref{I-twolevels2}]{BCSW1}, boundary points in the one-dimensional moduli
  spaces correspond to treed configurations with two levels
  $ u = (u_1,u_2)$.   The assumptions
  imply that $\R \times Z$ is tame by \cite[Lemma \ref{I-qtop}]{BCSW1}. 
  Moreover, \cite[Lemma \ref{I-1cons}]{BCSW1} and \cite[Lemma \ref{I-nospheres}]{BCSW1} imply that each component $u_v: C_v \to \XX$ of 
  $u$ is a disk 
  with at least one incoming Reeb chord.   Rigidity then implies
  that in each level $u_i$, exactly one connected component $u_v$ is not a trivial 
  strip, by Lemma \ref{trivstrip}.  Thus 
  $u_v$ has exactly one incoming puncture labelled by a Reeb chord $\gamma_-$ and 
  some non-negative number of outgoing punctures.  To check signs, 
  (the sign computation for the case of non-degenerate Reeb chords was
  carried out in Ekholm-Etnyre-Sullivan \cite{ees:orient})  let %
  \[ 
  \delta_d(\gamma_0;\gamma_1,\ldots,\gamma_d) \in \mhat{G}(\Lambda) \] 
  be the coefficient 
  of $\gamma_1 \ldots \gamma_d$ in $\delta_d(\gamma_0)$.    The coefficient
  of $\gamma_1 \ldots \gamma_d$ in $\delta^2(\gamma_0)$ is 
  \begin{multline} \sum_{e + f = d,\gamma_0'}  (-1)^{ \deg_{\Z_2}(\gamma_1) + 
  \ldots + \deg_{\Z_2}(\gamma_e)}
  \delta_{d - e+ 1}(\gamma_0,\gamma_1,\dots,\gamma_e, \gamma_0', \gamma_{e+f + 1},\ldots, \gamma_d) \\
  \delta_e(\gamma_0'; \gamma_{e+1}, \ldots \gamma_{e+f}) = 0.  
  \end{multline}
  This  identity is the \ainfty relation for the Lagrangian projection $\Pi \subset Y$
  and follows by the same computation; the most standard reference would be Seidel's book 
  \cite{se:bo} while our particular conventions are detailed in Charest-Woodward \cite{flips}.
\end{proof} 

\begin{definition} 
The contact homology of a compact, spin Legendrian $\Lambda$ in a compact circle-fibered
contact manifold $(Z,\alpha)$ is 
\[ HE(\Lambda) = \frac{\ker(\delta)}{\im(\delta)} .\]
\end{definition} 

\begin{lemma} The differential $\delta$ has degree $-1$ with respect to 
any of the gradings (over $\Z_2$, $\Z$, or $\R$) in Definition \ref{gradings}
that may be defined.  In particular, in the case that $\Pi$ is monotone, 
the homology $HE(\Lambda)$ inherits an $\R$-grading and $\Z_2$-grading from $CE(\Lambda)$.
\end{lemma}

\begin{proof} We apply the angle change formula \cite[Lemma \ref{I-anglechange}]{BCSW1}, 
which is an application of Stokes' formula.  Suppose $u: C \to X$ is a rigid punctured disk.  By \cite[Lemma \ref{I-anglechange}]{BCSW1} and monotonicity,
\begin{eqnarray*}
\sum_{e \in \cE_+ (S)} (\tau \theta_e  - 1) - 
\sum_{e \in \cE_-(S)} (\tau \theta_e - 1) &=&
(d- 2) + 2 \tau \int_{S} u_Y^* \omega_Y  \\
&=&  (d - 2) +  I(u_Y)  = 1 . 
\end{eqnarray*}
It follows that the degrees of the generators
are related by 
\[   
\sum_{e \in \cE_+ (S)} \deg_\R \gamma_e = 
\sum_{e \in \cE_- (S)} \deg_\R \gamma_e - 1. \]
The claim for the real degree follows. The claim for the $\Z_2$-degree is the same as in the Fukaya algebra case, and left to the reader.
\end{proof}

\begin{lemma} Let $(Z,\Lambda)$ be a tame pair as above.  If $CE(\Lambda)$ is defined
over the group ring $H_1(\Lambda)$ resp. 
$H_2(Z,\Lambda)$ then the contact differential $\delta$ is independent, up to isomorphism, of the choice of capping paths resp. capping disks. \end{lemma}

\begin{proof} 
Given two choices of capping paths 
\[ \hat{\gamma}_b: [0,1] \to \Lambda, \quad \hat{\gamma}_b': [0,1] \to \Lambda, \quad b \in \{ 0, 1 \} \]
let 
\[ \delta: CE(\Lambda) \to CE(\Lambda), \quad \delta' : CE(\Lambda) \to CE(\Lambda) \]
be the corresponding differentials.     The difference between the capping paths is a collection of loops 
\[ \thorn_b(\gamma) : S^1 \to \Lambda, \quad \thorn := \hat{\gamma}_b^{-1} \# \hat{\gamma}_b' .\]
Given a disk $u: C \to X$ contributing to $\delta$ and $\delta'$, the homology classes
$[\partial u]$ and $[\partial u']$ are related by 
\[ [ \partial u]' = [\partial u] \prod_\gamma \thorn(\gamma) \] 
where the product is over Reeb chords at the punctures of $u.$  It follows that the map 
\[ CE(\Lambda) \to CE(\Lambda), \quad \gamma \mapsto \thorn_0(\gamma)^{-1} \thorn_1(\gamma) \gamma \]
(with products taken as delta functions in the group ring)  intertwines the differentials $\delta$ and $\delta'$.  The case of coefficient
ring the group ring on $H_2(Z,\Lambda)$ is similar.
\end{proof}

We next compute some of the lowest order contributions to the differential.  In particular,
there are contributions to the differential from zero area disks with no punctures, which already
arise in the Morse \ainfty algebra of the Legendrian.  Recall from \eqref{sectors} that the {\em classical sector} $CE_\black(\Lambda) \subset CE(\Lambda)$ is generated by words 
$\cW_\black(\Lambda)$ on critical points
of the Morse function on the Legendrian.

\begin{lemma} \label{preserves} The differential $\delta$ preserves the classical sector 
$ CE_\black(\Lambda)$.  Furthermore, for any generator $\gamma
\in \cI_\black(\Lambda)$ the image $\delta(\gamma) \in CE(\Lambda)$
is equal to the Morse boundary $\delta_{\on{Morse}}(\gamma)$ plus multiples of words
$w \in \cW(\Lambda)$ of length at least two.
\end{lemma}

\begin{proof}    
By \cite[Corollary \ref{I-onepos}]{BCSW1}, any non-constant punctured disk has at least one incoming puncture.  Therefore the disks $u: C \to \R \times Z$ contributing to $\delta | CE_\black(\Lambda)$ are all constant, and
have no outgoing punctures.     Because of stability, each disk
component $u_v: S_v \to \R \times Z$ must have at least three adjacent edges $T_e \subset C$,
so $u$ must have at least one outgoing edge.  If there is a single outgoing edge, then $u$ consists of a 
Morse trajectory $u : C \cong \R \to \R \times Z$ and no disks. 
Thus $u$ contributes to the Morse differential $\delta_{\on{Morse}}$. 
Otherwise, there are at least two outgoing edges $T_{e_1},T_{e_2}$ and the output $\ul{\gamma}_+ \in \cW(\Lambda)$ is a word of length at least two.
\end{proof}

\begin{example} \label{twotorus} Suppose that $\Lambda = T^2$ is the two-torus with its standard Morse function given by the sum of the height functions on each of its circle factors. 
Let $\bb_\black \in \cI_\black(\Lambda)$ be the unique degree one (Morse degree two) generator.   The Morse
differential $\delta_{\on{Morse}}(\bb_\black) = 0$.  On the other hand, let 
\[ \cc_{\black,1},\cc_{\black,2} \in \cI_\black(\Lambda) \]
be the degree zero (Morse degree one) generators.  The unstable manifolds for $\cc_{\black,1},\cc_{\black,2}$
intersect in a unique point, contained in the (dimension two) stable manifold of $\bb_\black$ for generic choices of Morse datum.  As a result,
the leading order terms in $\delta(\bb_\black)$ (in the sense of lowest area and fewest outputs) are 
\[ \delta(\bb_\black) = \pm \cc_{\black,1} \cc_{\black,2} \mp  \cc_{\black,2} \cc_{\black,1}  + \ldots . \] 
Let $\aa_\black$ denote the degree $-1$ (Morse degree zero) generator in $\cI_\black(\Lambda).$  Constant disks to $\bb_\black$
with two outputs labelled $\aa_\black,\bb_\black$ and input labelled $\bb$ are rigid and contribute
\[ \delta(\bb_\black) = \ldots + \bb_\black \aa_\black  + \aa_\black \bb_\black  + \ldots .\]
The  higher order terms depend on choices of orientations  and the additional terms depend on the particular perturbation scheme (for example, $\aa_\black$ may be dual to a strict unit in the Morse \ainfty algebra of $\Lambda$, or not).
\end{example}

The structure coefficients of the differential on the classical sector for $\Lambda$ are equal to structure coefficients of the Morse \ainfty algebra for the Lagrangian $\Pi$.   Since self-intersections of cycles are not transverse, these structure coefficients depend on the perturbations chosen. 

\begin{remark}
The differential induces a {\em reduced differential } $\delta_\white$ on 
the space of chains $  CE_\white(\Lambda)$ generated by words on $\cI_\white(\Lambda)$
corresponding to Reeb chords by omitting the classical generators
in the definition of $\cW(\Lambda)$ in \eqref{words}. Define 
\[ HE_\white(\Lambda) = \frac{ \on{ker}(\delta_\white) }{  \on{im}(\delta_\white)} .\]
Although there is a well-defined homology theory corresponding to allowing Reeb generators
only, cobordisms such as the Harvey-Lawson filling do not in general define chain maps 
for this theory, at least directly and we will not have any use for the homology theory 
$HE_\white(\Lambda)$ in this paper. 
\end{remark} 

\subsection{Special perturbations}
\label{divisor}

In this section,  we show the existence of perturbations with special properties so that, for example, a version of the divisor equation holds. The basic 
phenomenon is the following:  
 Consider the situation that all holomorphic disks are made
regular by some domain-independent almost complex structure.  

\begin{definition}
Let $u: C \to X$ be a rigid configuration with such a leaf $T_e $.   For any integer $k \ge 1$, define a new configuration $u_{(k)}: C_{(k)} \to X$ by replacing $T_e$
with the union of a disk $S_v$ and a collection $T_{e_1},\ldots, T_{e_k}$
of $k$ copies of $T_e$:
\[ C_{(k)} =  C_{(k),z_1,\ldots, z_k} = (C - T_e) \cup (T_{e_1} \cup \ldots \cup T_{e_k}) \]
with attaching points chosen at some points $z_1,\ldots, z_k$ on the boundary of $S_v$.  There is 
A configuration $u_{(k)}$ obtained from a configuration $u$ by replacing 
an edge $T_e$ by $k$-copies, attached to a constant disk $S_v$
is said to be {\em  obtained repeating inputs}.
\end{definition}

The configurations described by the Definition are typically not regular.  Indeed
for $k \ge 2$ the configuration $u_{(k)}$ lies in a stratum of dimension at least 
$k - 2$ more than that containing $u$, which is not equal to the expected
dimension.  

 \begin{figure}[ht]
     \centering
\begingroup%
  \makeatletter%
  \providecommand\color[2][]{%
    \errmessage{(Inkscape) Color is used for the text in Inkscape, but the package 'color.sty' is not loaded}%
    \renewcommand\color[2][]{}%
  }%
  \providecommand\transparent[1]{%
    \errmessage{(Inkscape) Transparency is used (non-zero) for the text in Inkscape, but the package 'transparent.sty' is not loaded}%
    \renewcommand\transparent[1]{}%
  }%
  \providecommand\rotatebox[2]{#2}%
  \newcommand*\fsize{\dimexpr\f@size pt\relax}%
  \newcommand*\lineheight[1]{\fontsize{\fsize}{#1\fsize}\selectfont}%
  \ifx\svgwidth\undefined%
    \setlength{\unitlength}{288.03092176bp}%
    \ifx\svgscale\undefined%
      \relax%
    \else%
      \setlength{\unitlength}{\unitlength * \real{\svgscale}}%
    \fi%
  \else%
    \setlength{\unitlength}{\svgwidth}%
  \fi%
  \global\let\svgwidth\undefined%
  \global\let\svgscale\undefined%
  \makeatother%
  \begin{picture}(1,0.59376026)%
    \lineheight{1}%
    \setlength\tabcolsep{0pt}%
    \put(0,0){\includegraphics[width=\unitlength,page=1]{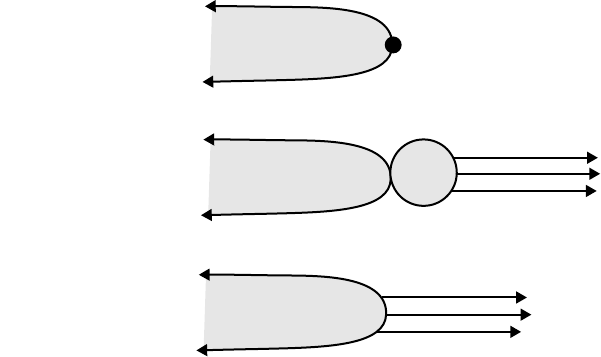}}%
    \put(0.66829334,0.29842701){\color[rgb]{0,0,0}\makebox(0,0)[lt]{\lineheight{1.25}\smash{\begin{tabular}[t]{l}E=0\end{tabular}}}}%
    \put(0.25865806,0.51984837){\color[rgb]{0,0,0}\makebox(0,0)[lt]{\lineheight{1.25}\smash{\begin{tabular}[t]{l}$u$\end{tabular}}}}%
    \put(0.22123389,0.06173514){\color[rgb]{0,0,0}\makebox(0,0)[lt]{\lineheight{1.25}\smash{\begin{tabular}[t]{l}$u_d$\end{tabular}}}}%
    \put(0.22000018,0.28069561){\color[rgb]{0,0,0}\makebox(0,0)[lt]{\lineheight{1.25}\smash{\begin{tabular}[t]{l}$u_d^{\on{sing}}$\end{tabular}}}}%
  \end{picture}%
\endgroup%

     \caption{Adding multiple trajectories at an intersection point}
     \label{adding}
 \end{figure}
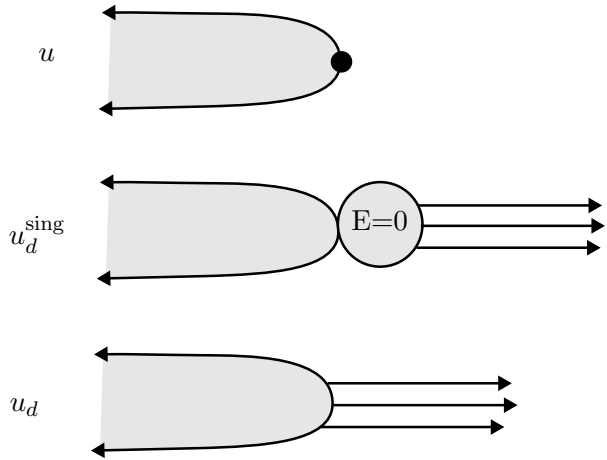

The ``canonical" way of constructing perturbations, as in the divisor equation in the theory of Fukaya algebras, is to construct the perturbations in a symmetric way so that the contribution of such configurations is an inverse factorial.  The simplest case of the Theorem
which we wish to prove is the following:

\begin{theorem} \label{single} Suppose $(X,L)$ is a tame cobordism pair, 
$u: C \to X$ is a rigid holomorphic treed disk bounding $L$ with weight $\wt(u)$, one incoming edge labelled $\aa$ and no outgoing edges
as above.  There exist coherent perturbations so that, if  $\Sigma$
intersects $\partial u$ once positively, then the coefficient of  $\cc^d$ in $\delta(\aa)$ is $1/d!$.
\end{theorem}

More generally we consider the following situation where various intersections with 
codimension one cycles are repeated  Let 
\[ \Sigma_1^s,\ldots, \Sigma_l^s \subset \Lambda \] 
be the stable manifolds of the index one critical points
$x_1,\ldots, x_l$ corresponding to  one generators 
\[ \cc_1,\ldots, \cc_l \in \cI(\Lambda) .\] 
Suppose $u: C \to X$ is a reduced configuration meeting these codimension one manifolds
transversally in cyclic order $\Sigma_{i_1}^s,\ldots, \Sigma_{i_r}^s$  around the boundary. 
Let
\[ d = (d_1,\ldots, d_r) \in \Z^r \] 
 a sequence of integers
  representing the number of repetitions of the labels $\cc_{i_1},\ldots, 
  \cc_{i_r}$ corresponding to the intersection points.   Suppose $u'$ has $d_j$ leaves labelled $x_j$ nearby 
these intersection points for $j = 1,\ldots, r$.
The more general version of Theorem \ref{single}
which allows for several intersections is the following:

 \begin{theorem}\label{repeating} \label{expprop} 
 Let $(X,L)$ be a tame cobordism pair. 
 Suppose that for vanishing perturbations, every reduced
configuration is regular and meets the codimension one stable manifolds $\Sigma_{x_1}^s,\ldots, \Sigma_{x_l}^s$ corresponding to degree zero generators 
$x_1,\ldots, x_l$ of 
$\cI(\Lambda_+)$ transversally.   There exist regular perturbations  $P_{\Gamma}$   satisfying the conditions in \cite[Theorem \ref{I-twolevels}]{BCSW1}
(coherent, stabilizing, and regular) for which for all collections
  $\ul{\gamma}_+$ of degree one there exists a forgetful map 
\[ \bigcup_{\Gamma \to \Gamma_0}  \M_\Gamma(L,{\gamma}_-,\ul{\gamma}_+) \to \M^{\red}_{\Gamma_\circ}(L,{\gamma}_-), \quad u \mapsto u_\circ \]
for the corresponding type $\Gamma_\circ$
with no outputs which has weighted count over each fiber given by 
\begin{equation} \label{wtud}
\wt(u_{(d)}) = \prod_{j=1}^r  (\pm 1)^{d_j} (d_j!)^{-1} \wt(u) \end{equation} 
  with sign $+1$ resp. $-1$ if $\partial u$ intersects $\Sigma_{x_j}^s$
  positively resp. negatively. 
  \end{theorem}  

\begin{proof}  Consider the following averaging procedure.  Let $\ul{P} = (P_\Gamma)$ be a generic perturbation  obtained by perturbing the Morse functions on the edges $T_{e_j}$ that are semi-infinite.  
The stable manifold of $x_j$  intersects each boundary $u| \partial C$ at a collection of distinct points 
\[ z_1,\ldots, z_{d_j} \in \partial C  \] 
which may not be in cyclic order around the boundary.  From such a perturbation and a choice of permutation $g \in \Sigma_{d_j}$ of $d_j$ letters,
one obtains a new perturbation $g P_\Gamma$ by pulling back $P_\Gamma$. 

For a perturbation system given by such an averaging procedure, the weight of each configuration is given by an inverse factorial. 
Indeed, the intersection points $z_1,\ldots, z_{{d_k}} \in \partial S$
will lie in cyclic order around the boundary for exactly one of these perturbations, hence
the factorial  $(d_j!)^{-1}$ in the statement of the Proposition.  The determinant lines are isomorphic except from the sign
$(\pm 1)^{d_j}$
arising from the intersection of $\partial u$ with the oriented
stable manifold for $x_j$.  The sign is obtained by definition 
of the orientation as induced from enforcing the matching condition 
at the intersection of $T_{e_j}$ with the surface part $S$; 
since the matching condition is cut out by the diagonal, the sign 
is positive exactly if the intersection is positive. 

  The fibers of the forgetful map 
consist of ways of ordering the intersection points with the hypersurfaces. 
For each intersection $z$ of $u_\circ(\partial S_\circ)$ with $\Sigma_{x_j}^s$,
there may be  a collection of boundary leaves mapping to perturbations of $\Sigma_{x_j}^s$;
or similarly for the intersections of $u_\circ(T_\circ)$ as in Figure \ref{forgetful}.
\begin{figure}[ht]
    \centering
    \scalebox{.5}{
\begingroup%
  \makeatletter%
  \providecommand\color[2][]{%
    \errmessage{(Inkscape) Color is used for the text in Inkscape, but the package 'color.sty' is not loaded}%
    \renewcommand\color[2][]{}%
  }%
  \providecommand\transparent[1]{%
    \errmessage{(Inkscape) Transparency is used (non-zero) for the text in Inkscape, but the package 'transparent.sty' is not loaded}%
    \renewcommand\transparent[1]{}%
  }%
  \providecommand\rotatebox[2]{#2}%
  \newcommand*\fsize{\dimexpr\f@size pt\relax}%
  \newcommand*\lineheight[1]{\fontsize{\fsize}{#1\fsize}\selectfont}%
  \ifx\svgwidth\undefined%
    \setlength{\unitlength}{666.38308548bp}%
    \ifx\svgscale\undefined%
      \relax%
    \else%
      \setlength{\unitlength}{\unitlength * \real{\svgscale}}%
    \fi%
  \else%
    \setlength{\unitlength}{\svgwidth}%
  \fi%
  \global\let\svgwidth\undefined%
  \global\let\svgscale\undefined%
  \makeatother%
  \begin{picture}(1,0.82405031)%
    \lineheight{1}%
    \setlength\tabcolsep{0pt}%
    \put(0,0){\includegraphics[width=\unitlength,page=1]{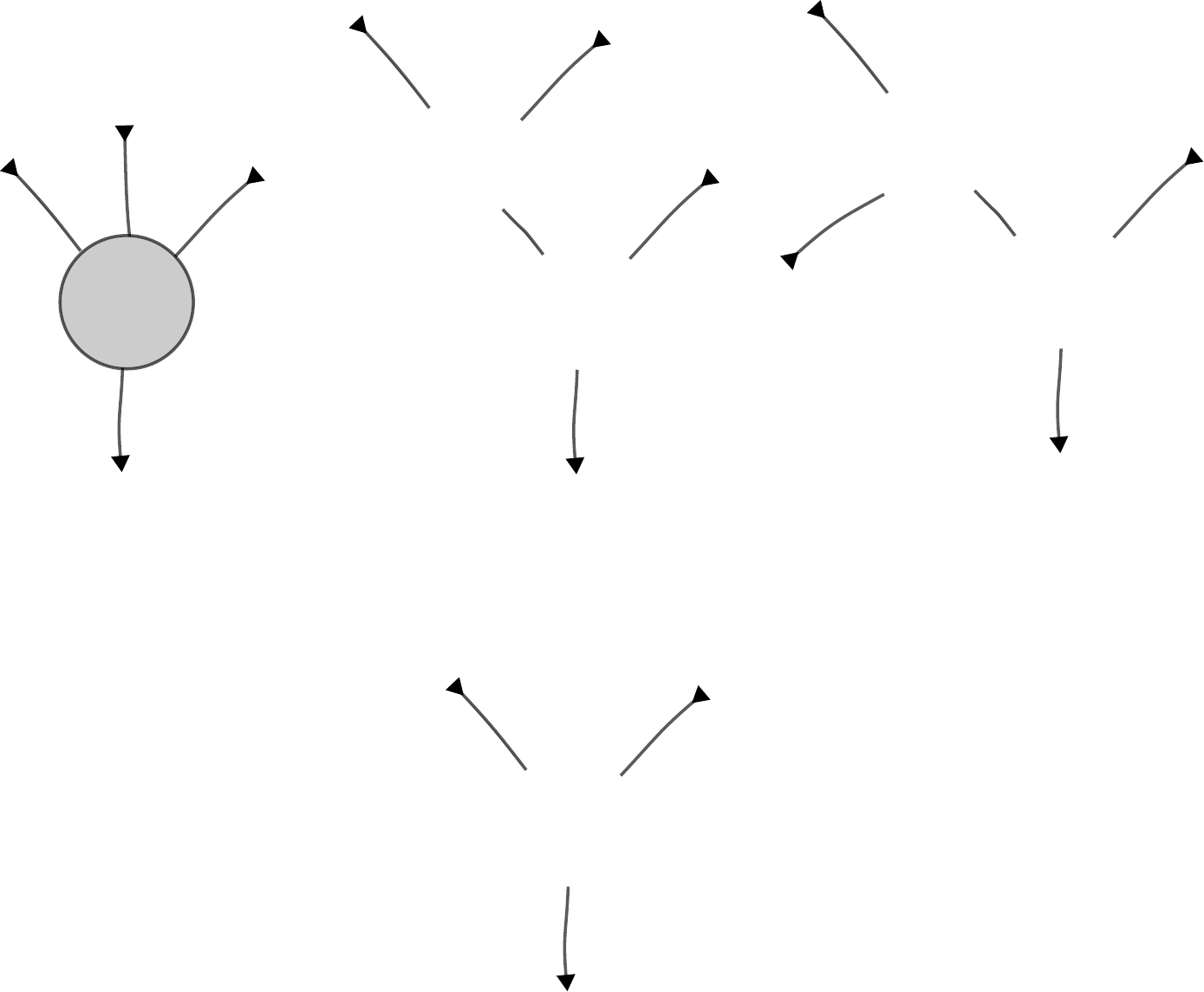}}%
    \put(0.09423893,0.72855492){\color[rgb]{0,0,0}\transparent{0.64502198}\makebox(0,0)[lt]{\lineheight{1.25}\smash{\begin{tabular}[t]{l}x\end{tabular}}}}%
    \put(0.51125205,0.80400366){\color[rgb]{0,0,0}\transparent{0.64502198}\makebox(0,0)[lt]{\lineheight{1.25}\smash{\begin{tabular}[t]{l}x\end{tabular}}}}%
    \put(0.63313438,0.59039541){\color[rgb]{0,0,0}\transparent{0.64502198}\makebox(0,0)[lt]{\lineheight{1.25}\smash{\begin{tabular}[t]{l}x\end{tabular}}}}%
    \put(0,0){\includegraphics[width=\unitlength,page=2]{forgetful.pdf}}%
  \end{picture}%
\endgroup%
}
    \caption{Examples of the forgetful map}
    \label{forgetful}
\end{figure}

In the latter case, the forgetful map collapses a constant disk $S_v$
at the intersection point with $S_v$ adjacent to exactly three edges of $T$. 
Configurations of the latter type come in pairs, with two configurations  $u_+, u_-: C \to X$
related by switching the order of the incoming edges $T_e$ at $S_v$.
The intersection number of the corresponding unstable manifold $\Sigma^s_{x_j}$
with the edge $u(T_e)$ of these configurations is opposite.  These configurations  cancel. 
\end{proof}

\begin{example} \label{cliffleg6}
  We partially compute the differential of the Clifford Legendrian $\Lambda_{\Cliff}$. 
  The projections of the disks $u_1,\ldots, u_n$
  contributing to the differential are of the form \eqref{filldisks} and
  permutations.  The disk potential is 
  \[ W(\ti{y}_1,\ldots, \ti{y}_n) = \ti{y}_1 + \ldots + \ti{y}_n\]
  subject to the relation $\ti{y}_1 \ldots \ti{y}_n = 1$ in coordinates 
  $y_1,\ldots, y_n$ on $\Rep(\Pi)$.
    The tori $\Pi,\Lambda$ may be identified with the lattice quotients
\[ \Pi = \Delta/ \left( \Delta \cap \{ (\lambda_1,\ldots, \lambda_n) \in (\Z/n)^n, 
\lambda_i - \lambda_j \in \Z \ \forall i,j \} \right) \] 
where $\Delta$ is the subgroup given by 
\[ \Delta = \{ (\lambda_1,\ldots, \lambda_n) \in \R^n, 
\lambda_1 + \ldots + \lambda_n  = 0 \} \] 
and 
\[ \Lambda = \Delta/ \Delta_0 , \quad \Delta_0 := \{ \lambda \in \Z^n, \lambda_1 + \ldots + \lambda_n = 0 \} .\]
  The boundaries of the disks $u_1,\ldots, u_n$ give paths of the form 
\[ \partial u_1,\ldots \partial u_n: [0,1] \to \Lambda, \quad \begin{array}{l} t \mapsto t{ ((n-1)/n, -1/n,\ldots, -1/n) } \\ \ldots \\
t \mapsto t{(-1/n,\ldots, -1/n,(n-1)/n)}  \end{array}
\ \text{mod} \ \Delta_0 .\]
Let the capping path $\hat{\gamma}$ be the reverse of the first path $\partial u_1$ above.  The capped-off 
boundaries $ \partial u_1^{-1} \# \partial u_i$ are then homotopic to the loops given by 
\[  \partial u_1^{-1} \# \partial u_i: \ S^1 \to \Lambda , \quad t \mapsto 1 + t{ (-1,1,0,\ldots,0)} + \ldots + t{(-1,0,\ldots,1)}  .\]
The non-trivial paths in this expression give a basis $\mu_1,\ldots, \mu_{n-1}$ for the homology $H_1(\Lambda)$.  
Let 
\[ \aa \in \RR(\Lambda) , \quad \deg_\R(\aa)= 2n \cdot\frac{1}{n} -1 =   1 \]%
denote the Reeb chord of length $1/n$  that is an index-zero  critical point of $f_{\white}$.   Let 
\[ \cc_i \in CE(\Lambda) , \quad \deg_\R(\cc_i) = 0 \]
denote the critical point of $f_\circ$ representing the Morse cycle
corresponding to $\mu_i$, for $i = 1,\ldots, n-1$.  Denote the stable and unstable manifolds 
\[ \Sigma_i^s , \Sigma_i^u \subset \Lambda, \quad \dim(\Sigma_i^s) = \dim(\Lambda) - 1, \quad 
\dim(\Sigma_i^u) = 1 \]
using the standard Morse function and metric. Each $\Sigma_i^s$ is the product
of factors in $\Lambda$ except the $i$-th factor. The $i$-th loop in $\Lambda$ has intersection number with the codimension one submanifold
$\Sigma_i^s$ equal to one, and none of  the other submanifolds $W_j^s, j \neq i$.  As in Theorem \ref{diffone} below, the repeating inputs Theorem
\ref{repeating} gives 
  \begin{equation} \label{deltaaa2}
  \delta(\aa) = 
  \left( \pm 1 + \sum_{i=1}^n \pm y_i, \right) + \ldots \in CE(\Lambda),
  \quad y_i   = {[\mu_i]} \exp(\cc_i) 
\end{equation}
where the additional 
contributions $\ldots$ arise from configurations lifting some  disks
in the base that are not Maslov index two; these could be constant disks lifting to 
punctured disks in a fiber, or non-constant disks of higher Maslov index with higher-codimension constraints at the leaves.  See Dimitriglou-Rizell-Golovko \cite[Section 6]{dr:bs} for computations using Legendrian isotopy 
and Morse flow trees.  We do not discuss here any computation of linearized or bilinearized contact homology groups; it's natural to wonder whether there is a relationship between the contact homology groups and the homology of the fillings as in Ekholm \cite{ekholm:rational}.
\end{example}

\label{inv}  We also wish to arrange that many of the zero-area terms in the differential may be arranged to vanish after passing to the abelianization.    Given a treed disk $C$ of some type $\Gamma$ containing an surface component $S_v$ and 
incoming edges $T_{e_1}, \ldots, T_{e_k}$ incident to $S_v$, consider the homeomorphism
from $T$ to itself given by re-ordering 
$T_{e_1}, \ldots, T_{e_d} $.  This operation induces an isomorphism of the one-dimensional part of the universal curve.
\[ \sigma_{\Gamma}: \ {\cT}_\Gamma \to {\cT}_{\Gamma} \]
defined in \cite[Section 4.3]{BCSW1}.
Let $\sigma_{\Gamma}$ be the identity 
on $\cS_\Gamma$.

 \begin{definition} A multivalued perturbation system $\ul{P} = (P_\Gamma)$ will be called {\em invariant}
if for each pair of types $\Gamma',\Gamma$ as above,  $P_\Gamma$ is pulled back from $P_{\Gamma'}$ under $\sigma_{\Gamma,\Gamma'}$.
\end{definition}

\begin{lemma}  The conclusion of \cite[Theorem \ref{I-twolevels}]{BCSW1} holds, that is, coherent, regular perturbations exist with the added condition that the perturbations are invariant.   
\end{lemma}

The proof is a repeat of the previous arguments with the symmetry condition imposed, and omitted.

We introduce notation for the coefficients of the differential. In \cite[Section 3.3]{BCSW1} we defined domain type $\Gamma$ and map type $\bGamma$ for treed maps. 
 Let 
$\delta_\bGamma(\gamma)$ is the contribution to $\delta$ from maps with domain type $\Gamma$ so that 
\[ \delta(\gamma) = \sum_\bGamma \delta_{\bGamma}(\gamma) .\] 
For $\gamma_1,\ldots, \gamma_d$ 
let 
\[ \delta_d(\gamma_0;\gamma_1,\ldots, \gamma_d) \in \mhat{G}(\Lambda) \]
denote the coefficient of $\gamma_1 \ldots \gamma_d$ in $\delta_d(\gamma_0).$ For any permutation $\sigma$
let  
\[ (-1)^{\deg_{\Z_2} (\sigma,\gamma_1,\ldots,\gamma_d)} \in \{ +1, -1 \} \] 
denote the Koszul sign of the expression 
$\gamma_{\sigma(1)} \ldots \gamma_{\sigma(d)}$ with respect to $\gamma_1 \ldots \gamma_d,$
that is the product of the signs $ (-1)^{ \deg_{\Z_2}(\gamma_i) \deg_{\Z_2} (\gamma_{i+1})} $
in a decomposition of $\sigma$ into transpositions. 

\begin{definition}  The differential $\delta$ is {\em classically commutative} for any $\gamma$ a degree two generator in the classical sector $CE_\black(\Lambda)$,  For any type $\Gamma $ any automorphism $\sigma \in \Aut(\Gamma)$ and any classical generators $\gamma_1,\ldots, \gamma_d$
\[  \delta_\bGamma(\gamma;\gamma_{\sigma(1)}, \ldots, \gamma_{\sigma(d)})  =(-1)^{\deg_{\Z_2} (\sigma,\gamma_1,\ldots,\gamma_d)} 
\delta_\bGamma(\gamma;\gamma_1,\ldots,\gamma_d) .\]
This ends the Definition.
\end{definition}

\begin{proposition}  \label{abelian} Let  $(Z,\Lambda)$ be a tame pair, $\ul{P} = (P_\Gamma)$ a coherent, regular system of 
invariant perturbations.  Then $\delta$ is classically commutative.
\end{proposition}

\begin{proof} Given a type $\Gamma$,  any tree automorphism $\sigma \in \Aut(\Gamma)$ induces an automorphism 
\[ \cT_\Gamma(\sigma): 
\ol{\cT}_\Gamma \to \ol{\cT}_\Gamma \]
of the universal tree $\ol{\cT}_\Gamma \to \ol{\cM}_\Gamma$.   By assumption
$\sigma$ then induces a bijection between zero-area moduli spaces\footnote{ We remark that  this action cannot in general extend to the positive-area moduli spaces, because in that case
the matching conditions at the intersection of the tree and surface parts is not preserved.  Such an action, 
if it existed, would imply that Lagrangian Floer homology is always graded commutative.  However, by mirror symmetry it is expected that there are Lagrangians corresponding to higher rank bundles, and these typically have
non-commutative endomorphism algebras.}
\[  \M(\Lambda, \gamma_0, \gamma_{\sigma(1)},\ldots, 
\gamma_{\sigma(d)} ) \to 
\M(\Lambda, \gamma_0, \gamma_1,\ldots, \gamma_d ) .\]
For $\Gamma$ so that each vertex has valence three, the induced action of $\sigma_{i(i+1)}$ 
on $\M_\Gamma$ is trivial since it preserves the edge lengths.
The orientation $\M_\Gamma(\Lambda)$ permutes the determinant lines associated
to the leaves and root edge, and so changes the
orientation by the claimed Koszul sign.
\end{proof}

\begin{definition} \label{def:ab} Denote by $S_\ell$ the symmetric group on $\ell$ letters, and 
\[ CE^{\ab}(\Lambda)   = \bigoplus_{\ell \ge 0 } CE_{\ell } (\Lambda)/ S_\ell \] 
the space obtained by identifying words equal up to ordering with Koszul sign.  The differential 
$\delta$ induces an {\em abelianized differential} denoted 
 \[ \delta^{\ab}: CE^{\ab}(\Lambda) \to CE^{\ab}(\Lambda) .\] 
 Let 
\[ \delta^{\ab,0}:  CE^{\ab}_0(\Lambda) \to CE^{\ab}_0(\Lambda) \]
denote
the composition of $\delta^{\ab}$ with the projection on the space of words
in the degree zero generators (that is, projecting out words whose total degree is
zero but whose individual letters may have non-zero degree.)  
\end{definition}

\begin{lemma} \label{lem:noconst} Let $(Z,\Lambda)$ be a tame pair.  For invariant perturbations $\ul{P}$, configurations $u: C \to X$  with a leaf $T_{e_i}$ labelled by a degree one generator $\cc_i$ attached to a constant disk $S_v$ do not contribute to the abelianized differential $\delta^{\ab}: CE^{\ab}(\Lambda) \to CE^{\ab}(\Lambda).$
\end{lemma}

\begin{proof}  Any such configuration produces another configuration $u': C' \to X$ of type $\Gamma'$, with opposite sign, by changing the order of leaves at the $S_v$.  Such contributions cancel by the invariance assumption.
\end{proof}

\begin{corollary} \label{classicalbound}  For $\gamma \in \cI_\black(\Lambda)$ of degree two, the projection of $\delta(\gamma)$ on 
$CE^{\ab}(\Lambda)$ (identifying words equal up to permutation) is the classical boundary 
$\delta_{\on{Morse}}(\gamma) = \delta(\gamma)$ plus terms involving words with at least one 
Morse-degree-zero generator.
\end{corollary}

\begin{proof} Since the moduli space of trees with $d$ leaves and $1$ root has dimension $d-2$,
the constrained moduli space can be rigid only if the output 
is of the form
$\gamma_1 \ldots \gamma_d \in \cI(\Lambda)^d$ for some $\gamma_1,\ldots, \gamma_d$ of total Morse degree $d$. 
If all of $\gamma_1,\ldots,\gamma_d$ are not Morse degree $1$, then at least one has Morse degree zero, since 
the average Morse degree is one.
For each such output arising from a tree $\Gamma$, choose edges $e_i, e_{i+1}$
adjacent to the same vertex.  By the previous Proposition \ref{abelian}, the contribution of maps 
with domain type $\Gamma$ is skew-symmetric under the transposition $\sigma_i, \sigma_{i+1}$.  
So the contribution of such maps disappears after abelianization.
\end{proof}

\begin{example} \label{coveringex}
We continue the discussion of the differential for the Clifford Legendrian in Example \ref{cliffleg6}.
For degree reasons, the output of $\delta^{\ab,0}$ consists of a linear combination of 
words $\ul{\gamma}_-$ in the degree one generators $\cc_1,\ldots, \cc_{n-1}$.
The coefficients of these are determined by Theorem \ref{repeating} to be inverse factorials.
Hence
  \begin{equation} \label{deltacc}
  \delta^{\ab,0}(\aa) = 
  \left( \pm 1 + \sum_{i=1}^n \pm y_i, \quad y_i
  = {[\mu_i]} \exp(\cc_i) \right) \in CE^{\ab}(\Lambda)
\end{equation}
in coordinates $y_1,\ldots, y_n $ on $\Rep(\Lambda)$.

In the case of the Clifford Legendrian $\Lambda \cong T^2$
in $Z = S^5$, the disks in Example \ref{cliffleg4} lifting the Maslov
index two disks in $Y = \CP^2$ give rise to the leading order terms in the
differential as in 
\eqref{deltaaa2}
\[ \delta^{\ab,0}(\aa) =
 1 \pm y_1 \pm y_2    . \] 
On the other hand, one may also take $Z$ to be the unit anti-canonical bundle.    The projection from $\Lambda$ to $\Pi$ is an isomorphism obtains using the capping path  corresponding to the monomial $y_1y_2$ 
\[ \delta^{\ab,0}(\aa) = 1 \pm y_1^2 y_2 \pm y_1 y_2^2 + \ldots .\]
corresponding to the shifted disk potential for the Clifford torus.  This ends the Example. 
 \end{example}

 \label{toricex3} We compute the leading order term in the abelianized differential for connected horizontal lifts of monotone tori. 
Let $Z$ be a negative bundle over
  a monotone symplectic manifold $Y$  and $\Lambda$ a connected lift of a monotone Lagrangian torus  
  $\Pi$.   The natural map $\Lambda \to \Pi$ induces a map
  $\pi_1(\Lambda) \to \pi_1(\Pi)$.  By pull-back, the map on fundamental groups induces a map on representation varieties  
  \[\Rep(p): \Rep(\Pi) \to \Rep(\Lambda) .\]
If $W: \Rep(\Pi) \to \C$ is a function that 
is the pull-back of a function $W'$ on $\Rep(\Lambda)$
then we write
\[ W' = \Rep(p)_* W: \Rep(\Lambda) \to \C .\] 
Let $W_\Pi$ be the disk potential of $\Pi$ and choose $v\in H_1(\Pi)$ such that $x^v$ corresponds to a vertex in the Newton polytope of $W$.  
Choose a capping path to be equal to the lift of a path representing $v$, and consider the Laurent polynomial 
\[  x^{-v}  W_\Pi \in \C[\Rep(\Pi)] .\]
This choice of capping path ensures that the Newton polytope of $ W_\Pi$ has $0$ as a vertex. Since the Newton polytope is convex, we can choose a basis of $H_1(\Lambda)$ to ensure that  the Newton polytope of $W_\Pi$ lies in the positive cone generated by the basis, see \cite[Figure \ref{III-fig:basisforlam}]{BCSW3} for the case when $H_1$ is two-dimensional, and so in the corresponding coordinates
$ x^{-v} W_\Pi $ is a polynomial. 

\begin{lemma}  If $u_1,u_2$ are Maslov index two disks in $Y$ bounding $\Pi$, then 
  the boundary difference $[\partial u_1] - [\partial u_2] \in 
  \pi_1(\Pi)$ lies in the image of $\pi_1(\Lambda)$.
\end{lemma} 

\begin{proof} Let $u_1,u_2$ be as in the statement of the Lemma.
  The monodromy of $Z$ around $[\partial u_1],[\partial u_2]$ is equal, by the monotonicity 
  relation, hence $[\partial u_1] - [\partial u_2] $ lifts to an element in 
  $\pi_1(\Lambda)$.
\end{proof}

\begin{corollary}   Let $\Lambda \to \Pi$ be a Legendrian covering a monotone Lagrangian $\Pi$
as above.  The function  $x^{-v} W_\Pi$ descends to a function on $\Rep(\Lambda)$ denoted 
\begin{equation} \label{Wlam}
W_\Lambda := \Rep(p)_* (x^{-v} W_\Pi ) \in \C[\Rep(\Lambda)].  \end{equation}
We call $W_\Lambda$ the augmentation polynomial for the Legendrian $\Lambda.$

\end{corollary}

The potential is related to the leading order term in the differential
as follows.  Let 
\[ \aa \in \cI_\white(\Lambda), \quad \deg_\R(\aa) = 1 \] 
denote the degree one generator representing a point in the space of Reeb chords $\cR(\Lambda)$ with minimal action $2 / \tau$, where $\tau$ is the monotonicity constant for $\Pi$.

\begin{theorem} \label{diffone} Let $(Z,\Lambda)$ be a tame pair as above.  The differential of the degree one generator $\aa
  \in CE_1(\Lambda)$ is
  \begin{equation} \label{diffreeb} \delta^{\ab,0}(\aa) = 
  W_\Lambda( [\mu_1] \exp(\cc_1),\ldots, [\mu_k] \exp(\cc_k)). \end{equation} 
  \end{theorem}

\begin{proof}  We will show that the disks contributing to the abelianized, 
  words-in-degree-zero part of the differential arise from lifts of Maslov index two disks with a single puncture in the lift.   Since the incoming angle from $\aa$ is minimal, \cite[Lemma \ref{I-anglechange}]{BCSW1} 
  implies that any contribution to $\delta(\aa)$ has a single puncture.  The projection to $Y$ must be  a disk bounding $\Pi$ 
  with Maslov index two by monotonicity.     Conversely, by Theorem  \ref{liftthm} any Maslov two disk $u_Y: S \to Y$ bounding $\Pi$ lifts to a punctured disk in $\R \times Z$ bounding $\R \times \Lambda$ with incoming punctured labelled by $\aa \in \cI_\white(\Lambda)$.   The 
  contribution of each disk is now governed by Theorem \ref{repeating}
  which gives an appearance of coefficient $\exp( \lan \cc, [\partial u] \ran)$
  from the intersection with the divisor cycles.
    \end{proof} 
  
  \begin{example} For the standard
  lift $\Lambda$ of the torus $\Pi$ in $\P^1 \times \P^1$ (also arising as conormal for the
  unknot in \cite{eens:knot}) there is a single Reeb chord $\aa$ of index one with
\[ \delta^{\ab,0}(\aa) = 
\left( 1 \pm y_1 \pm {y_2} \pm {y_1}
{y_2}  \right) \]
which is the knot contact differential for the unknot
consider in, for example, Aganagic-Ekholm-Ng-Vafa \cite[p. 39]{aganagic}.
\end{example}

\begin{example} \label{hopf3} 
We partially compute the differential for the
  generalized Hopf Legendrian $\Lambda \subset S^{2n-1} $ in \eqref{hopf} and Example \ref{hopf2}. Let
\[ \aa_{11}, \aa_{22} \in \cI(\Lambda), \quad \deg_\R(\aa_{11}) = \deg_\R(\aa_{22}) = 1  \] 
denote the point generators in the components of ${\cR}(\Lambda)$
connecting $\Lambda_1$ resp. $\Lambda_2$ with action $1/n$.
Introduce coordinates by the projection given by 
\[ \Lambda \to T^{n-1}, \quad (z_1,\ldots, z_n) \mapsto (z_1,\ldots, z_{n-1}) .\]
Denote the standard Morse function  on $\cR(\Lambda)$ 
(the sum of the height functions on each torus component) by 
\[ f_\white: \cR(\Lambda) \to \R .\]
This is the function that on each component of $\cR(\Lambda)$ diffeomorphic to $\Lambda \cong T^{n-1}$ by evaluation at the starting point of the Reeb
chord is given by the sum of the heights on each circle component. 
Denote by 
\[ \aa_{12}, \aa_{21} \in \cI_\white(\Lambda) , \quad \deg_\R(\aa_{12}) = \deg_\R(\aa_{21}) = n\cdot (1/n) -1 = 0 \] 
degree-zero generators in the component of ${\cR}(\Lambda)$ connecting $\Lambda_1$ to $\Lambda_2$ or vice versa with action $\pi/n$.  Denote by 
\[ \cc_{12,k}, \cc_{21,k} \in \cI_\white(\Lambda), \quad 
\deg_\R(\cc_{12,k}) = \deg_\R(\cc_{21,k}) = n\cdot(1/n) + 1 -1  = 1 \] 
for $k = 1,\ldots, {n-1} $ 
the generators corresponding to
Morse-index one critical points for the standard Morse function on the $k$-th
factor in $\Lambda \cong T^{n-1}$

 \begin{lemma} The image of the degree one generator $\cc_{12,k} \in \cI_\white(\Lambda)$
 under the abelianized differential is given by 
  \begin{equation} \label{deltacc2} \delta^{\ab}(\cc_{12,k}) = (1 - y_{k,1}  y_{k,2}^{-1} ) \aa_{12}, \quad
  \delta^{\ab}(\cc_{21,k}) = (1 - y_{k,2}  y_{k,1}^{-1} ) \aa_{21} ; \end{equation}
\begin{equation} \label{deltaaa21}
\delta^{\ab}(\aa_{11}) = 
\left( 1 - y_{1,1} - y_{2,1} \right)  + \aa_{12} \otimes \aa_{21} + \ldots \end{equation}
where the additional terms disappear after  projection onto words in the degree zero generators; 
\[ \delta^{\ab}(\aa_{22}) = 
\left(  1 - y_{2,1} -  y_{2,2} \right)  + \aa_{21} \otimes \aa_{12} + \ldots 
.\]
\end{lemma} 

\begin{proof} First we consider the minimal length generators.  Any configuration $u: C \to \R \times Z$ contributing to $ \delta(\cc_{12,k}) $ has no disk components, since the angle of $\cc_{12,k}$ is minimal and \cite[Lemma \ref{I-anglechange}]{BCSW1} implies
that the angles at the outgoing punctures would be less than the angle at the incoming puncture.  
Thus $u$ must be a Morse trajectory. There are two  Morse trajectories of $f_\circ$ in $\RR(\Lambda)$ connecting $\cc_{12,k}$ to $\aa_{12}$, corresponding to the two different ways of 
descending along the $k$-th circle factor in $\Lambda \cong T^{n-1}$
from the maximum to the minimum.    To calculate the contribution of these 
trajectories,  we must associate to each a homology class
in $H_1(\Lambda)$.  We may assume that the capping paths are chosen so that 
the homology class of the first gradient trajectory is trivial:  Choose the 
capping path for $\aa_{12}$ by concatenating the capping path for $\cc_{12,k}$
with the gradient trajectories images in $\Lambda_1$ and $\Lambda_2$. 
The homology class of the second gradient trajectory then 
differs from the first by the difference in homology classes
$\mu_{k,1} - \mu_{k,2}$.   After including the contributions with 
outgoing edges limiting to the support of $\cc_i$ as in Lemma \ref{repeating}, 
we obtain \eqref{deltacc2}.

Next, we consider the Reeb chords which may be written as concatenations.
The differential of $\aa_{11}$ includes a term
\[ \delta(\aa_{11}) = \aa_{12} \otimes \aa_{21} + \ldots \]
which corresponds to a constant map to the base $\P^{n-1}$ but with
Reeb chord from $\Lambda_1$ to $\Lambda_1$ breaking into two pieces as
in Proposition \ref{prop:vertex}.
Let 
\[ \mu_{1,b},\ldots, \mu_{n-1,b} \in H_1(\Lambda_b) \]
be the standard basis for $H_1(\Lambda_b)$
and 
\[ y_{i,b} = [\mu_{i,b}] \exp(\cc_{i,b})  \in CE^{\ab}(\Lambda).\]  
Combining with the contributions from Example 
\ref{coveringex}, the differential of $\aa_{11}, \aa_{22}$ is given by 
\begin{equation} \label{deltaaa21_proof}
\delta^{\ab}(\aa_{11}) = 
\left( 1 - y_{1,1} - y_{2,1} \right)  + \aa_{12} \otimes \aa_{21} + \ldots \end{equation}
where the additional terms disappear after  projection onto words in the degree zero generators.
The claim for $\aa_{22}$ is similar.  This ends the Example.
\end{proof}
\end{example}

\section{Chain maps from cobordisms} 
\label{mcob}

In this section,  we use the moduli space of holomorphic maps to cobordisms defined in \cite[Section \ref{I-tocob}]{BCSW1} to associate to any cobordism $L$ with concave end $\Lambda_-$ and convex end
$\Lambda_+$ satisfying tameness
conditions, a map of differential algebras
\[ CE(\Lambda_+,\hat{G}(L)) \to CE(\Lambda_-,\hat{G}(L))  .\]
In particular, for tame fillings we obtain augmentations.\footnote{In Ekholm-Ng \cite{ekholmng:higher} and similar papers, 
the augmentation associated to a filling is phrased in terms of a
open Gromov-Witten invariants counting disks without outputs.  We make no claim about existence of such invariants, but rather consider the chain map as an invariant
up to chain homotopy.} 

\subsection{Cobordism maps}

Recall the following definitions of symplectic cobordism and Lagrangian cobordism as in \cite[Section 2]{BCSW1}.

\begin{definition}\label{sympcob2}
Fix two circle-fibered contact manifolds $(Z_{\pm}, Y_{\pm}, \alpha_{\pm})$.
A {\em symplectic cobordism} with convex end $(Z_{+}, Y_{+}, \alpha_{+})$ and concave end $(Z_{-}, Y_{-}, \alpha_{-})$ is a symplectic manifold with boundary $(\tilde{X},\partial \tilde{X}, \omega)$ whose boundary is equipped with a partition
\[ \partial \tilde{X}= \partial_{+} \tilde{X}\cup \partial_{-} \tilde{X} \] 
and diffeomorphisms 
\[\iota_\pm : Z_\pm \to \partial_\pm \tilde{X}  \]
so that
\begin{enumerate} 
\item the two forms satisfy the proportionality relation
\[ \iota_\pm^* \omega\mid_{\partial_{\pm} \tilde{X}}=  \lambda_\pm  \omega_\pm,
\quad \omega_\pm := \d \alpha_\pm  \]
for some positive constants $\lambda_\pm \in \R$, and 
\item  the 
isomorphism  between the normal bundle of $Z_\pm$
in $\widetilde{X}$ and the kernel of $D p: TZ_\pm \to TY_\pm$
induced by the symplectic form is orientation preserving resp. reversing for the convex resp. concave boundary. 
\end{enumerate}
Denote by $X$ the interior of $\widetilde{X}$, viewed as a manifold with cylindrical ends, and, abusing terminology, 
call $X$ the cobordism. 
\end{definition}

\begin{definition}
For a symplectic cobordism $X$ with convex end $Z_{+}$ and concave end $Z_{-}$ and two Legendrian submanifolds $\Lambda_{\pm}\subset Z_{\pm}$ so that $X$ is equipped with cylindrical ends, a {\em Lagrangian cobordism cylindrical near infinity} with convex end $\Lambda_{+}$ and concave end $\Lambda_{-}$ is a   Lagrangian submanifold $L\subset X$ such that
\[ L\cap ((\sigma_{\pm,0}, \sigma_{\pm,1})\times Z_\pm)= (\sigma_{\pm,0}, \sigma_{\pm,1}) \times \Lambda_\pm \]
for some intervals $(\sigma_{\pm,0}, \sigma_{\pm,1 }) \subset \R$.
\end{definition}

By the symplectic cutting construction, we can also compactify the cobordism $(X,L)$ into $(\ol{X}, \ol{L})$, where $\ol{X}$ is a closed symplectic manifold and $\ol{L}$ is a Lagrangian with certain singular points. The following tameness conditions are key to us to define cobordism maps between Chekanov-Eliashberg algebras. We refer to \cite[Section 3.4]{BCSW1} for more detailed discussions.

\begin{definition} \label{nonnegcob} \label{def:tamepair}   The pair $(X,L)$ with concave end $(Z_-,\Lambda_-)$ and convex end $(Z_+,\Lambda_+)$ is a {\em  tame cobordism pair} if and only if both ends $Z_\pm$ are tame, and the following conditions hold.

\begin{enumerate}[label={(\bfseries P\arabic*)}]
  \item \label{p1} {\rm (Rationality)} The symplectic classes $[\omega_{Y_\pm}] \in H^2(Y_\pm)$ and the class  $[\ol{\omega}] \in H^2(\ol{X})$ are rational classes, and the class $ [ \ol{\omega}_{\ol{X} - Y_-}] \in H^2(\ol{X} - Y_-)$ is an integral class.   Also, the Lagrangian $\ol{L}$ is rational.
\item \label{p2} {\rm (No-cap condition)} The logarithmic first Chern class of $\ol{X} - Y_-$ is a very positive
multiple of the symplectic class on the complement of $Y_-$ in the sense that there exists a constant $\lambda_- >0$ so that 
\[  c_1^{\on{log}}(\ol{X} - Y_-) = (1 + \lambda_-)[\ol{\omega} |_{\ol{X} - Y_-}] \in H^2(\ol{X} - Y_-) .\]
\item \label{p3} {\rm (Monotonicity for the concave end)}  If the concave end $\Lambda_-$ of the cobordism $L$ is non-empty then the relative Thom class of the divisor $[Y_-]^\dual \in H^2(\ol{X} - Y_+,\ol{L} - Y_+)$ is non-positive in the sense that there exists a constant $\lambda_+  \ge 0$ so that 
 \[  [Y_-]^\dual  =   -\lambda_+ [ \ol{\omega}|_{\ol{X} - Y_+}]  \in H^2(\ol{X} - Y_+,\ol{L} - Y_+) .\] 
\end{enumerate}
\end{definition}

Now let $L \subset X$ be a tame compact Lagrangian cobordism with concave end $\Lambda_- \subset Z_-$ and convex end $\Lambda_+ \subset Z_+$ equipped with a spin structure.  We assume that $\Lambda_\pm$ are equipped with Morse data as in Definition \ref{morsedatum2}, and in particular with Morse
functions 
\[ f_{\black,\pm}: \Lambda_\pm \to \R. \] 
To simplify notation 
define 
\[  H_2 := H_2(\ti{X},\ti{L} \cup {p}^{-1}(\Pi),\Z) .\]

Any punctured disk in $X$ bounding $L$ extends to a map from a manifold 
with corners $\ti{S}$ mapping to $\ti{X}$ and bounding $\ti{L} \cup {p}^{-1}(\Pi_\pm)$, by adding in the Reeb chords at infinity, and so defines a homology class in $H_2$. This follows from the fact that $p^{-1}(\Pi_\pm)$ is equal to the image of $\Lambda_\pm$ under Reeb flow.  The symplectic form $\ti{\om}$ vanishes on $ \ti{L} \cup {p}^{-1}(\Pi)$ and so 
induces a natural area map 
\[  A: H_2 \to \R .\] 
In our examples, $H_2(\ti{X})$ will often vanish and the long exact sequence
gives an identification
\[ H_2(\ti{X},\ti{L} \cup {p}^{-1}(\Pi),\Z) \cong H_1( \ti{L} \cup {p}^{-1}(\Pi),\Z) ) .\]
\vskip .1in \noindent
The {\em completed group ring} $\hat{G}(L)$, similar to $\mhat{G}(\Lambda)$ in (\ref{cglam}),
is 
\begin{equation} \label{cgl} 
\hat{G}(L) = \Set{ g: H_2 \to \C, \quad 
\forall A, \# \{ c  \in H_2 |  g(c) \neq 0, A(c) < A \} 
\ \on{finite} } .\end{equation}
It admits a natural filtration 
\begin{equation} \label{filt} \hat{G}(L) = \bigcup_A \hat{G}(L)_A, \quad \hat{G}(L)_A = \{ g | g(c) = 0, \forall A(c) < A \} . \end{equation}

\begin{definition} A Morse datum for $L$ is a Morse function
\[ f_L:L \to \R , \quad   f_L |_{\Pi_\pm} = f_{\black,\pm} \] 
extending the given Morse functions $f_{\black,\pm}: \Pi_\pm \to \R$.

\vskip .1in \noindent
The set of {\em relative} resp. {\em absolute generators} is 
\[ 
\cI(L,\partial L ) := \crit(f_L) , \quad 
\cI(L) := \cI(L,\partial L) \cup \cI(\Lambda_-) \cup \cI(\Lambda_+) .\]
The space of {\em absolute} resp. {\em relative chains} is 
\begin{eqnarray}  \label{cel}
CE(L) &:=& \sum_{x \in \cI(L)} \hat{G}(L) x \\
CE(L,\partial L) &:=& \sum_{x \in \cI(L,\partial L)} \hat{G}(L) x \end{eqnarray}
given by the formal sum of generators over the ring 
$ \hat{G}(L) $, with length filtrations going to infinity.
\end{definition}

Define maps between Chekanov-Eliashberg algebras by counts of holomorphic disks bounding
 the Lagrangian cobordism; we will show that this count defines a chain map if a certain Maurer-Cartan equation is satisfied.   Given a treed disk $C = S \cup T$, we orient the edges $T_e \subset T$ away from the root edge, and let $s$ denote a local coordinate on $T_e$ compatible with the metric structure.
 
 \begin{definition} For integers $d_-,d_+ \ge 0$
 let 
 \[ \M(L) = \Set{ u: C \to X | \begin{array}{l} \olp_J u |_S = 0
 \\ \partial_s u  =  - \grad(f_\star(u)) \end{array}
  , \quad u(\partial C) \subset L  }  \] 
 denote the moduli space of perturbed holomorphic treed disks  in $X$ bounding $L$,
 where $f_\star$ means either $f_L, f_{\black,\Lambda_\pm}$ or $f_{\white,\Lambda_\pm}$.
 Let $\M_{d_-,d_+}(L) \subset \M(L)$ denote the locus with 
\begin{center} \begin{tabular}{ll} 
 1 & root edge, with limit an element in $\cI(\Lambda_+)$\\ 
 $d_+$  & edges with limits in $\cI(L,\partial L)$ \\
 $d_-$ &  edges with limits in $\cI(\Lambda_-)$.  \end{tabular}
 \end{center}
See Figure \ref{obs5}.
\end{definition}

\begin{figure}[ht]
    \centering
    \scalebox{.5}{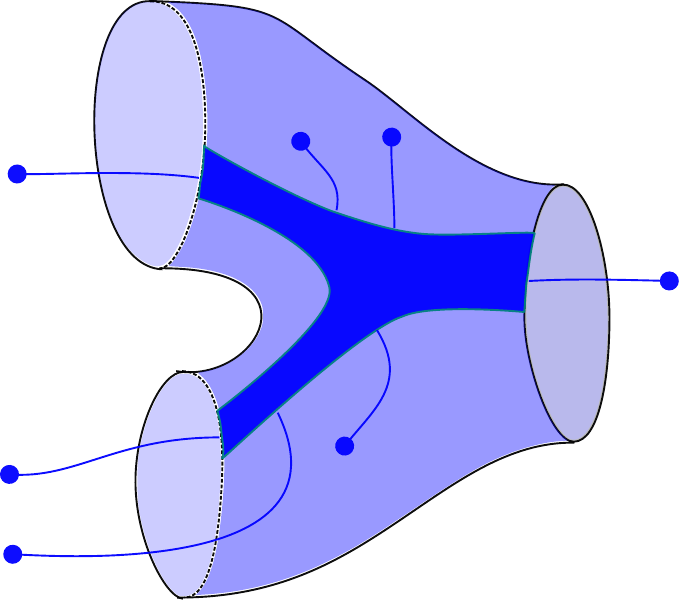}
    \caption{Treed disks defining the chain map}
    \label{obs5}
\end{figure}

By \cite[Section \ref{I-foundsec}]{BCSW1}, for generic systems of domain-dependent perturbations the components of the 
moduli spaces $ \M_{d_-,d_+}(L) $ of dimension at most one are 
transversally cut out, equipped with orientations, and compact for any given energy bound.

 \begin{definition}  Given a cobordism pair $(X,L)$ with concave end $(Z_-,\Lambda_-)$ and convex end $(Z_+,\Lambda_+)$
 and an element $b \in CE(L,\partial L)$, 
 the {\em cobordism map} 
\[ \varphi(L,b):  CE(\Lambda_+,\hat{G}(L)) \to CE(\Lambda_-, \hat{G}(L)) \]
is  defined counting holomorphic treed curves $u: S \to X$ bounding $L$ with arbitrarily 
many insertions of $b$, and exactly one incoming puncture
in each component, as follows: For a single incoming generator 
\[ \gamma = \sum_{x \in \cI(\Lambda_+) } \gamma(x)x   \] 
we have for 
\[x_0 \in \cI(\Lambda_+), \quad \ul{x}_- = (x_1,\ldots, x_{d_-}) \in \cI(\Lambda_-)^{d_-}, \quad \ul{x}_\black = (x_{\black,1},\ldots, x_{\black,d_+}) \in \cI(L,\partial L)^{d_+}\]
\[ (\varphi(L,b))(\gamma) := 
\sum_{ \ul{x}, u \in \M(L,x_0,\ul{x}_-,\ul{x}_\black)_0 }
(-1)^{\sum_{i=1}^{d_-} i |x_i| } \left( \prod  b(x_{\black,i}) \right) \gamma(x_0) x_1 \otimes \ldots \otimes x_{d_-} , \]
where $b(x)$ denotes the coefficient of $x$ in $b$.
See \eqref{heartsign}.  For words 
\[ w = \gamma_1 \ldots \gamma_d \]
define $(\varphi(L,b))(w)$ by
\[ (\varphi(L,b))(\gamma_1 \ldots \gamma_d ) 
= \varphi(L,b)(\gamma_1 ) \ldots \varphi(L,b) (\gamma_d) .\]
Note that such an assignment automatically makes $\varphi(L,b)$ an algebra map.
\end{definition}

In the next section, we discuss what properties of the bounding cochain guarantee the cobordism map is a chain map.   We first discuss the grading properties of the cobordism maps.  The cobordism map
\[ \varphi_L: CE(\Lambda_+) \to CE(\Lambda_-) \] 
defined by a Lagrangian cobordism is $\Z_2$-graded, by the standard computation from Fukaya categories. The cobordism map  $\varphi_L$ is $\R$-graded under a suitable trivial-Maslov-class condition.

\begin{corollary}  Let $(X,L)$ be a tame 
cobordism with concave end $(Z_-,\Lambda_-)$ and convex end $(Z_+,\Lambda_+)$.
Suppose that the trivialization of the canonical bundle of $\R \times Z_\pm$,  agreeing with the given trivialization over $\R \times \Lambda_\pm$ extends over the cobordism $X$ and $b$ has degree one.   Then the real grading on $CE(\Lambda_\pm)$ is preserved by  the cobordism map $\varphi(L,b): CE(\Lambda_+) \to CE(\Lambda_-)$.
\end{corollary}

\begin{proof} For strata $\M_\Gamma$ of maximal dimension, the dimension is 
\[ \dim \M_\Gamma = d - 2 =  \sum_{ e \in \mE_+(S)} 1 - \sum_{ e \in \mE_-(S)} 1 - 1 .\]
Rigidity implies that 
\begin{eqnarray*} 
& &\sum_{e \in \mE_+(S)} ( \on{ind}
(f_{\white,+}))(\gamma_e) + \tau \theta_e- 1 
-  \sum_{ e \in \mE_-(S)} (\on{ind}
(f_{\white,-}))(\gamma_e) + \tau \theta_e  - 1 \\
&=& \sum_{e \in \mE_+(S)}  (\on{ind}
(f_{\white,+}))(\gamma_e)  
-  \sum_{ e \in \mE_-(S)} (\on{ind}
(f_{\white,-}))(\gamma_e)   
+ (d-1) \\
&=& \sum_{e \in \mE_+(S)}  (\on{ind}
(f_{\white,+}))(\gamma_e) + \tau \theta_e  
-  \sum_{ e \in \mE_-(S)} (\on{ind}
(f_{\white,-}))(\gamma_e) + \tau \theta_e  \\
&=& \Ind(D_u) - \dim(\Lambda) - 1 \\
&=& 0 .\end{eqnarray*}
Thus the difference is real gradings between incoming and outgoing generators vanishes. 
\end{proof}

\begin{example} In our motivating examples, the canonical bundle is not typically trivial in the 
sense above. 
For example, take $Z_\pm \cong S^1$, let $  \Lambda_-$ be a disjoint union of two points and 
$\Lambda_+$ the emptyset.  Let $L \cong \R$ be a cobordism with concave end $\Lambda_-$ and convex end $\Lambda_+$,
that is, a path joining the two points of $\Lambda_-$.  One sees easily the trivialization of the canonical bundle on $\R \times Z_\pm$ (agreeing with the trivialization determined by the orientation on $\R \times \Lambda_\pm$).  Clearly, there is no trivialization
of the canonical bundle on $(\R \times Z_\pm , L)$ which extends the given trivializations on 
$(\R \times Z_\pm, \R \times \Lambda_\pm)$.
\end{example}

\begin{example} \label{hopf4}  We discuss the grading of the augmentation map associated to the exact filling of the Hopf Legendrian with a single component discussed earlier in \eqref{hopf2}.   Let $\Lambda_{\Hopf} \subset S^{2n-1}$ denote the Hopf Legendrian and $L_{(2)}$ the filling in \eqref{hopf2} containing punctured holomorphic disks 
asymptotic to Reeb chords with action $ 1/ n$.  The degrees
of the corresponding generator $\cc_{12} \in \cR(\Lambda), \cc_{21} \in \cR(\Lambda)$
are
\[ \deg_\R(\cc_{12}) = \deg_\R(\cc_{21}) = 0. \]
Depending on whether the path defining
the filling $L_{(2)}$ passes below or above the critical value in the Lefschetz fibration, 
the augmentation $\varphi(\cc_{12})$ resp. $\varphi(\cc_{21})$ has either $1$ or $n$ terms corresponding to the disks in \eqref{nor1}
or \eqref{nor12}.
\end{example}

\begin{example} \label{hopf4b}   We give an example of a filling, arising naturally from 
a union of paths of a symplectic fibration, which cannot preserve the grading in any natural way.
Suppose $\Lambda' \subset S^{2n-1}$ is the Legendrian obtained
by taking the disjoint union of $\Lambda$ with $e^{i \theta} \Lambda$, where $\theta$ is close, 
but not equal, to $1/6$, for simplicity with $n = 3$.   Consider the filling obtained by a path from $\Lambda$ to $\Lambda'$ as in Lemma \ref{upslem}. Let $\varphi$ denote the corresponding augmentation.
Since the filling is exact, the element $1 + y_{1,1} + y_{2,2}$ maps to $1 + y_1 + y_2$ under $\varphi$,  so we must have $\varphi(\cc_{12}) \varphi(\cc_{21}) \neq 0$.
On the other hand, the elements $\cc_{12}, \cc_{21}$ do not necessarily have zero grading.
There seems to be no way of fixing this issue:  Even if the angle $\varphi$ is chosen 
so that $\cc_{12}$ and $\cc_{21}$ have degree zero, we could add another copy of the Clifford
Legendrian at angle $\theta'$ between $\theta$ and $1/3$.   The differential 
of the Reeb generator $\cc_{21}$ connecting the second sheet to the first sheet passes
through the third sheet, so that 
\[ \delta(\cc_{21}) = \cc_{23} \otimes \cc_{31} .\]
For degree reasons, we must have 
\[ \deg(\cc_{23} \cc_{31}) = \deg(\cc_{23}) + \deg(\cc_{31}) = \deg(\cc_{21}) - 1 = 0 - 1 = -1.  \]
So not both of $\cc_{23}$ and $\cc_{31}$ can have degree zero. But we can easily 
find a filling, using Lemma \ref{upslem} again, whose associated augmentation 
$\varphi$ maps $\cc_{23}$ or $\cc_{31}$ to a non-zero element of $G(\varphi')$.
This shows that there is no grading for which all of the augmentations associated
to fillings of the type considered in Lemma \ref{upslem} are graded.
\end{example}

\begin{remark} \label{goodcases}
In good cases, one can define a  map associated to a cobordism over 
group ring on first homology, rather than relative second homology; this better matches
conventions in earlier papers in the field.  Suppose for example that $(X,L)$ is a filling of $(Z,\Lambda)$, and only holomorphic disks in $(X,L)$ with no punctures contribute to the definition 
of $\varphi(L,b).$  (This could be the case for grading reasons, for example, since
the Reeb chords all have positive $\R$-degree, so $\varphi(L,b)$ is non-vanishing only 
on the classical generators $\cI_\black(\Lambda)$).     Suppose furthermore, that $X$
is simply-connected, so that 
\[ H_2(X,L) \to H_1(L) \]
is a surjection whose kernel is equal to the image $H_2(X)$ in $H_2(X,L)$. From the monotonicity condition \cite[\ref{I-p2}]{BCSW1} of tame fillings, we have that for each $ {[\gamma]} \in H_1 (L) $ there is an upper bound on the area of rigid disks $u$ such that $[\partial u]  = [\gamma]$. Thus we can apply the usual Gromov compactness argument to show that $\varphi(L,b)$ can be defined using $H_1(L)$ coefficients. 
\end{remark}

\subsection{Bounding chains}

\label{bchain} \label{unobs}

In this section,  we  study under what conditions the cobordism maps from the previous section
are chain maps.  Recall from \cite[Theorem \ref{I-twolevels3}]{BCSW1} that 
 the boundary configurations in the one-dimensional components of the moduli space  of curves bounding the Lagrangian include buildings  with two levels and buildings  with a single level with an infinite length edge.    In order to obtain a chain map, contributions from the second type of map must vanish.

Following Fukaya-Oh-Ohta-Ono \cite{fooo:part1} we define a notion of Maurer-Cartan solutions as follows.  The Maurer-Cartan equation counts curves with inputs arising from chains
in the interior of the cobordism, and outputs either in the interior or on the boundary, as in Figure \ref{curvature}.   As such, it does not arise from an \ainfty algebra because of the additional breakings into levels.   We first define the combinatorial types of curves that occur. 

\begin{definition}   A treed building $u:C = S \cup T \to X$ bounding $L$ is {\em of obstruction type} if 
all of the edges $T_{e_i} \subset T, i= 1,\ldots, d$ have limits that are critical points  $x(e) \in \cI(L, \partial L)$ in the interior of $L$, rather than 
critical points on $\Lambda_\pm$ or $\cR(\Lambda_\pm)$.
\end{definition}

Thus, for obstruction type treed buildings   the root edge $T_{e_0}$ is either asymptotic to a critical point $x(e_0)$ in the interior of $L$ or in $\Lambda$ or $\cR(\Lambda)$.    We introduce notation for moduli spaces of holomorphic disks of obstruction type.  Denote by 
\[ \M_{d}(L,\partial L) = \bigcup_{\bGamma} \M_{\bGamma}(L, \partial L)  \]
the moduli space of treed buildings in $X$ bounding $L$ of obstruction types with arbitrary $d$ leaves labelled by $\cI(L,\partial L)$,  and one root edge labelled by $\cI(L,\partial L)$ as in Figure \ref{curvature}.  By definition, configurations $u: C \to X$ in $\M(L,\partial L)$ are treed disks in $X$ bounding $L$ not meeting  $Z_\pm$.   As a special kind of building, the moduli spaces of $\M_d(L,\partial L)$ of expected dimension at most one are compact and regular by Theorems \ref{I-twolevels} and \ref{I-compactness} in \cite{BCSW1}.

\begin{definition}\label{defmc}
Let $(X,L)$ be a tame cobordism. The {\em higher relative composition maps} 
\begin{equation} \label{mds} m_d: CE(L,\partial L)^{\otimes d} \to CE(L,\partial L) \end{equation}
are defined by 
\[ (x_d,\ldots, x_1) \mapsto \sum_{u \in
    \M_d(L,x_0,\ldots,x_d)_0} (-1)^{\heartsuit} \wt(u)x_0  \]
    where the weight $\wt(u)$ is defined by the expression in \eqref{wu}
    and the sign is from \eqref{heartsign}.
    Let 
\[ CE(L,\partial L)_{> 0} := \Set{  b = \sum_{x \in \cI(L,\partial L)} b(x) x | \forall x \in \cI(L), \  \exists A > 0, b(x) \in \hat{G}(L)_A  }  \] 
be the subset of critical points with positive part of the filtration 
\eqref{filt}.

\vskip .1in \noindent 
Given $b \in CE(L,\partial L)_{> 0}$ the {\em curvature} of $b$ is the element 
\[ m(b) := \sum_{d \ge 0} \frac{1}{d!} m_d(b,\ldots, b)  \in CE(L,\partial L) \]
obtained by all possible insertions of $b$ in $m_d$, well defined by definition of the 
completion in \eqref{cgl}.

\vskip .1in \noindent 
The chain $b$ is a {\em bounding chain} or a {\em Maurer-Cartan solution} if $m(b) = 0$. In particular, if $m_0(1)=0$ then the element $b=0$ is a bounding chain.

\vskip .1in \noindent 
The  {\em Maurer-Cartan space} of bounding chains is 
\[ MC(L) = \{ b \in CE(L, \partial L)_{>0} \ | \  m(b) = 0 \} .\]

\vskip .1in \noindent 
The cobordism is  $L$ {\em unobstructed} if $MC(L)$ is non-empty, that is, there exists a bounding chain.       This ends the Definition.
\end{definition}

\begin{figure}[ht]
    \centering
    
    \scalebox{.5}
{  
\begingroup%
  \makeatletter%
  \providecommand\color[2][]{%
    \errmessage{(Inkscape) Color is used for the text in Inkscape, but the package 'color.sty' is not loaded}%
    \renewcommand\color[2][]{}%
  }%
  \providecommand\transparent[1]{%
    \errmessage{(Inkscape) Transparency is used (non-zero) for the text in Inkscape, but the package 'transparent.sty' is not loaded}%
    \renewcommand\transparent[1]{}%
  }%
  \providecommand\rotatebox[2]{#2}%
  \newcommand*\fsize{\dimexpr\f@size pt\relax}%
  \newcommand*\lineheight[1]{\fontsize{\fsize}{#1\fsize}\selectfont}%
  \ifx\svgwidth\undefined%
    \setlength{\unitlength}{248.20522278bp}%
    \ifx\svgscale\undefined%
      \relax%
    \else%
      \setlength{\unitlength}{\unitlength * \real{\svgscale}}%
    \fi%
  \else%
    \setlength{\unitlength}{\svgwidth}%
  \fi%
  \global\let\svgwidth\undefined%
  \global\let\svgscale\undefined%
  \makeatother%
  \begin{picture}(1,1.21415594)%
    \lineheight{1}%
    \setlength\tabcolsep{0pt}%
    \put(0,0){\includegraphics[width=\unitlength,page=1]{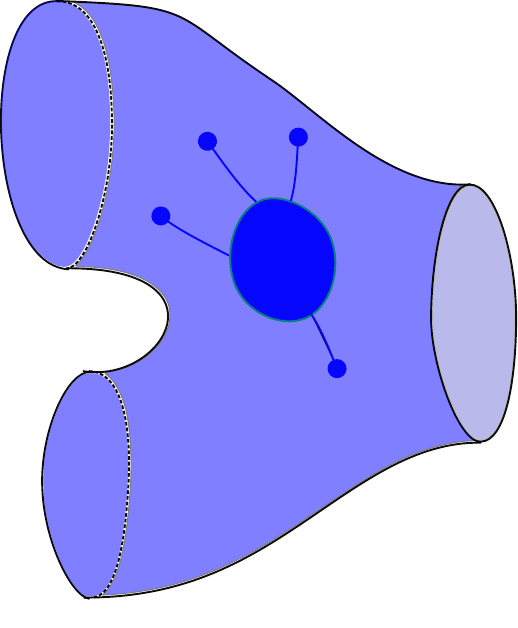}}%
    \put(0.11530615,0.00625242){\color[rgb]{0.2,0.2,0.2}\transparent{0.64502198}\makebox(0,0)[lt]{\lineheight{1.25}\smash{\begin{tabular}[t]{l}$Z_-$\end{tabular}}}}%
    \put(0.88779785,0.31754711){\color[rgb]{0.2,0.2,0.2}\transparent{0.64502198}\makebox(0,0)[lt]{\lineheight{1.25}\smash{\begin{tabular}[t]{l}$Z_+$\end{tabular}}}}%
    \put(0.38630046,0.31345759){\color[rgb]{0.2,0.2,0.2}\transparent{0.64502198}\makebox(0,0)[lt]{\lineheight{1.25}\smash{\begin{tabular}[t]{l}$X$\end{tabular}}}}%
  \end{picture}%
\endgroup%
}
    \caption{Holomorphic curves defining the curvature}
    \label{curvature}
\end{figure}
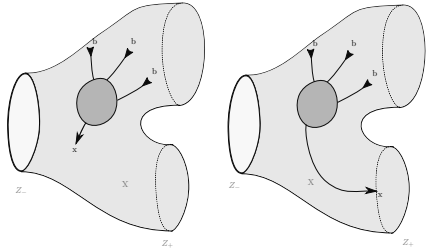

\begin{lemma} \label{lem:exactunobs} Suppose $L \subset X$ is exact.   Then the element $b = 0$ is a bounding chain, so $L$ is unobstructed.
\end{lemma}

\begin{proof} By the stability condition, any such disk 
$u$ contributing to $m(b)$  has at least one non-constant component $S_v \subset C$, corresponding to a leaf of the type $\Gamma$ of $u$.   This is impossible, since 
exactness implies that any holomorphic disk is constant. 
\end{proof}

\begin{theorem} \label{indep} Let $(X,L)$ be a tame pair and 
suppose $b \in MC(L)$ is a Maurer-Cartan solution.  Then the map ${\varphi}(L,b)$  is a map of differential algebras (chain map and algebra homomorphism) and so induces  maps in contact homology
\[ H(\varphi(L,b)) : HE(\Lambda_+) \to HE(\Lambda_-), 
\quad H(\varphi_\white(L,b)) : HE_\white(\Lambda_+) \to HE_\white(\Lambda_-).\]
\end{theorem}

\begin{proof}   The map $\varphi(L,b)$ is an algebra map by definition, since the product is given by concatenation.  To prove it is a chain map under the hypotheses given in the Theorem, we classify the types of breakings in the true boundary configurations.  Any true boundary configuration $u:C \to X$ of the one-dimensional component of the moduli space $\M(L)$ of treed disks bounding $L$ is a configuration with some collection of trajectories $T_{e_1},\ldots, T_{e_k}$ of infinite length.  The broken trajectories $T_{e_1},\ldots, T_{e_k}$ may either connect two different levels $C_i,C_{i+1}$ of $C$, or two disks $S_{v_1}, S_{v_2}$ within the same level  $C_i$ as shown in Figure \ref{obs13}.  In case there are two different levels, one of the levels maps to either
the incoming or outgoing cylindrical end. 

We analyze the count of boundary configurations in each case.
\begin{figure}[ht]
    \centering
    \scalebox{.3}{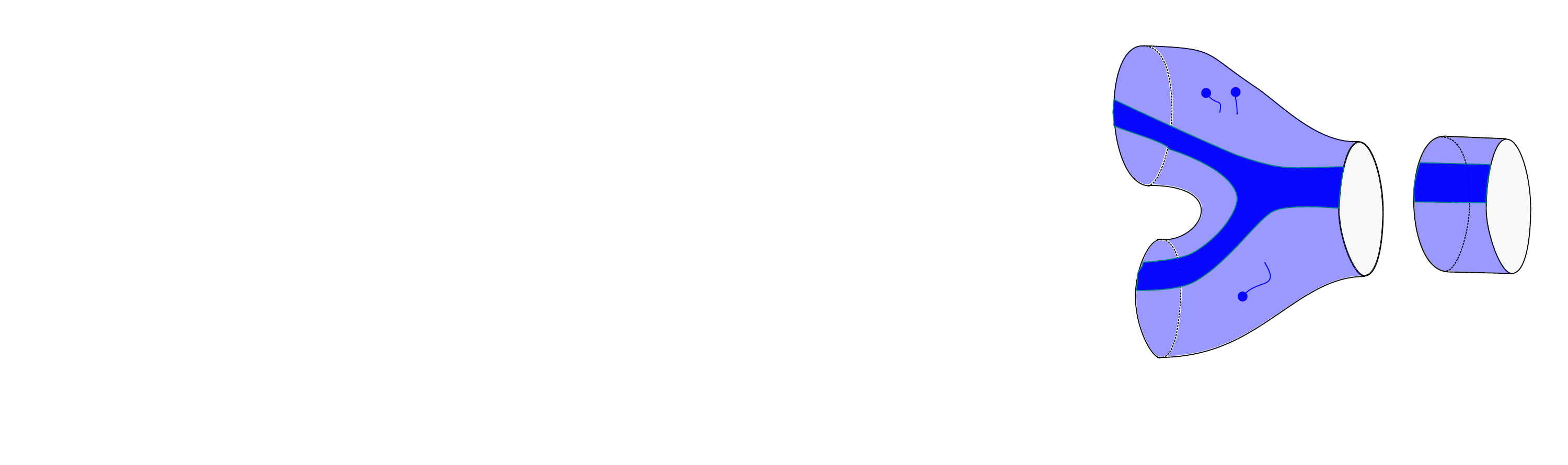}
    \caption{Boundaries of moduli spaces of obstruction type}
    \label{obs13}
\end{figure}
First, consider configurations $u: C \to X$  with two levels $C_1,C_2$ in  $X$ and $\R \times Z_+$, with  the level in $X$ having two connected components $C_{1,1}, C_{1,2}$, one $C_{1,2}$ with no incoming puncture  and the other $C_{1,2}$ with one incoming  puncture.  Suppose $C_{1,1}$ is joined to $C_2$ by an edge passing through a generator $x \in \cI(\Lambda_+)$. The count of configurations $C_{1,2}$ is the coefficient  of $x$ in $m(b)$ which vanishes by assumption. 

Second, consider configurations $u: C \to X$ with two levels $C_1,C_2$ either in $\R \times Z_-$ and $X$ or in $X$ and $\R \times Z_+$, with exactly one incoming puncture on each component.  These configurations contribute to the expressions
\[ \delta_+ \circ \varphi(L,b), \quad \varphi(L,b) \circ \delta_- \]
respectively. 

The chain property follows the fact 
that the boundary of the one-dimensional moduli space of configurations has vanishing signed count; with the sign computation being that in Karlsson 
\cite{karlsson:coherent}.
\end{proof}

\subsection{Independence of the cobordism maps from choices}

In this section,  we show that the chain maps associated to cobordisms
are independent of all choices, up to chain homotopy and a suitable change in the 
bounding chain, as in the theory of Fukaya-Oh-Ohta-Ono \cite{fooo:part1, fooo:part2}.  To state the theorem, we recall from Charest-Woodward \cite{flips} the notion of treed quilted disk. 

\vskip .1in \noindent 
\begin{definition}  Let $k \ge 1$.   A {\em quilted disk}  is a disk $S$  with $k$ seams in the complex plane $\C$, equipped with $k$ {\em seams} $Q= (Q_1,\ldots, Q_k)$ which are circles in $S$ tangent to the
boundary  $\partial S$.   Up to isomorphism we may assume
that $S$ is the unit disk centered at $0$.

\vskip .1in \noindent 
A {\em nodal pre-stable quilted disk}  is a  connected union $(S,Q)$ of disks, spheres, and quilted disks $S_v \subset S$ corresponding to vertices $v \in \Ver(\Gamma)$
of some tree $\Gamma$, together with  a collection of boundary markings
\[ \ul{z}_\white = (z_{\white,0},\ldots, z_{\white,d}) \in \partial S \]
disjoint from the nodes, 
and a collection of interior markings 
\[ \ul{z}_\black =  (z_{\black,0}, \ldots, z_{\black,e}) \in \on{int}(S) \]
with the following property:

\vskip .1in \noindent On any path of components $S_v$ from the {\em root marking} $z_0$ to a {\em leaf marking} $z_i, i > 0$
there is exactly one quilted component $S_v \subset S$.

\vskip .1in \noindent  A pre-stable
quilted disk $(S,Q)$ is {\em stable}  if each unquilted resp. quilted component has at least three resp. two marked or nodal points.

\vskip .1in \noindent
A {\em stable treed quilted disk} $(C,Q)$ is obtained from a  stable 
quilted disk $(S,Q)$ by replacing each node $S_{v_-} \cap S_{v_+}$ with an edge $T_e$ 
which is an interval of length $\ell(e) \in [0,\infty]$, and each 
marking with a semi-infinite edge $T_e$ with the following property:

\vskip .1in \noindent  If $T_{e'}$ and $T_{e''}$ are two boundary leaves, 
and $T_{e_1},\ldots, T_{e_k}$ are the edges in a path connecting them, then 
\[ \sum_{i=1}^k \pm \ell(e_i) = 0 \]
where the sign is positive if the edges $T_{e_i}$ in the path 
points towards the root edge $T_{e_0}$ and negative otherwise.

\vskip .1in \noindent
Each disk component $S_{v_1},\ldots, S_{v_k}$ of $S$ is equipped with a natural 
(possibly infinite) 
{\em distance to the seam} 
\begin{equation} \label{dv} 
d(v) = \pm \sum_{e \in \gamma} \ell(e) \in [-\infty,\infty] \end{equation}
where $\gamma$ is the path of edges  to the component 
$S_{v_0}$ connecting to the quilted vertices, and the sign is negative
if $S_v$ lies closer resp. further away from the root vertex than the seam.
Define the {\em normalized distance} as 
\begin{equation} \label{tidv}
\ti{d}(v) = \frac{d(v)}{\sqrt{d(v)^2 + 1}} \in [-1,1] \end{equation}
taking values in $\{ -1, 1 \}$ whenever the distance  $d(v)$ is infinite. 
\vskip .1in \noindent
The 
boundary edges $T_e$ are partitioned into those that are closer
to the root edge $T_{e_0}$ than the quilted disk components, 
and those that are further away; we call these groups
of edges {\em root-close} and {\em root-distant} respectively
and denote the subsets of such edges $\Edge_\pm(\Gamma)$
respectively. 
See Figure \ref{mw21}.  
\end{definition}

\begin{figure}[ht]
    \centering       
    \scalebox{.5}{
\includegraphics{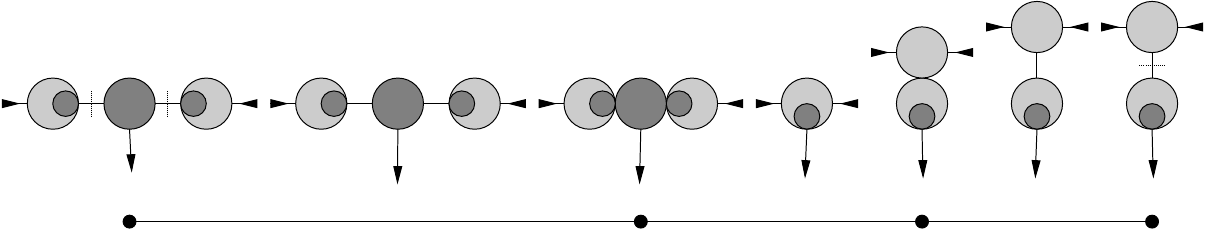}}
    \caption{Moduli space of treed quilted disks}
    \label{mw21}
\end{figure}

We define quilted disks associated to  homotopies of almost complex structures and Morse functions. 
Let  $(J_\pm,f_\pm)$  be two choices of  cylindrical almost complex structure  and Morse data and $J_t, t \in [-1,1]$ a homotopy of almost complex structures with ends modelled on $J_\pm$ resp. $f_t$ an isotopy of Morse data from $f_-$ to $f_+$.   

\begin{definition}
    A {\em quilted holomorphic disk} of domain type $\Gamma$ is a continuous map $u: C \to X$ that is 
$J_{\ti{d}(v)}$ holomorphic on each disk component $S_v, v \in \Ver(\Gamma)$, where $\ti{d}(v)$
was defined in \eqref{tidv} and on each edge the map $u$
is an $f_{\ti{d}(e)}$-gradient trajectory.  \end{definition}

In particular, the map satisfies the $f_+$ gradient vector field  equation  on root-distant edges $T_e, e \in \Edge_+(\Gamma)$ and the $f_-$ gradient vector field on root-close edges $T_e, e \in \Edge_-(\Gamma)$.      See Figure \ref{qdisks}.

\begin{figure}[ht]
    \centering
    \scalebox{.5}{
\begingroup%
  \makeatletter%
  \providecommand\color[2][]{%
    \errmessage{(Inkscape) Color is used for the text in Inkscape, but the package 'color.sty' is not loaded}%
    \renewcommand\color[2][]{}%
  }%
  \providecommand\transparent[1]{%
    \errmessage{(Inkscape) Transparency is used (non-zero) for the text in Inkscape, but the package 'transparent.sty' is not loaded}%
    \renewcommand\transparent[1]{}%
  }%
  \providecommand\rotatebox[2]{#2}%
  \newcommand*\fsize{\dimexpr\f@size pt\relax}%
  \newcommand*\lineheight[1]{\fontsize{\fsize}{#1\fsize}\selectfont}%
  \ifx\svgwidth\undefined%
    \setlength{\unitlength}{248.20524441bp}%
    \ifx\svgscale\undefined%
      \relax%
    \else%
      \setlength{\unitlength}{\unitlength * \real{\svgscale}}%
    \fi%
  \else%
    \setlength{\unitlength}{\svgwidth}%
  \fi%
  \global\let\svgwidth\undefined%
  \global\let\svgscale\undefined%
  \makeatother%
  \begin{picture}(1,1.23676042)%
    \lineheight{1}%
    \setlength\tabcolsep{0pt}%
    \put(0,0){\includegraphics[width=\unitlength,page=1]{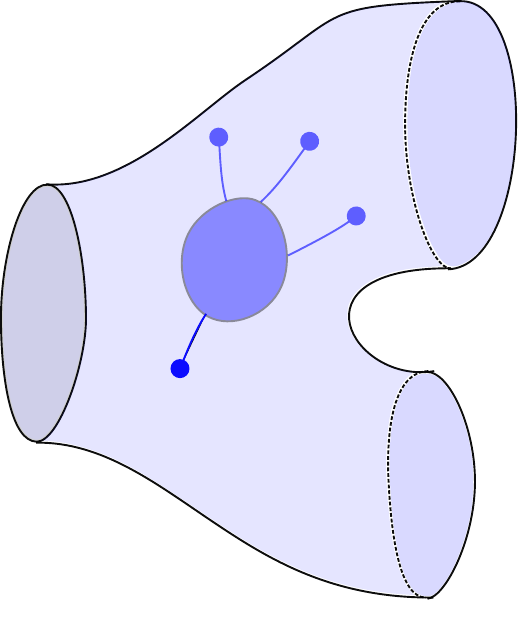}}%
    \put(0.33065654,0.48159729){\color[rgb]{0,0,0}\makebox(0,0)[lt]{\lineheight{0}\smash{\begin{tabular}[t]{l}x\end{tabular}}}}%
    \put(0.04153872,0.27488784){\color[rgb]{0.2,0.2,0.2}\transparent{0.64502198}\makebox(0,0)[lt]{\lineheight{1.25}\smash{\begin{tabular}[t]{l}$Z_-$\end{tabular}}}}%
    \put(0.7828393,0.00625252){\color[rgb]{0.2,0.2,0.2}\transparent{0.64502198}\makebox(0,0)[lt]{\lineheight{1.25}\smash{\begin{tabular}[t]{l}$Z_+$\end{tabular}}}}%
    \put(0.61369954,0.33606212){\color[rgb]{0.2,0.2,0.2}\transparent{0.64502198}\rotatebox{-180}{\makebox(0,0)[lt]{\lineheight{1.25}\smash{\begin{tabular}[t]{l}X\end{tabular}}}}}%
    \put(0,0){\includegraphics[width=\unitlength,page=2]{obs2.pdf}}%
  \end{picture}%
\endgroup%
}
    \caption{A quilted disk in a cobordism}
    \label{qdisks}
\end{figure}

Map types for quilted tree disks are graphs decorated with homology classes as before.  In particular, 
a map type $\bGamma$ consists of a tree $\Gamma$ with vertices $\Ver(\Gamma)$ and edges $\Edge(\Gamma)$,
a subset $\Ver^q(\Gamma) \subset \Ver(\Gamma)$ of vertices corresponding to the quilted components, and a labelling of vertices $v$
by homology classes $\beta(v) \in H_2(\ol{X},\ol{L}) \cup  H_2(\ol{X})$.   For any type $\bGamma$ of quilted treed disk
denote by  $\M^{q,\black}_{\bGamma}(L)$ the moduli space of quilted treed disks of type $\bGamma$ and
\[  \M_{d,1}^q(L) = \bigcup_\bGamma \M^{q,\black}_{\bGamma}(L) \] 
the union over combinatorial types with $d$ incoming leaves,  $\M^{q,\black}_{d,1}(L,\ul{x}_-,x_+)$ the locus with limits $\ul{x}_-,x_+$ along the incoming and outgoing edges and  $\M^{q,\black}_{d,1}(L,\ul{x}_-,x_+)_0$ the locus of rigid maps in the sense of \cite[Equation \eqref{I-rigidmap}]{BCSW1}.   Define 
\[ 
  q_d: CE(L,\partial L;f_0)^{\otimes d} \to CE(L,\partial L; f_1), \
  \ul{x}_- \mapsto \sum_{u \in
    \M_{d,1}(L,\ul{x}_-,x_+)_0} (-1)^{\heartsuit } \wt(u) x_+ .\]
The use of quilts here is a purely combinatorial device, and the seams of the quilts have no geometric meaning as opposed to their use for Lagrangian correspondences in \cite{ainfty}.

\begin{theorem}  Let $(J_\pm,f_\pm)$  be two choices of  cylindrical almost complex structure and $(J_t,f_t)$ a homotopy  as above. 
   Counting quilted holomorphic disks defines  a map 
\[ MC(J_\bullet, f_\bullet): MC(J_+,f_+) \to MC(J_-,f_-). \] 
In particular, non-emptiness of $ MC(J,f)$ is independent of the choice of almost complex structure $J$ on $X$ and Morse function $f_L$ on $L$.
\end{theorem}

\begin{figure}[ht]
    \centering
    \scalebox{.5}{
\begingroup%
  \makeatletter%
  \providecommand\color[2][]{%
    \errmessage{(Inkscape) Color is used for the text in Inkscape, but the package 'color.sty' is not loaded}%
    \renewcommand\color[2][]{}%
  }%
  \providecommand\transparent[1]{%
    \errmessage{(Inkscape) Transparency is used (non-zero) for the text in Inkscape, but the package 'transparent.sty' is not loaded}%
    \renewcommand\transparent[1]{}%
  }%
  \providecommand\rotatebox[2]{#2}%
  \newcommand*\fsize{\dimexpr\f@size pt\relax}%
  \newcommand*\lineheight[1]{\fontsize{\fsize}{#1\fsize}\selectfont}%
  \ifx\svgwidth\undefined%
    \setlength{\unitlength}{248.20524441bp}%
    \ifx\svgscale\undefined%
      \relax%
    \else%
      \setlength{\unitlength}{\unitlength * \real{\svgscale}}%
    \fi%
  \else%
    \setlength{\unitlength}{\svgwidth}%
  \fi%
  \global\let\svgwidth\undefined%
  \global\let\svgscale\undefined%
  \makeatother%
  \begin{picture}(1,1.15974545)%
    \lineheight{1}%
    \setlength\tabcolsep{0pt}%
    \put(0,0){\includegraphics[width=\unitlength,page=1]{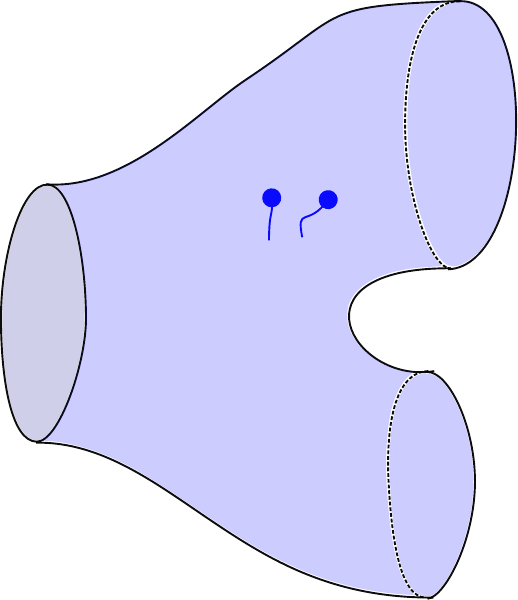}}%
    \put(0.14798554,0.84527062){\color[rgb]{0,0,0}\makebox(0,0)[lt]{\lineheight{0}\smash{\begin{tabular}[t]{l}$b_-$\end{tabular}}}}%
    \put(0.32049532,0.85014142){\color[rgb]{0,0,0}\makebox(0,0)[lt]{\lineheight{0}\smash{\begin{tabular}[t]{l}$b_-$\end{tabular}}}}%
    \put(0.47237676,0.86063315){\color[rgb]{0,0,0}\makebox(0,0)[lt]{\lineheight{0}\smash{\begin{tabular}[t]{l}$b_-$\end{tabular}}}}%
    \put(0.62008695,0.85343529){\color[rgb]{0,0,0}\makebox(0,0)[lt]{\lineheight{0}\smash{\begin{tabular}[t]{l}$b_-$\end{tabular}}}}%
    \put(0,0){\includegraphics[width=\unitlength,page=2]{obs3.pdf}}%
  \end{picture}%
\endgroup%
}
    \caption{Configurations showing that $m(b_+)$ vanishes}
    \label{mb1}
\end{figure}

\begin{proof}   The map between Maurer Cartan spaces is defined as follows.   The count of quilted disks with incoming edges labelled $b_-
\in MC(J_-,f_-)$ defines a formal sum 
\[ b_+ = \sum_{d \ge 0} q_d(b_-,\ldots, b_-) . \] 

To see that this sum defines a Maurer-Cartan solution,  we consider the boundary components of the one-dimensional moduli spaces.  For some $d \ge 0$ consider the one-dimensional boundary component $\M_d^q(L)_1$ of the moduli space of quilted disks
with $d$ incoming leaves. Boundary configurations $u \in \partial \M_d^q(L)_1$ are of the following types: either
there is a single edge of infinite length connecting components in same level, or a configuration of infinite length edges connecting different levels.  We treat the different cases in turn and their contributions to the map of Maurer-Cartan spaces. 

First consider a configuration $u$ with one level and an unquilted disk $u_{v_0}$
attached to a collection of quilted disks $u_{v_1},\ldots, u_{v_r}$ via a collection of broken edges $T_{e_1},\ldots, T_{e_r}$ passing through critical points of $f_L$.
The sum of such configurations is 
\begin{eqnarray*}
     \sum_{r,j_1,\ldots,j_r \ge 0} m_r(q_{j_1}(b_-,\ldots, b_-),\ldots, q_{j_r}(b_-,\ldots, b_-)) &=&
\sum_{r \ge 0} m_r(b_+,\ldots,b_+) \\ &=& m(b_+) .
\end{eqnarray*}

Next consider configurations $u$ with one level and a quilted disk $u_{v_1}$ attached
to an unquilted disk $u_{v_2}$  via a broken edge. These configurations contribute zero because of the Maurer-Cartan condition:
\begin{eqnarray*}
     & \sum_{r,d} q_{r-d+1}(b_-,\ldots, b_-,m_d(b_-,\ldots,b_-),b_-,\ldots, b_-)\\
     &= \sum_{r \ge 0} q_{r-d + 1}(b_-,\ldots,b_-,0,b_-,\ldots, b_-)= 0.
\end{eqnarray*}

The last two cases involve breaking off a level onto a cylindrical end of the configuration:  One the one hand, there are configurations $u$ with two levels with one level $u_1$ mapping to $X$  and one level $u_2$ mapping to $\R \times Z_+$ attached separated by a broken trajectory $u| T_e$.  On the other hand, there are configurations $u$ with two levels $u_1,u_2$ with one level  $u_1$ in $\R \times Z_-$ and one level $u_2$ in $X$ separated by a broken trajectory $u| T_e$ of $\zeta_\circ$.   Such configurations  are impossible since the level $u_1$ in $\R \times Z_-$ necessarily is collection of disks $u_v: S_v \to \R \times Z_-$, at least
one of which has positive area $A(u_v) > 0$ and so at least one puncture.  This contradicts 
\cite[Lemma \ref{I-1cons}]{BCSW1}. The same argument also excludes configurations in (c).

The signed count of the boundary points in the one-dimensional component of the moduli space  vanishes, for any energy bound.  It follows that 
\[ 0 =\sum_{u \in M_d^q(L,x_0,\ldots,x_d)} (-1)^{ \heartsuit } \left( \prod_{i=1}^d b_+(x_i)  \right) \wt(u)
x_0 = m(b_+)   \] 
as desired.
\end{proof}

\begin{theorem} \label{thm:chainh} If $b_+ \in MC(J_+,f_+)$ is mapped to $b_- \in MC(J_-,f_-)$ then 
the induced maps 
\[ \varphi(L,b_-),\varphi(L,b_+): CE(\Lambda_+, \hat{G}(L)) \to CE(\Lambda_-, \hat{G}(L)) \] 
are chain homotopic algebra maps, and similarly for the maps 
on the complex generated by Reeb chords only the maps
\[ \varphi_\white(L,b_+),\varphi_\white(L,b_-): CE_\white(\Lambda_+, \hat{G}(L)) \to CE_\white(\Lambda_-, \hat{G}(L)) \]
are chain homotopic.   Furthermore, the chain homotopy descends to the abelianization in 
\ref{def:ab}.
\end{theorem}

The proof is an argument using quilted moduli spaces.   In preparation for the proof, given an isotopy $(J_t,f_t), t \in [-1,1]$ as above we consider the moduli space $\M^{q,\white}_{d_-,1}(L)$ of quilted punctured disks in $X$ bounding $L$ of expected dimension one with a single incoming puncture $x_+$ at the convex end and some collection $\ul{x}_-$ of outgoing punctures at the concave end, together with a collection of inputs $\ul{x}_\black$ at the edges mapping to the interior of $L$.   

Boundary configurations of the one-dimensional components  
of the moduli spaces of quilted disks defining the homotopy operator are those configurations with some collection of edges with infinite length. 
The first possibility is that of configurations with two levels, connected
by infinite length trajectories, where the first level is quilted.
Let  $u = (u_1,u_2)$  where the first level $(C_1,u_1)$ is a quilted disk  mapping to $\R \times Z_-$ 
and the second level 
$(C_2,u_2)$ maps to $X$ as in Figure \ref{fig:varphib0}. Because the almost complex structure on $\R \times Z_-$ is translation-invariant, forgetting the quilting produces a stable unquilted treed disk 
$u: C \to \R \times Z$ in a moduli space of
negative expected dimension, unless the unquilted disk is a constant strip in which case it is unstable.

\begin{figure}[ht]
    \centering
    \scalebox{.4}{
\begingroup%
  \makeatletter%
  \providecommand\color[2][]{%
    \errmessage{(Inkscape) Color is used for the text in Inkscape, but the package 'color.sty' is not loaded}%
    \renewcommand\color[2][]{}%
  }%
  \providecommand\transparent[1]{%
    \errmessage{(Inkscape) Transparency is used (non-zero) for the text in Inkscape, but the package 'transparent.sty' is not loaded}%
    \renewcommand\transparent[1]{}%
  }%
  \providecommand\rotatebox[2]{#2}%
  \newcommand*\fsize{\dimexpr\f@size pt\relax}%
  \newcommand*\lineheight[1]{\fontsize{\fsize}{#1\fsize}\selectfont}%
  \ifx\svgwidth\undefined%
    \setlength{\unitlength}{843.9941214bp}%
    \ifx\svgscale\undefined%
      \relax%
    \else%
      \setlength{\unitlength}{\unitlength * \real{\svgscale}}%
    \fi%
  \else%
    \setlength{\unitlength}{\svgwidth}%
  \fi%
  \global\let\svgwidth\undefined%
  \global\let\svgscale\undefined%
  \makeatother%
  \begin{picture}(1,0.34827704)%
    \lineheight{1}%
    \setlength\tabcolsep{0pt}%
    \put(0,0){\includegraphics[width=\unitlength,page=1]{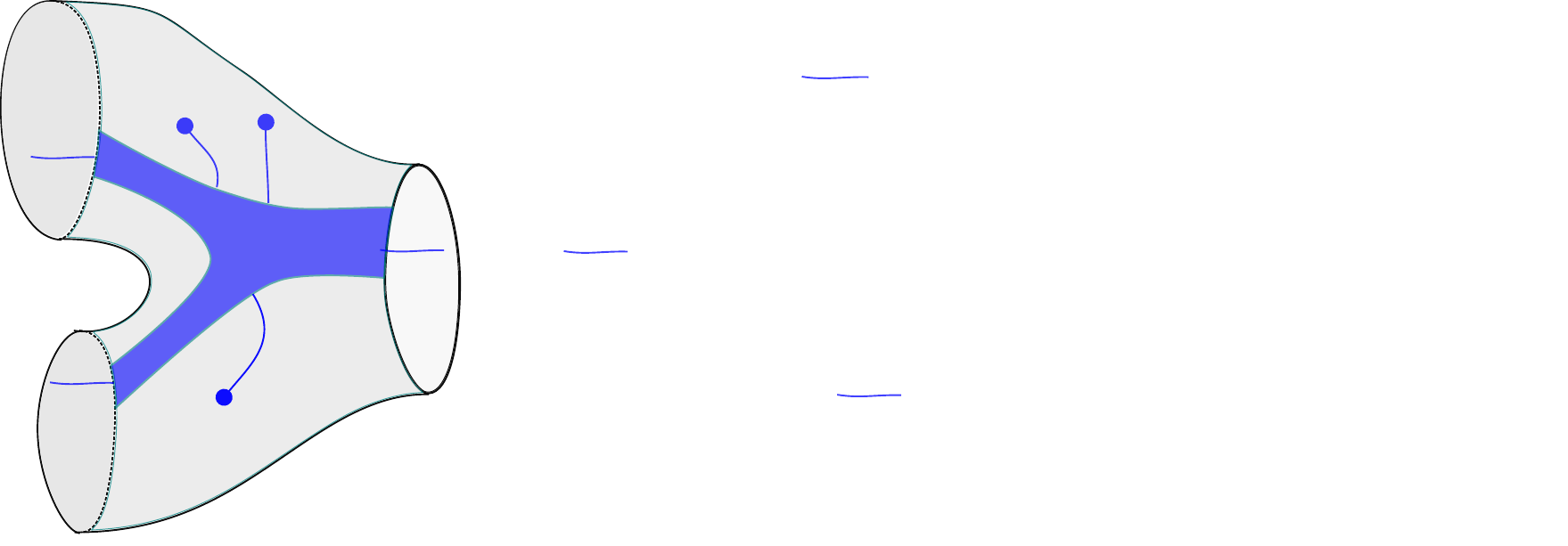}}%
    \put(0.16107261,0.28657637){\color[rgb]{0,0,0}\makebox(0,0)[lt]{\lineheight{0}\smash{\begin{tabular}[t]{l}$b_-$\end{tabular}}}}%
    \put(0.10362278,0.28753386){\color[rgb]{0,0,0}\makebox(0,0)[lt]{\lineheight{0}\smash{\begin{tabular}[t]{l}$b_-$\end{tabular}}}}%
    \put(0.12599656,0.08044259){\color[rgb]{0,0,0}\makebox(0,0)[lt]{\lineheight{0}\smash{\begin{tabular}[t]{l}$b_-$\end{tabular}}}}%
    \put(0,0){\includegraphics[width=\unitlength,page=2]{obs78.pdf}}%
    \put(0.83939603,0.26928805){\color[rgb]{0,0,0}\makebox(0,0)[lt]{\lineheight{0}\smash{\begin{tabular}[t]{l}$b_-$\end{tabular}}}}%
    \put(0.758257,0.27086896){\color[rgb]{0,0,0}\makebox(0,0)[lt]{\lineheight{0}\smash{\begin{tabular}[t]{l}$b_-$\end{tabular}}}}%
    \put(0.7912286,0.03198424){\color[rgb]{0,0,0}\makebox(0,0)[lt]{\lineheight{0}\smash{\begin{tabular}[t]{l}$b_-$\end{tabular}}}}%
    \put(0,0){\includegraphics[width=\unitlength,page=3]{obs78.pdf}}%
  \end{picture}%
\endgroup%
}
    \caption{Contributions to $\varphi(L,b_-)$ (left) and $\varphi(L,b_+)$ (right)}
    \label{fig:varphib0}
\end{figure}

The second possibility is that the second level is quilted. 
Let  $(C,u)$ be a configuration with two levels $(C_1,u_1)$ and $(C_2,u_2)$
where  $(C_1,u_1)$ is a $(J_+,f_+)$-holomorphic map to  $X$ and 
$(C_2,u_2)$ is a collection of quilted disks and spheres in $\R \times Z_+$ as in Figure \ref{fig:varphib0}.  Again, forgetting the quiltings produces a configuration in a moduli space of lower expected dimension unless the quilted disks are constant.  

A third possibility is that  of a configuration with a single level, separated
by an infinite length trajectory.
Configurations $(C,u)$ consisting of a quilted disk $(C_0.u_0)$ mapping to $X$ connecting
to a  unquilted disk $(C_1,u_1)$ mapping to $X$  with edges labelled $b_-$.

 Finally, there are configurations with unquilted disks bubbling off the ends.
 Configurations  $(C,u)$ consisting of a quilted disk $(C_0,u_0)$ mapping to $X$ connected to an unquilted disk $(C_1,u_1)$ mapping to  $\R \times Z_\pm$.

In order to define the chain homotopy claimed in the theorem,  we 
also need moduli spaces of unions of disks.    The necessary cobordism cannot be simply the product of moduli spaces for the components, since this product will have dimension more than one.
 Instead, we need to take a fibered product.   Note that the boundary $\partial \M^{q,\white}_{d_-,1}(L)$ of $\M^{q,\white}_{d_-,1}(L)$ consists two distinguished components:  
\begin{enumerate} 
\item Let 
$\partial_+ \M^{q,\white}_{d_-,1}(L)$ denote the set of configurations in 
$\partial \M^{q,\white}_{d_-,1}(L)$ where the quilted components maps to the convex end $\R \times \Lambda_+$;
\item Let $\partial_- \M^{q,\white}_{d_-,1}(L)$ denote the set of configurations in $\partial \M^{q,\white}_{d_-,1}(L)$ where one of the quilted components maps to the concave end $\R \times \Lambda_-$.
\end{enumerate}
We say that a {\em good function} on  $\M^{q,\white}_{d_-,1}(L)$ is a smooth 
map 
\begin{equation} \label{eq:goodfun} f_{d_-}:  \ \M^{q,\white}_{d_-,1}(L) \to [-1,1] \end{equation} 
that takes value $-1$ on  $\partial_- \M^{q,\white}_{d_-,1}(L)$, value $+1$ on  $\partial_+ \M^{q,\white}_{d_-,1}(L)$, and 
takes values in $(-1,1)$ on the complement; we write $f = f_{d_-}$ for short.   In particular, 
the third type of boundary of $\M^{q,\white}_{d_-,1}$ has image a finite set in $(0,1)$ for any given energy bound.   It follows easily from the fact that  $\M^{q,\white}_{d_-,1}(L)$ is a compact one-manifold with boundary 
that such good functions exist.    The moduli spaces for finite unions of disks are now defined by 
\begin{equation} \label{eq:mql}
\M^{q,\white}(L) = \bigcup_{k=1}^{\infty} \bigcup_{d_{1,-}, \ldots, d_{k,-}} \prod_{i=1}^k \M^{q,\white}_{d_{1,-}}(L) 
\times_{[-1,1]} \ldots \times_{[-1,1]} \M^{q,\white}_{d_{k,-}}(L) \end{equation}
Here the fiber product ranges over the maps $f$ of \eqref{eq:goodfun}.
The boundary of this fiber product is not cut out transversally in general stratum-wise, because of repetition of the boundary components that map to $(0,1)$.  After a generic perturbation in $f$, both the boundary and interior of $\M^{q,\white}(L)$ is transversally cut out; these perturbations may be constructed inductively starting with the moduli spaces of expected dimension less than one so that the perturbations on each component $u_k, k = 1,2$ corresponding to splitting of types $\bGamma = \bGamma_1 \# \bGamma_2$  depend only on $u_k$.  By allowing multi-valued perturbations also may assume that if $\bGamma$ is disconnected, then the perturbations are invariant 
under permutation of the components.  The boundary components of $\M^{q,\white}(L)$ then consist of three types  $\partial_+ \M^{q,\white}(L)$, $\partial_- \M^{q,\white}(L), \partial_0 \M^{q,\white}(L)$ where
\begin{enumerate}
    \item $\partial_+ \M^{q,\white}(L)$ consists of configurations
    where the quilted components have moved onto the convex end, $\R \times \Lambda_+$, and is in bijection with the product of the moduli spaces $\M(L,J_+)$ for the almost complex structure $J_+$ (by removing the quilted components).
    \item $\partial_- \M^{q,\white}(L)$ consists of configurations
    where some of the quilted components move onto the concave end $\R \times \Lambda_-$, and is in bijection with the product of the moduli spaces $\M(L,J_+)$ for the almost complex structure $J_+$ with moduli spaces contributing to the map of moduli spaces $MC(J_-) \to MC(J_+)$;
    \item $\partial_0 \M^{q,\white}(L)$ consists of configurations where
     the quilted components lie on the cobordism $X$, and consists
     of configurations with two levels, one of which maps to $X$
     and the other maps to $\R \times Z_+$ or $\R \times Z_-$ (and so contributes to the differential for $\Lambda_+$ or $\Lambda_-$.)
\end{enumerate}

We use the third kind of count to define the homotopy operator. 
Let 
\[ \M^{q,\white}_{d_-,d_+,e}(L\ul{x}_-,\ul{x}_+,\ul{x}_\black)_0 \subset \M^{q,\white}(L) \] 
denote the moduli space of finite unions of rigid disks with 
limits at the ends
\[ \ul{x}_\pm = ( x_1,\ldots, x_{d_\pm} ) \in \cI(\Lambda_\pm)^{d_\pm} \] 
and limits
\[ \ul{x}_\black = ( x_{\black,1},\ldots, x_{\black,d_\black} ) \in \cI(L,\partial L)^{d_\black} \]   
along the edges in the interior of $L$.   Counts of 
rigid elements of the moduli space of finite unions of quilted
disks counted in \eqref{eq:mql} define a map 
\[ h(b_+):  CE(\Lambda_+,\hat{G}(L)) \to CE(\Lambda_-,\hat{G}(L)) .\]
by
\[ (h(b_+))(\ul{x}_-) = 
\sum_{\ul{x}, u \in \M^{q,\white}_{d_+,d_-,e}(\ul{x}_+,\ul{x}_-;\ul{x}_\black)_0 }
(-1)^{\heartsuit } \left( \prod_{i=1}^e b_+(x_{\black,i}) \right) \gamma(x_0) x_{-,1} \otimes \ldots \otimes x_{-,d_+} \]
as in Figure \ref{fig:homotopy}.    Invariance of the perturbations under permutation implies that the operator $h$ is well-defined
on the space of abelianized chains in Definition \ref{def:ab}.

\begin{figure}[ht]
    \centering
    \scalebox{.5}{
\begingroup%
  \makeatletter%
  \providecommand\color[2][]{%
    \errmessage{(Inkscape) Color is used for the text in Inkscape, but the package 'color.sty' is not loaded}%
    \renewcommand\color[2][]{}%
  }%
  \providecommand\transparent[1]{%
    \errmessage{(Inkscape) Transparency is used (non-zero) for the text in Inkscape, but the package 'transparent.sty' is not loaded}%
    \renewcommand\transparent[1]{}%
  }%
  \providecommand\rotatebox[2]{#2}%
  \newcommand*\fsize{\dimexpr\f@size pt\relax}%
  \newcommand*\lineheight[1]{\fontsize{\fsize}{#1\fsize}\selectfont}%
  \ifx\svgwidth\undefined%
    \setlength{\unitlength}{248.20522278bp}%
    \ifx\svgscale\undefined%
      \relax%
    \else%
      \setlength{\unitlength}{\unitlength * \real{\svgscale}}%
    \fi%
  \else%
    \setlength{\unitlength}{\svgwidth}%
  \fi%
  \global\let\svgwidth\undefined%
  \global\let\svgscale\undefined%
  \makeatother%
  \begin{picture}(1,1.15974555)%
    \lineheight{1}%
    \setlength\tabcolsep{0pt}%
    \put(0,0){\includegraphics[width=\unitlength,page=1]{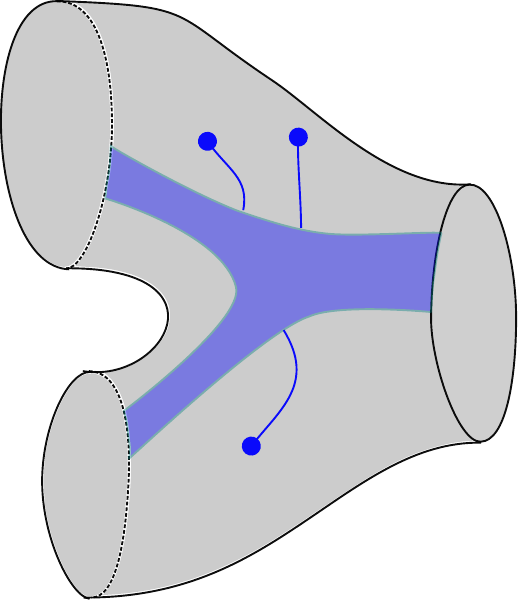}}%
    \put(0.54770942,0.94993932){\color[rgb]{0,0,0}\makebox(0,0)[lt]{\lineheight{0}\smash{\begin{tabular}[t]{l}b\end{tabular}}}}%
    \put(0.35235768,0.95319517){\color[rgb]{0,0,0}\makebox(0,0)[lt]{\lineheight{0}\smash{\begin{tabular}[t]{l}b\end{tabular}}}}%
    \put(0.42843722,0.24900444){\color[rgb]{0,0,0}\makebox(0,0)[lt]{\lineheight{0}\smash{\begin{tabular}[t]{l}b\end{tabular}}}}%
    \put(0,0){\includegraphics[width=\unitlength,page=2]{obs6.pdf}}%
  \end{picture}%
\endgroup%
}
    \caption{Disks defining the homotopy operator}
    \label{fig:homotopy}
\end{figure}

\begin{proof}[Proof of Theorem \ref{thm:chainh}]
Boundary components $\partial_\pm \M^{q,\white}(L)$ 
contribute the maps $\varphi(L,b_\pm)$ while 
the third type of boundary contributes 
\[  h(b_+) \delta_+ + \delta_- h(b_+) .\]
See Figure \ref{fig:deltaplus} and Figure \ref{fig:deltaminus}. The justification of signs is similar to that for the 
square-zero property of the differential.

\begin{figure}[ht]
    \centering
    \scalebox{.5}{
\begingroup%
  \makeatletter%
  \providecommand\color[2][]{%
    \errmessage{(Inkscape) Color is used for the text in Inkscape, but the package 'color.sty' is not loaded}%
    \renewcommand\color[2][]{}%
  }%
  \providecommand\transparent[1]{%
    \errmessage{(Inkscape) Transparency is used (non-zero) for the text in Inkscape, but the package 'transparent.sty' is not loaded}%
    \renewcommand\transparent[1]{}%
  }%
  \providecommand\rotatebox[2]{#2}%
  \newcommand*\fsize{\dimexpr\f@size pt\relax}%
  \newcommand*\lineheight[1]{\fontsize{\fsize}{#1\fsize}\selectfont}%
  \ifx\svgwidth\undefined%
    \setlength{\unitlength}{359.97917356bp}%
    \ifx\svgscale\undefined%
      \relax%
    \else%
      \setlength{\unitlength}{\unitlength * \real{\svgscale}}%
    \fi%
  \else%
    \setlength{\unitlength}{\svgwidth}%
  \fi%
  \global\let\svgwidth\undefined%
  \global\let\svgscale\undefined%
  \makeatother%
  \begin{picture}(1,0.81293685)%
    \lineheight{1}%
    \setlength\tabcolsep{0pt}%
    \put(0,0){\includegraphics[width=\unitlength,page=1]{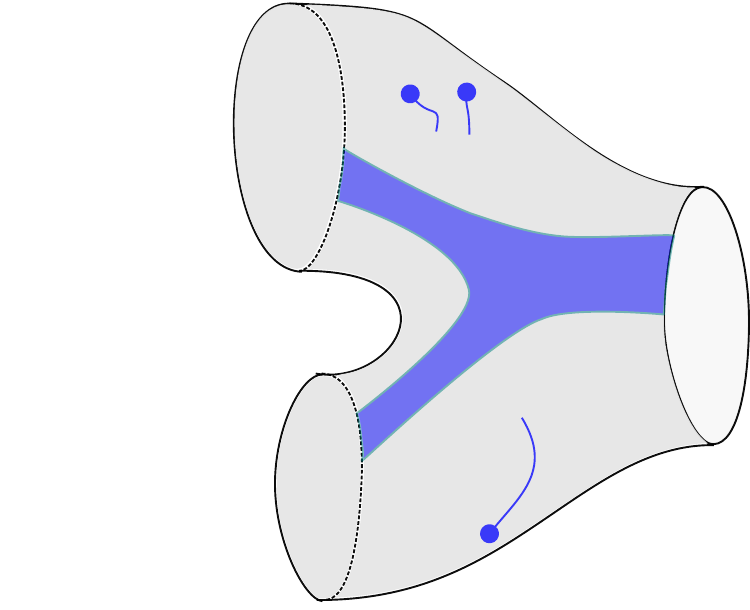}}%
    \put(0.54431819,0.08930125){\color[rgb]{0,0,0}\makebox(0,0)[lt]{\lineheight{0}\smash{\begin{tabular}[t]{l}$b_-$\end{tabular}}}}%
    \put(0.44592056,0.73513133){\color[rgb]{0,0,0}\makebox(0,0)[lt]{\lineheight{0}\smash{\begin{tabular}[t]{l}$b_-$\end{tabular}}}}%
    \put(0.54215477,0.73128199){\color[rgb]{0,0,0}\makebox(0,0)[lt]{\lineheight{0}\smash{\begin{tabular}[t]{l}$b_-$\end{tabular}}}}%
    \put(0,0){\includegraphics[width=\unitlength,page=2]{obs10.pdf}}%
  \end{picture}%
\endgroup%
}
    \caption{Contributions to $\delta_- h(b_-) $}
    \label{fig:deltaplus}
\end{figure}

\begin{figure}[ht]
    \centering
    \scalebox{.5}{
\begingroup%
  \makeatletter%
  \providecommand\color[2][]{%
    \errmessage{(Inkscape) Color is used for the text in Inkscape, but the package 'color.sty' is not loaded}%
    \renewcommand\color[2][]{}%
  }%
  \providecommand\transparent[1]{%
    \errmessage{(Inkscape) Transparency is used (non-zero) for the text in Inkscape, but the package 'transparent.sty' is not loaded}%
    \renewcommand\transparent[1]{}%
  }%
  \providecommand\rotatebox[2]{#2}%
  \newcommand*\fsize{\dimexpr\f@size pt\relax}%
  \newcommand*\lineheight[1]{\fontsize{\fsize}{#1\fsize}\selectfont}%
  \ifx\svgwidth\undefined%
    \setlength{\unitlength}{364.76408711bp}%
    \ifx\svgscale\undefined%
      \relax%
    \else%
      \setlength{\unitlength}{\unitlength * \real{\svgscale}}%
    \fi%
  \else%
    \setlength{\unitlength}{\svgwidth}%
  \fi%
  \global\let\svgwidth\undefined%
  \global\let\svgscale\undefined%
  \makeatother%
  \begin{picture}(1,0.78915363)%
    \lineheight{1}%
    \setlength\tabcolsep{0pt}%
    \put(0,0){\includegraphics[width=\unitlength,page=1]{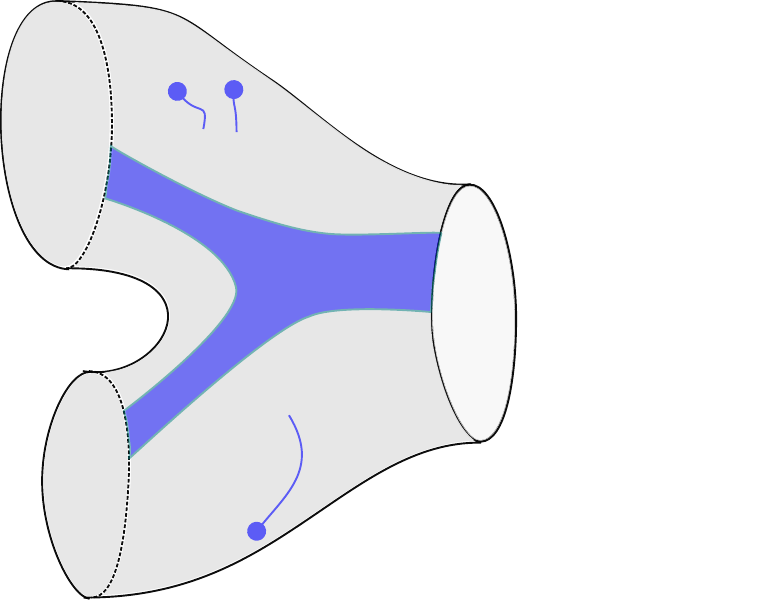}}%
    \put(0.23203202,0.07467703){\color[rgb]{0,0,0}\makebox(0,0)[lt]{\lineheight{0}\smash{\begin{tabular}[t]{l}$b_-$\end{tabular}}}}%
    \put(0.13710038,0.71653696){\color[rgb]{0,0,0}\makebox(0,0)[lt]{\lineheight{0}\smash{\begin{tabular}[t]{l}$b_-$\end{tabular}}}}%
    \put(0.26050749,0.70331476){\color[rgb]{0,0,0}\makebox(0,0)[lt]{\lineheight{0}\smash{\begin{tabular}[t]{l}$b_-$\end{tabular}}}}%
    \put(0,0){\includegraphics[width=\unitlength,page=2]{obs11.pdf}}%
  \end{picture}%
\endgroup%
}
    \caption{Contributions to $h(b_-) \delta_-$}
    \label{fig:deltaminus}
\end{figure}
 Putting everything together, we obtain 
\[ \varphi(L,b_+) - \varphi(L,b_-) = h(b_-) \delta_- + \delta_+ h(b_-) .\]
\end{proof}

\begin{remark}
    With some more work, one 
    should be able to arrange that $h$ is a homotopy of 
    dg algebra maps as in Ekholm-Honda-Kalman \cite[Lemma 3.13]{ehk}, Kalman \cite[Table 1]{kalman}. Namely, 
    for each energy bound one may  break up the interval $[-1,1]$ into subintervals $[t_i,t_{i+1}]$ such that each subinterval contains the image of at most one element of the space of quilted disks, we may assume that there is a single configuration contributing to the chain homotopy.   In this configuration, which is a disjoint union of quilted disks, there is a single quilted disk which is in a moduli space of expected dimension zero, and the remaining quilted disks in the union are in a moduli space of expected dimension one.  The count of disks of the latter type are independent of the choice of $t \in [t_i, t_{i+1}]$, and one obtains a dga homotopy between the chain maps associate to $t_{i}$ and $t_{i+1}$.  The required map from $\varphi(L,b_+)$ and $\varphi(L,b_-)$ is defined by composing the given homotopies.  See Seidel \cite[Chapter 10]{se:bo} and Ma'u-Wehrheim-Woodward \cite[Chapter 7]{ainfty} for the corresponding results in the \ainfty context.  
\end{remark}

\begin{lemma}  \label{deformrem} 
Let $\Lambda_\rho$
be a family of Legendrians in a family  of stable Hamiltonian manifolds $(Z,\alpha_\rho,\omega)$ for $\rho \in [0,1]$ so that 
 $\Pi = p(\Lambda_0)$ is a monotone Lagrangian in $Y$. 
Then for suitable choices of almost complex structure the differential algebra $CE(\Lambda_\rho)$  is equal to $ CE(\Lambda)$ for $\rho$ sufficiently small, with coefficients in the completion 
$\hat{G}(\Lambda_\rho)$ which contains $G(\Lambda_0)$.  In particular, any tame filling of $\Lambda_\rho$
defines an augmentation for $CE(\Lambda)$.
\end{lemma}
 
\begin{proof} 
For any fixed number of punctures, the number of punctured disks $u: C \to \R \times Z$ bounding $\R \times \Lambda$
is finite as the completion in \eqref{cglam} is not necessary.   
 Given an almost complex
structure $J$ taming $\omega_0$, $J$ also tames $\omega_\rho$ for $\rho$
sufficiently small.  The holomorphic disks bounding $\Lambda_\rho$ which define the differential on $CF(\Lambda_\rho)$ are in bijection with disks bounding $\Lambda_0$,
which define the differential on $CF(\Lambda_0)$, by transversality.  
\end{proof}

\begin{remark}  The homotopy type of $CE(\Lambda_\rho)$ is not independent of $\rho$, because it is only defined over a completion which depends on $\Lambda_\rho$.  The content of the Lemma is that
since $\Lambda_0$ is monotone, the completion is not necessary 
and $CE(\Lambda_\rho)$ is the tensor product of the uncompleted
algebra $CE(\Lambda_0)$ with a suitable completed coefficient ring.
\end{remark}

We have in mind for applications of the Lemma the example especially of the Clifford Legendrian $\Lambda$.  In this case, any filling of a perturbation of $\Lambda$ (for example,
the intersection $L \cap Z$ of the Harvey-Lawson filling $L$ with $Z = S^{2n-1}$) defines
an augmentation of $CE(\Lambda)$.

\begin{example}  Continuing Example \ref{hopf2}, the two fillings
  $L_{(2),\pm} \subset \C^n$ obtained from paths $\gamma_\pm : \R \to \R^2$ above respectively below the
  critical value of the projection are not isotopic through exact
  fillings.  Indeed, since the generators $\cc_{12}, \cc_{21}$ are
  closed, existence of an isotopy would imply that the values
  of the augmentations $\varphi(L_{(2),\pm})$ on $\cc_{12}, \cc_{21}$
  are equal, which is not the case (having values $n,1$ and $1,n$
  respectively.)    In the case $n = 2$, the non-existence of an isotopy is the local version of
  the statement that the Chekanov and Clifford tori are not
  Hamiltonian isotopic as in, for example, Vianna \cite{vianna:inf}.  That is, existence of a local isotopy would
  imply existence of a global isotopy. 
\end{example}

\begin{lemma}\label{lem:isotope}   Suppose that $(X,\omega_s), s \in [0,1]$ is family of symplectic cobordisms 
with convex end $(Z_+,\alpha_+)$ and concave end $(Z_-,\alpha_-)$, so that $\frac{\d}{\d s} \omega_s$ is supported in a compact subset of the complement of the boundary, and $L_s \subset X_s$
is a family of Lagrangian cobordisms with concave end $\Lambda_+ \subset Z_+$
and convex end $\Lambda_- \subset Z_-$.    There exists a symplectomorphisms $\psi_s: (X,\omega_0) \to (X,\omega_s)$ with $\psi_s(L_0) = L_s$. 
\end{lemma}

 \begin{proof} First choose a diffeomorphism $\kappa: X \to X$ so that $\kappa(L_0) = L_s$ and $\kappa$ is equal to the identity on $Z_\pm$;  explicitly, $\kappa$ is the flow of a vector field on the submanifold$\cup_s( \{ s \} \times L_s) \subset \R\times X$ that projects to $\dds$ on $[0,1]$.  
 After replacing $L_s$ with $\kappa_s(L_s)$ an $\omega_s$ with $\kappa_{s,*} \omega_s$, 
 we may assume that $L_s$ is independent of $s$.  The variation $\dds \omega_s$ is exact, since it vanishes on $Z_\pm \cup L$, and so  $\dds \omega_s = \d \beta_s$ for some family of one-forms $\beta_s$ that vanishes on $L$.    Moser isotopy then produces a  family of diffeomorphisms $\psi_s: X \to X$ so that $\psi_s^* \omega_s = \omega_0$.
\end{proof}
 
As a result, we may treat an isotopy of Lagrangian cobordisms as a variation in 
perturbation data, by replacing a given almost complex structure by its pull-back under
the diffeomorphism in Lemma \ref{lem:isotope}.

\subsection{The composition theorem}

According to the philosophy of topological quantum field theory, chain maps associated
to cobordisms should satisfy a composition-equals-gluing axiom.  For cobordism maps associated
to tame cobordisms equipped with bounding chains we show that for a natural operation on 
bounding chains (the formal sum in the limit of long neck length), the composition of the cobordism maps is the cobordism map for the composition.

\begin{theorem} \label{composition}
Suppose $(X',L')$ and $(X'',L'')$ are positive composable tame cobordism
pairs equipped with bounding chains $b' \in MC(L'), b'' \in MC(L'')$.   Denote the cobordism maps
\[ {\varphi}(L',b'):  HE(\Lambda_+') \to HE(\Lambda_-'), 
\quad  {\varphi}(L'',b''):  HE(\Lambda_+'') \to HE(\Lambda_-'') \] 
and $X = X'' \circ X'$ is the composed cobordism then there exists a map 
\begin{equation} \label{mcmap} \circ:  MC(L') \times MC(L'') \to MC(L) \end{equation}
such that
 \begin{equation} \label{complaw}
 [\varphi({L'' \circ L'},b'' \circ b')]  = [\varphi({L''},b'')] \circ [\varphi({L'},b')] . \end{equation}
Furthermore, the cobordism map associated to the trivial cobordism is the identity:
\[ \varphi({\R \times \Lambda},0) = \on{Id}_{CE(\Lambda)} .\]
\end{theorem}

\begin{proof}   We first construct the composition map \eqref{mcmap}.
It suffices to construct the map for the complex structure on $X$ in the neck
stretching limit in which $X$ approaches the broken symplectic manifold $X = X' \cup X''$ glued along $Y'_+ = Y''_-$.  In this limit, holomorphic curves in $X$ are replaced by  buildings in $X = X' \cup X''$, and the Morse function $f_L$ on $L$ is equal to 
the Morse functions $f_{L'},f_{L''}$ on $L',L''$.  The critical locus is the union
\[ \crit(f_L) = \crit(f_{L'}) \cup \crit(f_{L''}) . \]
With respect to this splitting, define $b'' \circ b'$ as the formal sum of $b''$ and $b'$. 

 We prove the composition law \eqref{complaw}.  Let $u: C \to X$ be a building contributing to $\varphi({L'' \circ L'},b'' \circ b')$, necessarily consisting of two levels $u',u''$.
Each of $u',u''$ have a single incoming puncture.
Indeed,  otherwise there would be a component of $u'$
with no incoming puncture, which is impossible 
by \cite[Lemma \ref{I-1cons}]{BCSW1}. Let $v=(v',v'')$ denote the component of $u$ which corresponds to the bounding chain insertions. The sum over all treed disks $u'$
with $x$ as output and remaining leaves labelled by 
$u'$ vanishes, being the coefficient of $x$ in  $m(b').$
Thus, the only contributions arise from configurations
$u=(u',u'')$ where both $u',u''$ have outputs in the interior of their corresponding levels
as in Figure \ref{fig:compose}. 
\begin{figure}[ht]
    \centering
    \scalebox{.5}{
\begingroup%
  \makeatletter%
  \providecommand\color[2][]{%
    \errmessage{(Inkscape) Color is used for the text in Inkscape, but the package 'color.sty' is not loaded}%
    \renewcommand\color[2][]{}%
  }%
  \providecommand\transparent[1]{%
    \errmessage{(Inkscape) Transparency is used (non-zero) for the text in Inkscape, but the package 'transparent.sty' is not loaded}%
    \renewcommand\transparent[1]{}%
  }%
  \providecommand\rotatebox[2]{#2}%
  \newcommand*\fsize{\dimexpr\f@size pt\relax}%
  \newcommand*\lineheight[1]{\fontsize{\fsize}{#1\fsize}\selectfont}%
  \ifx\svgwidth\undefined%
    \setlength{\unitlength}{626.77473414bp}%
    \ifx\svgscale\undefined%
      \relax%
    \else%
      \setlength{\unitlength}{\unitlength * \real{\svgscale}}%
    \fi%
  \else%
    \setlength{\unitlength}{\svgwidth}%
  \fi%
  \global\let\svgwidth\undefined%
  \global\let\svgscale\undefined%
  \makeatother%
  \begin{picture}(1,0.61523038)%
    \lineheight{1}%
    \setlength\tabcolsep{0pt}%
    \put(0,0){\includegraphics[width=\unitlength,page=1]{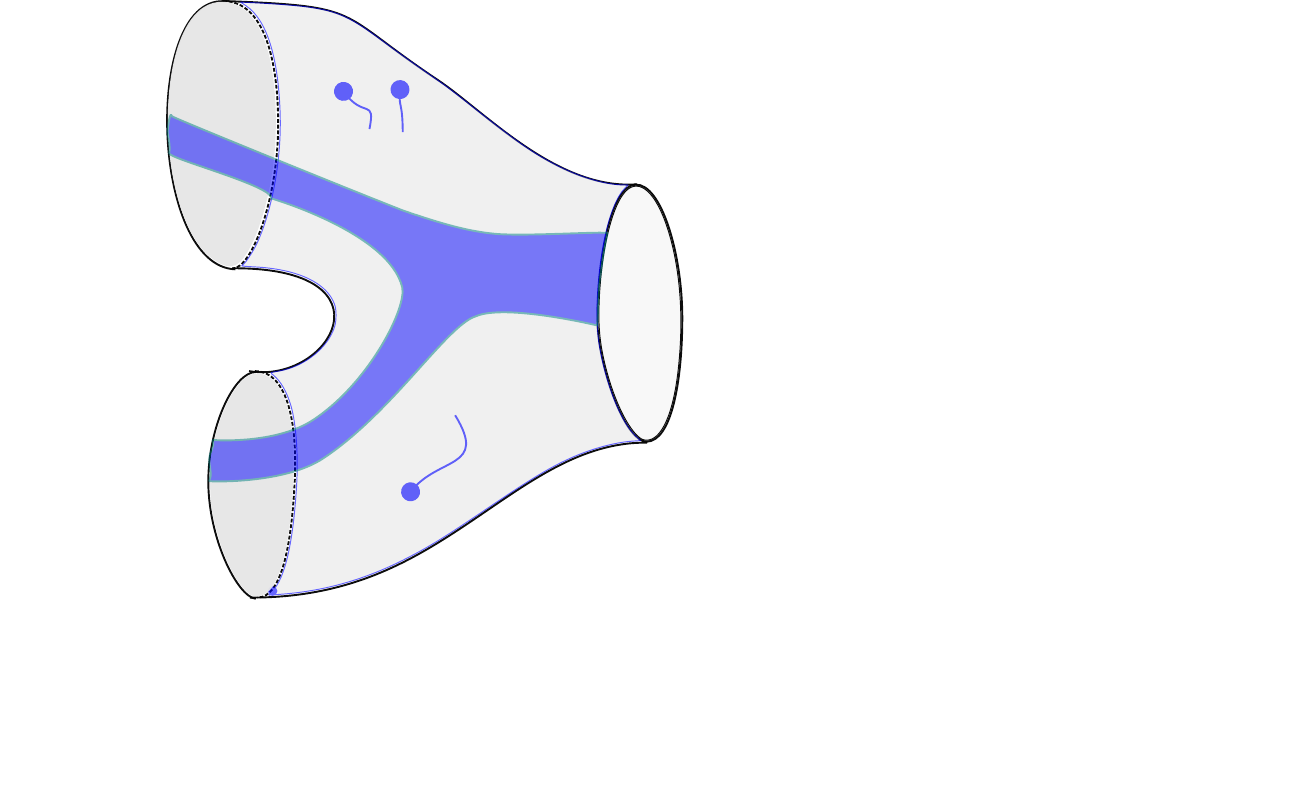}}%
    \put(0.28889433,0.56342647){\color[rgb]{0,0,0}\makebox(0,0)[lt]{\lineheight{0}\smash{\begin{tabular}[t]{l}b''\end{tabular}}}}%
    \put(0.23447306,0.56575846){\color[rgb]{0,0,0}\makebox(0,0)[lt]{\lineheight{0}\smash{\begin{tabular}[t]{l}b''\end{tabular}}}}%
    \put(0.28545429,0.22433587){\color[rgb]{0,0,0}\makebox(0,0)[lt]{\lineheight{0}\smash{\begin{tabular}[t]{l}b''\end{tabular}}}}%
    \put(0,0){\includegraphics[width=\unitlength,page=2]{obs12.pdf}}%
    \put(0.72802479,0.42215261){\color[rgb]{0,0,0}\makebox(0,0)[lt]{\lineheight{0}\smash{\begin{tabular}[t]{l}b'\end{tabular}}}}%
    \put(0.68649552,0.41837202){\color[rgb]{0,0,0}\makebox(0,0)[lt]{\lineheight{0}\smash{\begin{tabular}[t]{l}b'\end{tabular}}}}%
    \put(0.73664134,0.08296503){\color[rgb]{0,0,0}\makebox(0,0)[lt]{\lineheight{0}\smash{\begin{tabular}[t]{l}b'\end{tabular}}}}%
    \put(0,0){\includegraphics[width=\unitlength,page=3]{obs12.pdf}}%
  \end{picture}%
\endgroup%
}
    \caption{Composition of bounding chains}
    \label{fig:compose}
\end{figure}
Furthermore, the tameness assumptions guarantee that each disk
has at least one incoming puncture as in \cite[Corollary \ref{I-onepos}]{BCSW1}. Thus if $b =  b'' \circ b'$ then 
\[ m(b) = m(b') + m(b'') = 0 .\]
The same arguments as before show that the map defined by counts of buildings
in $X$ and counts of maps in $X$ for some modified $b$ define chain-homotopic maps 
from $CE(\Lambda,\hat{G}(L))$ to $CE(\Lambda'',\hat{G}(L))$, and the claim follows. 
\end{proof} 

\subsection{Examples of unobstructed fillings}

We wish to show the the fillings of Harvey-Lawson type are unobstructed and so define augmentations of the corresponding Chekanov-Eliashberg algebras.  The existence of these fillings will be used to compute the augmentation varieties in various examples.  Specifically, the computations will be used to give lower bounds for dimensions of augmentation varieties in the examples where there exist unobstructed fillings.  

\begin{lemma} \label{lem:hl2} The Harvey-Lawson filling
  $L \cong \R^2 \times (S^1)^{n-2} \subset \C^n$ in Example \ref{ex:HLfilling} is unobstructed.
\end{lemma}

\begin{proof} 
We claim that the zero chain is a bounding chain for the standard complex structure and a judicious choice of Morse function. Consider a Morse function $f: L \to \R$ with the property that any gradient trajectory  $\gamma: [0,\infty) \to L$ starting from the
boundary of a holomorphic disk $u: C \to X$ flows to infinity in the sense
that 
\[ \inf |\gamma(t) | \ge 1/2, \quad \lim_{t \to \infty} |\gamma(t)| = \infty . \] 
Fix a coordinate $z_1$ on the $\R^2 \cong \C$
factor in $L$ so that the basic disks map to $z_1 = 0$, by the classification in 
Lemma \ref{lem:allmult}.
A particular
function with the required property is 
\[ f(z_1) =  \frac{(|z_1|^2 - c^2)^2 } { ( |z_1|^2 - c^2)^2
+ 1 } .
\] 
Any generic perturbation of $f$ is Morse and still has the property that gradient trajectories meeting $z_1 = 0$ flow to infinity.   By the choice of Morse function, 
there are no configurations in which a holomorphic disk is connected to a critical
point by a gradient trajectory.
\end{proof}

 Next we generalize the filling of the Hopf Legendrian in \eqref{L2} 
to fillings of Legendrians with two components.    First, we consider the following one-dimensional situation.    

\begin{definition}  A {\em matching} is a locally compact one-manifold $\Upsilon \subset \R \times S^1$ which is cylindrical near infinity  in the sense that 
\begin{eqnarray*}  \Upsilon \cap ( [T - \eps ,T] \times S^1) &=& [T - \eps, T] \times \Upsilon_+ \\ 
  \Upsilon \cap ( [-T,-T + \eps] \times S^1) &=&  [-T,-T+\eps] \times \Upsilon_-
 \end{eqnarray*}
for some finite sets $\Upsilon_\pm$ and any large enough $T>0$, as shown in Figure \ref{Upsilon}.   
\end{definition}

\begin{figure}[ht]
    \centering
    \scalebox{.5}{
\begingroup%
  \makeatletter%
  \providecommand\color[2][]{%
    \errmessage{(Inkscape) Color is used for the text in Inkscape, but the package 'color.sty' is not loaded}%
    \renewcommand\color[2][]{}%
  }%
  \providecommand\transparent[1]{%
    \errmessage{(Inkscape) Transparency is used (non-zero) for the text in Inkscape, but the package 'transparent.sty' is not loaded}%
    \renewcommand\transparent[1]{}%
  }%
  \providecommand\rotatebox[2]{#2}%
  \newcommand*\fsize{\dimexpr\f@size pt\relax}%
  \newcommand*\lineheight[1]{\fontsize{\fsize}{#1\fsize}\selectfont}%
  \ifx\svgwidth\undefined%
    \setlength{\unitlength}{295.5410286bp}%
    \ifx\svgscale\undefined%
      \relax%
    \else%
      \setlength{\unitlength}{\unitlength * \real{\svgscale}}%
    \fi%
  \else%
    \setlength{\unitlength}{\svgwidth}%
  \fi%
  \global\let\svgwidth\undefined%
  \global\let\svgscale\undefined%
  \makeatother%
  \begin{picture}(1,0.66348476)%
    \lineheight{1}%
    \setlength\tabcolsep{0pt}%
    \put(0,0){\includegraphics[width=\unitlength,page=1]{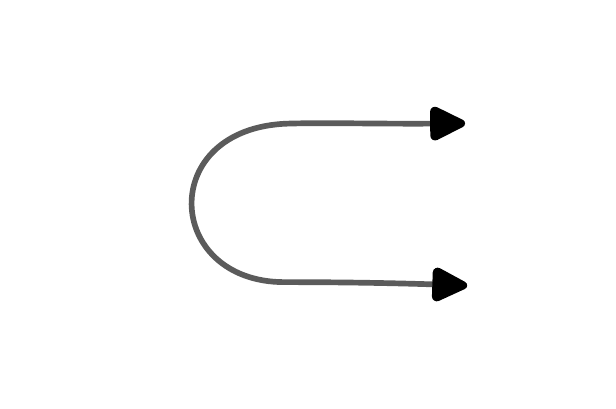}}%
    \put(0.3799854,0.47581024){\color[rgb]{0,0,0}\transparent{0.64502198}\makebox(0,0)[lt]{\lineheight{1.25}\smash{\begin{tabular}[t]{l}$\Upsilon$\end{tabular}}}}%
    \put(0,0){\includegraphics[width=\unitlength,page=2]{Upsilon.pdf}}%
  \end{picture}%
\endgroup%
}
    \caption{A one-dimensional Lagrangian with cylindrical ends}
    \label{Upsilon}
\end{figure}

Let $Z$ be a tame circle-fibered contact manifold and $\R \times Z$ the symplectization
equipped with symplectic form $\d(e^s \alpha)$. 

\begin{definition}  Given a compact Legendrian $\Lambda \subset Z$ and a matching 
$\Upsilon$, the \textit{matching cobordism}  is the image of  $\Upsilon \times \Lambda$
under the action map $(\R \times S^1) \times Z \to Z$, denoted
\[\Upsilon\Lambda = \Big\{ (r,e^{i\theta}\cdot z) \Big| z\in \Lambda, \;\; (r,e^{i\theta}) \in \Upsilon \Big\} .\]
\end{definition}

\begin{lemma} \label{upslem} The submanifold $\Upsilon \Lambda \subset X$ is a Lagrangian cobordism 
from $\Upsilon_- \Lambda$ to $\Upsilon_+ \Lambda$.   If $Y$ has integral symplectic class 
and $\Lambda_+$ is empty then $L$ is tame.  If $\Upsilon$ is simply connected then $L$ is unobstructed.
\end{lemma} 

\begin{proof}  We first check the Lagrangian condition.  The tangent space $T(\Upsilon \Lambda)$
is spanned by  the images of $T\Upsilon$ and $T\Lambda$.  Let $v,w \in \Vect(L)$
 be vector fields in the images of $\Vect(\Upsilon)$ and
 $\Vect(\Lambda)$ respectively.   The contraction of $v,w$ with the
 symplectic form is 
\[ \d (e^s \alpha)  (v,w) = e^s \d \alpha(v,w) + e^s \d s \alpha (v,w)   = 0 \] 
since $\iota(v) \d \alpha = \iota(w) (\d s \wedge \alpha)  = 0 $.
It follows that $L$ is Lagrangian.

Next we check the tameness conditions.  By assumption
 $Y$ has integral symplectic class 
and $\Lambda_+$ is empty.   
The condition \cite[\ref{I-p2}]{BCSW1} holds since 
\[ c_1(\ol{X}) - [Y_-]^\dual = p^* c_1(Y). \] 
The conditions \ref{I-p1} and \ref{I-p3} in \cite{BCSW1} hold trivially.
Suppose $\Upsilon$ is simply connected, that is, each component is an interval. 
 The first homology $H_1(\Lambda)$ surjects onto $H_1(L)$, and it follows that $L$ is exact and bounds no holomorphic disks.  Therefore, $L$ is unobstructed by Lemma \ref{lem:exactunobs}.
\end{proof} 

\subsection{Invariance of Legendrian contact homology} 

The results of the previous section allow us to prove that $HE(\Lambda)$
is independent of the choice of almost complex structure and vector fields.

\begin{lemma} For any tame pair $(Z,\Lambda)$
the trivial cobordism pair $(\R \times Z, \R \times \Lambda)$ is unobstructed.  
\end{lemma}

\begin{proof} Indeed, $L = \R \times \Lambda$ bounds no holomorphic
 disks without punctures, by the angle change formula \cite[Lemma \ref{I-anglechange}]{BCSW1}.  It follows that the curvature $m(b)$ of the trivial chain $b = 0 $ vanishes.   \end{proof} 
 
\begin{lemma}  \label{lem:identity} 
Let $\Lambda_- = \Lambda_+$, and suppose that $\R \times Z$ is equipped with a translation-invariant almost complex structure and a Morse datum 
so that the vector field $\zeta_L$ is translation-invariant.  
Then the chain map $\varphi: CE(\Lambda_-) \to CE(\Lambda_+)$ is the identity map. 
\end{lemma} 

\begin{proof} We check that  in this case the only rigid maps contributing to the cobordism map are trivial cylinder.  By assumption $\Lambda_ - = \Lambda_+$ and each moduli space has the action of translations since the  isotopy from $J_-$ to $J_+$ and $\zeta_{\black,-}$ to $\zeta_{\black,+}$ is invariant under  the $\R$-action.  It follows that the only rigid configurations are those that  are invariant under the $\R$-action, so are trivial cylinders with the same incoming and outgoing Reeb chord. The induced map $\varphi: CE(\Lambda_\pm)  \to CE(\Lambda_\pm)$ is therefore  the identity map.  
 \end{proof} 
 
\begin{theorem} \label{thm:indep} Suppose that $(Z,\Lambda)$ is a tame pair.   Given two choices of cylindrical almost complex structures, Morse data and perturbations there exists a homotopy equivalence between the corresponding chain complexes $CE(\Lambda)$ and $CE(\Lambda)'$.   Furthermore, this homotopy equivalence may be taken to preserve classical sectors, that is, map $CE_\black(\Lambda)$ to $CE_\black(\Lambda)'$ and intertwine with the projections of $CE_\black(\Lambda) \to \mhat{G}(\Lambda)$ and $CE_\black(\Lambda)' \to \mhat{G}(\Lambda)$ given by taking the coefficient of the empty word. As a result, the contact homology $HE(\Lambda)$ is independent, up to isomorphism, of all such choices. 
\end{theorem} 

\begin{proof}  For convenience we recall the choices involved.   Let $J_\pm, \zeta_{\black, \pm}, \zeta_{\white,\pm}$ be almost complex structures resp. vector fields and  $\ul{P}_\pm =(P_{\Gamma,\pm})$ two choices of perturbations making  all moduli spaces of expected dimension at most one regular. 
Let $CE(\Lambda)_\pm$ denote the corresponding Chekanov-Eliashberg algebras.

We wish to extend the given data over the symplectization, with the given limits
on the positive and negative ends, so that the corresponding cobordism maps 
induce isomorphisms of contact homology.   First, extend the almost complex structures and vector fields 
\[ J_\pm \in \J(\R \times Z_\pm), \quad  \zeta_{\black,\pm}  \in \Vect(\R \times \Lambda_\pm) \] 
over $X = \R \times Z$ to an almost complex structure and  vector field 
\[ J \in \J(X), \quad \zeta_\black \in \Vect(L)  \] 
that agree with the given ones on the cylindrical ends.  Since $\zeta_{\black,\pm}$ is positive, we may assume that its extension $\zeta_\black$ has no zeroes on $\R \times Z$. 
 
Given two choices of almost complex structures, Morse data,  and perturbation data
we obtain from the trivial cobordism equipped with almost complex structures, Morse data and perturbations extending the given ones on the ends a chain map 
\[ {\varphi} : CE(\Lambda)_+ \to CE(\Lambda)_- .\]
By Theorem \ref{thm:chainh}, any homotopy between two different choices $J', \zeta'_\black$
and $J'', \zeta''_\black$ (again with no zeroes, and defining chain maps $\varphi',\varphi''$) defines a chain homotopy
\[ h :  CE(\Lambda)_+ \to CE(\Lambda)_- \] 
satisfying 
\[ {\varphi}'' - {\varphi}' = \delta_+ h + h \delta_- .\]
Similarly, we have a map 
\[ \varphi'': CE(\Lambda)_- \to CE(\Lambda)_+  \]
inducing a map in homology independent of all choices up to chain homotopy.
 By the composition Theorem \ref{composition}, we have
\[ [ \varphi''] \circ [\varphi' ] = \on{Id}_{HE(\Lambda_-)},
\quad [ \varphi'] \circ [\varphi''] = \on{Id}_{HE(\Lambda_+)} \] 
so that $[\varphi']$ and $[\varphi'']$ are inverses.

By the angle-change formula in \cite[Lemma \ref{I-anglechange}]{BCSW1} one sees that if the disk
has no incoming puncture then it has no outgoing punctures either,  so the classical sector $CE_\black(\Lambda)$ maps to $CE_\black(\Lambda)'$.  Furthermore, if the input is from the classical generators $\cI_\black(\Lambda)$ then all disks are constant and so the map $CE_\black(\Lambda) \to CE_\black(\Lambda)'$ preserves the length filtration, and in particular, maps words of positive 
length to words of positive length.  It follows that the chain map intertwines with projection to the span of the zero length word, as claimed.  
\end{proof}

\begin{corollary}
Suppose that $\Lambda_0, \Lambda_1$ are isotopic Legendrians in $(Z,\alpha)$. 
Then $HE(\Lambda_0)$ is isomorphic to $HE(\Lambda_1)$.
\end{corollary}

\begin{proof}  By Lemma I-\ref{I-isotopy}, there exist cobordisms  $L_{01}, L_{10}$ between $\Lambda_0$ and $\Lambda_1$.  Consider the associated chain maps $\varphi(L_{01},0)$ and $\varphi(L_{10}, 0)$
map $HE(\Lambda_0)$ to $HE(\Lambda_1)$ and vice versa. By the composition theorem, 
composition $\varphi(L_{01},0) \circ \varphi(L_{10},0)$ is the map associated
to the concatenation $L_{01} \circ L_{10}$ of the two cobordisms, in the trivial 
cobordism $\R \times Z$ equipped with the symplectic form from Lemma I-\ref{I-isotopy}.
One sees easily  that this symplectic cobordism, as well as the Lagrangian cobordisms
$L_{01} \circ L_{10}$, is isotopic to the trivial cobordism $(\R \times Z, \R \times \Lambda_0)$.  By Lemma \ref{lem:isotope}, there exists a symplectomorphism mapping 
$L_{01} \circ L_{10}$ to the trivial cobordism.  By Lemma \ref{lem:identity}, 
$\varphi(L_{01}) \circ \varphi(L_{10})$ is the identity.  Combining with the statement for the composition in reverse order, it follows that $HE(\Lambda_0)$ is isomorphic to $HE(\Lambda_1)$.
\end{proof}

\begin{remark}  Naturally one expects the version of Legendrian contact
homology here to be related to the one described in  Ekholm-Etnyre-Sullivan \cite{ees:leg}, in cases where both are defined.  The natural expectation would be that the complexes on $CE_\white(\Lambda)$
(that is, projecting out the classical generators) would be chain homotopic. 
However, this would require a perturbation scheme allowing generic contact forms, which is certainly outside the scope of this paper.
\end{remark}

\bibliography{leg}{}
\bibliographystyle{plain}
\end{document}